\documentclass[12pt]{article}
\pdfoutput=1

\usepackage{amsmath, amsthm, amssymb}
\usepackage{fullpage}
\usepackage{enumerate}
\usepackage{ifpdf}
\usepackage{mathrsfs}
\ifpdf
\usepackage[pdftex]{graphicx}
\else
\usepackage[dvips]{graphicx}
\fi
\usepackage{tikz}
 	 \usetikzlibrary{arrows,backgrounds}
\usepackage[all]{xy}

\input xy
\xyoption{all}

\usepackage[pdftex,plainpages=false,hypertexnames=false,pdfpagelabels]{hyperref}
\newcommand{\arXiv}[1]{\href{http://arxiv.org/abs/#1}{\tt arXiv:\nolinkurl{#1}}}
\newcommand{\arxiv}[1]{\href{http://arxiv.org/abs/#1}{\tt arXiv:\nolinkurl{#1}}}

\newcommand{\doi}[1]{\href{http://dx.doi.org/#1}{{\tt DOI:#1}}}

\newcommand{\googlebooks}[1]{(preview at \href{http://books.google.com/books?id=#1}{google books})}

\usepackage{xcolor}
\definecolor{dark-red}{rgb}{0.7,0.25,0.25}
\definecolor{dark-blue}{rgb}{0.15,0.15,0.55}
\definecolor{medium-blue}{rgb}{0,0,.8}
\definecolor{DarkGreen}{RGB}{0,150,0}
\hypersetup{
   colorlinks, linkcolor={purple},
   citecolor={medium-blue}, urlcolor={medium-blue}
}

\theoremstyle{plain}
\newtheorem{thm}{Theorem}[section]
\newtheorem*{thm*}{Theorem}
\newtheorem{thmalpha}{Theorem}

\newtheorem{cor}[thm]{Corollary}
\newtheorem*{cor*}{Corollary}
\newtheorem{lem}[thm]{Lemma}
\newtheorem*{claim}{Claim}
\newtheorem{prop}[thm]{Proposition}

\newtheorem*{quest*}{Question}

\theoremstyle{definition}
\newtheorem{defn}[thm]{Definition}

\newtheorem{nota}[thm]{Notation}

\newtheorem{exs}[thm]{Examples}
\newtheorem{ex}[thm]{Example}
\newtheorem{rem}[thm]{Remark}
\newtheorem{rems}[thm]{Remarks}
\newtheorem{fact}[thm]{Fact}
\newtheorem{facts}[thm]{Facts}

\DeclareMathOperator{\Ad}{Ad}
\DeclareMathOperator{\Aut}{Aut}

\DeclareMathOperator{\coker}{coker}
\DeclareMathOperator{\core}{core}

\DeclareMathOperator{\depth}{depth}
\DeclareMathOperator{\Hom}{Hom}
\DeclareMathOperator{\Gr}{Gr}

\DeclareMathOperator{\id}{id}
\DeclareMathOperator{\KMS}{KMS}

\DeclareMathOperator{\op}{op}

\DeclareMathOperator{\spann}{span}

\DeclareMathOperator{\Tr}{Tr}
\DeclareMathOperator{\tr}{tr}

\newcommand{\D}{\displaystyle}
\newcommand{\comment}[1]{}

\newcommand{\be}{\begin{enumerate}[(1)]}
\newcommand{\ee}{\end{enumerate}}

\newcommand{\Z}{\mathbb{Z}}

\newcommand{\F}{\mathbb{F}}
\newcommand{\R}{\mathbb{R}}
\newcommand{\T}{\mathbb{T}}
\newcommand{\C}{\mathbb{C}}

\newcommand{\I}{\infty}
\newcommand{\set}[2]{\left\{#1 \middle| #2\right\}}
\newcommand{\ketbra}[2]{|#1\rangle \langle #2|}

\newcommand{\e}{\epsilon}

\newcommand{\Pro}{{\sf Pro}}

\newcommand{\Rep}{{\sf Rep}}
\newcommand{\sEnd}{{\sf End}}
\newcommand{\noshow}[1]{}
\newcommand{\MR}[1]{}

\newcommand{\Asterisk}{\mathop{\scalebox{1.5}{\raisebox{-0.2ex}{$*$}}}}%

\newcommand{\PA}{\cP\hspace{-.1cm}\cA}
\newcommand{\TL}{\cT\hspace{-.08cm}\cL}
\newcommand{\NC}{\cN\hspace{-.03cm}\cC}
\newcommand{\TP}[1]{\cT_{#1}(\cP_\bullet)}
\newcommand{\OP}[1]{\cO_{#1}(\cP_\bullet)}
\newcommand{\SP}[1]{\cS_{#1}(\cP_\bullet)}
\newcommand{\FP}[1]{\cF_{#1}(\cP_\bullet)}
\newcommand{\KP}[1]{\cK_{#1}(\cP_\bullet)}
\newcommand{\coreTP}[1]{\core(\TP{#1})}
\newcommand{\coreOP}[1]{\core(\OP{#1})}

\def\semicolon{;}
\def\applytolist#1{
    \expandafter\def\csname multi#1\endcsname##1{
        \def\multiack{##1}\ifx\multiack\semicolon
            \def\next{\relax}
        \else
            \csname #1\endcsname{##1}
            \def\next{\csname multi#1\endcsname}
        \fi
        \next}
    \csname multi#1\endcsname}

\def\calc#1{\expandafter\def\csname c#1\endcsname{{\mathcal #1}}}
\applytolist{calc}QWERTYUIOPLKJHGFDSAZXCVBNM;
\def\bbc#1{\expandafter\def\csname bb#1\endcsname{{\mathbb #1}}}
\applytolist{bbc}QWERTYUIOPLKJHGFDSAZXCVBNM;
\def\bfc#1{\expandafter\def\csname bf#1\endcsname{{\mathbf #1}}}
\applytolist{bfc}QWERTYUIOPLKJHGFDSAZXCVBNM;
\def\sfc#1{\expandafter\def\csname s#1\endcsname{{\sf #1}}}
\applytolist{sfc}QWERTYUIOPLKJHGFDSAZXCVBNM;
\def\ffc#1{\expandafter\def\csname f#1\endcsname{{\mathfrak #1}}}
\applytolist{ffc}QWERTYUIOPLKJHGFDSAZXCVBNM;

\usepackage{yfonts}
\newcommand{\err}{\mathfrak{r}}

\tikzstyle{shaded}=[fill=red!10!blue!20!gray!30!white]
\tikzstyle{unshaded}=[fill=white]
\tikzstyle{empty box}=[circle, draw, thick, fill=white, opaque, inner sep=2mm]
\tikzstyle{annular}=[scale=.7, inner sep=1mm, baseline]
\tikzstyle{rectangular}=[scale=.75, inner sep=1mm, baseline=-.1cm]
\usetikzlibrary{calc}

\newcommand{\nbox}[6]{
	\draw[thick, #1] ($#2+(-#3,-#3)+(-#4,0)$) rectangle ($#2+(#3,#3)+(#5,0)$);
	\coordinate (ZZa) at ($#2+(-#4,0)$);
	\coordinate (ZZb) at ($#2+(#5,0)$);
	\node at ($1/2*(ZZa)+1/2*(ZZb)$) {#6};
}

\newcommand{\ncircle}[5]{
	\draw[thick, #1] #2 circle (#3);
	\node at #2 {#5};
	\node at ($#2+(#4:.15cm)+(#4:#3cm)$) {$\star$};
}

\begin{document}

\title{$C^*$-algebras from planar algebras I:
\\ 
{\large canonical $C^*$-algebras associated to a planar algebra}
}
\author{Michael Hartglass and David Penneys}
\date{\today}
\maketitle
\begin{abstract}

From a planar algebra, we give a functorial construction to produce numerous associated $C^*$-algebras.
Our main construction is a Hilbert $C^*$-bimodule with a canonical real subspace which produces Pimsner-Toeplitz, Cuntz-Pimsner, and generalized free semicircular $C^*$-algebras.
By compressing this system, we obtain various canonical $C^*$-algebras, including Doplicher-Roberts algebras, Guionnet-Jones-Shlyakhtenko algebras, universal (Toeplitz-)Cuntz-Krieger algebras, and the newly introduced free graph algebras.
This is the first article in a series studying canonical $C^*$-algebras associated to a planar algebra.
\end{abstract}


\section{Introduction}

Since Jones' landmark article \cite{MR696688}, the modern theory of subfactors has developed deep connections to numerous branches of mathematics, including, but not limited to, representation theory, category theory, knot theory, topological quantum field theory, statistical mechanics, conformal field theory, and free probability.
The purpose of this series of articles is to develop connections to $C^*$-algebras with a view toward connections to non-commutative geometry.

A {\rm II}$_1$-subfactor $N\subset M$ has finite index if $\sb{N}M$ is a finitely generated projective $N$-module.
In this case, the index $[M: N]$ is the trace of the corresponding idempotent in $K_0(N)^+$.
Jones remarkably proved that the index has discrete and continuous ranges
$$
[M:N]\in \set{4\cos^2(\pi/n)}{n=3,4,5,\dots}\cup [4,\I],
$$
and he constructed an example with each admissible index \cite{MR696688}.

Decomposing the alternating tensor powers of $\sb{N}L^2(M)_M$ and $\sb{M}L^2(M)_N$ into irreducible bimodules yields a unitary 2-category called the \underline{standard invariant}, generalizing the representation categories of quantum groups. 
The objects are $N$ and $M$, the 1-morphisms are bimodules generated by $\sb{N}L^2(M)_M$ and its dual $\sb{M}L^2(M)_N$, and the 2-morphisms are bimodule maps.
The principal graphs are the induction/restriction multi-graphs corresponding to tensoring with $\sb{N}L^2(M)_M$ and $\sb{M}L^2(M)_N$.

The standard invariant has been axiomatized in three similar ways, each emphasizing slightly different structure: Ocneanu's paragroups \cite{MR996454,MR1642584}, Popa's $\lambda$-lattices \cite{MR1334479}, and Jones' planar algebras \cite{math/9909027}.
In \cite{MR1198815,MR1334479,MR1887878}, Popa starts with a $\lambda$-lattice $A_{\bullet,\bullet}=(A_{i,j})$ and constructs a {\rm II}$_1$-subfactor whose standard invariant is $A_{\bullet,\bullet}$.
If the standard invariant has finite depth, or more generally, is strongly amenable, the resulting subfactor is hyperfinite.
Hence we have a Tannaka-Krein like duality between hyperfinite (strongly) amenable subfactors and (strongly) amenable standard invariants.
For finite depth subfactors, Popa's canonical commuting squares \cite{MR1055708} are roughly equivalent to Ocneanu's flat connections and string algebras \cite{MR996454}.

The techniques used for the above duality, particularly Ocneanu's string algebras in finite depth, are highly similar to those used by Doplicher-Roberts \cite{MR1010160} in their duality for subgroups of compact groups independent of the classical Tannaka-Krein theory \cite{MR1173027}. 
Given a closed subgroup $G\subset SU(n)$, the representation category $\Rep(G)$ forms a symmetric rigid $C^*$-tensor category.
Conversely, given a suitably nice object $\rho$ in a symmetric rigid $C^*$-tensor category $\sC$, they construct a $C^*$-algebra $\cO_\rho$ together with a canonical endomorphism $\widehat{\rho}$ which produces a closed subgroup $G\subset SU(n)$ that encodes the category $\sC$. 

The algebra $\cO_\rho$ is a compression of the Cuntz-Krieger graph algebra $\cO_{\vec{\Gamma}_\rho}$ of (a directed version of) the fusion graph $\Gamma_\rho$ with respect to $\rho$ \cite{MR1183780} (see Remark \ref{rem:DRIsCornerOfCK}).
In fact, Izumi used a slight variation of the Doplicher-Roberts algebras to find a link with subfactor theory in \cite{MR1604162}, where he obtained inclusions of (compressions of) Cuntz-Krieger algebras with finite Watatani indices (see Remark \ref{rem:Izumi}).

In \cite{MR2732052,MR2807103}, starting with a subfactor planar algebra $\cP_\bullet$, Guionnet, Jones, and Shlyakhtenko (GJS) gave a diagrammatic proof of Popa's celebrated reconstruction theorem, developing connections to free probability and random matrices.
They showed that the resulting factors are interpolated free group factors when $\cP_\bullet$ is finite depth.
When $\cP_\bullet$ is infinite depth, Hartglass showed the factors are $L(\F_\infty)$ \cite{MR3110503}.

In this article, we provide a framework to fit the above constructions together.
Our construction is motivated by the work of Voiculescu and Pimsner and ideas of Jones \cite{JonesSeveral}.

First, in \cite{MR799593}, Voiculescu produced his free Gaussian functor.
Starting with an $n$-dimensional real Hilbert space $H_\bbR$, and real vectors $\xi\in H_\bbR$, Voiculescu forms creation and annihilation operators $L(\xi),L(\xi)^*$ acting on the full Fock space $\cF(H_\C)$ of the complexified Hilbert space. 
The $L(\xi),L(\xi)^*$ generate the Toeplitz algebra $\cT_n$ which contains the compacts $\cK$, and $\cT_n/\cK$ is isomorphic to the Cuntz algebra $\cO_n$ \cite{MR0467330}. The sums $L(\xi)+L(\xi)^*$ generate Voiculescu's free semicircular algebra $\cS_n$.

Starting with a $C^*$-algebra $A$ and a Hilbert $A-A$ bimodule $\cY$, one can mimic the same construction, as in work of Pimsner \cite{MR1426840}.
We get the Pimsner-Toeplitz algebra $\cT(\cY)$ of creation and annihilaiton operators on the Pimsner-Fock space $\cF(\cY)$, which contains the compacts $\cK(\cF(\cY))$.  
If $A$ acts on $\cY$ by compact operators, the Cuntz-Pimsner algebra is the quotient $C^*$-algebra $\cO(\cY)=\cT(\cY)/\cK(\cF(\cY))$.
If we have a distinguished real subspace $\cY_\bbR\subset \cY$ such that $\cY_\bbR \cdot A= \cY$, we also get a free semicircular algebra $\cS(\cY_\bbR)$ and $KK$-equivalences $A\hookrightarrow \cS(\cY_\bbR)\hookrightarrow \cT(\cY)$ \cite{Germain}.

Starting with a (sub)factor planar algebra $\cP_\bullet$ (see Subsection \ref{sec:PAs}), we define a ground $C^*$-algebra $\cB=\cB(\cP_\bullet)$ and a $\cB-\cB$ Hilbert bimodule $\cX=\cX(\cP_\bullet)$ with a distinguished real subspace $\cX_\bbR\subset \cX$ satisfying $\cX_\bbR \cdot \cB= \cX$.
From our bimodule, we obtain the Pimsner-Fock space $\cF(\cP_\bullet)$, a Pimsner-Toeplitz algebra $\TP{}$ of creation and annihilation operators, a Cuntz-Pimsner quotient $\OP{}$, and a free semicircular algebra $\SP{}$, which is isomorphic to the $C^*$-analog of the semifinite GJS von Neumann algebra \cite{MR2807103}.
Moreover, we have the following.

\begin{thmalpha}\label{thm:MainA}
The assignment $\cP_\bullet$ to $\cX(\cP_\bullet)$, $\FP{}$, $\TP{}$, $\OP{}$, and $\SP{}$ is functorial.
\end{thmalpha}

Compressions of $\cX(\cP_\bullet)$ yield other canonical Hilbert bimodules associated to $\cP_\bullet$ and its principal graph $\Gamma$.
First, we recover the Cuntz-Krieger bimodule of \cite{MR1426840,MR1722197}.
The corresponding compressions of the Cuntz-Pimsner and Pimsner-Toeplitz algebras are exactly the Cuntz-Kriger algebra \cite{MR561974} of $\vec{\Gamma}$, a certain directed graph associated to $\Gamma$, and its universal Toeplitz extension \cite{MR1722197}.

We define a new canonical $C^*$-algebra associated to a undirected graph, which is the $C^*$-analog of the free graph von Neumann algebras appearing in \cite{MR2807103, MR3016810, MR3110503}.

\begin{thmalpha}\label{thm:MainB}
Given an arbitrary undirected graph $\Lambda$ and a (non-degenerate) weighting $\mu$ on its vertices, we define the free graph algebra $\cS(\Lambda,\mu)$ as a canonical subalgebra of the Toeplitz-Cuntz-Krieger algebra $\cT_{\vec{\Lambda}}$, provided we remember the edge generators $S_\epsilon$ of $\cT_{\vec{\Lambda}}$.
The inclusion $C_0(V(\Gamma))\hookrightarrow \cS(\Lambda,\mu)$ is a $K$-equivalence.
In certain situations, $\cS(\Lambda,\mu)$ passes injectively to the quotient Cuntz-Krieger algebra $\cO_{\vec{\Lambda}}$.
(See Subsection \ref{sec:FreeGraphAlgebraOfUnoriented} for more details.)
\end{thmalpha}

When $\Lambda=\Gamma$, we let $\mu$ be the quantum dimension weighting induced from the trace on $\cP_{\bullet}$, which satisfies the Frobenius-Perron condition for the modulus $\delta$.
In this case, $\cS(\Gamma)=\cS(\Gamma,\mu)$ arises as the compression of the semifinite GJS $C^*$-algebra corresponding to the compression of $\cX(\cP_\bullet)$ realizing the Cuntz-Krieger bimodule.
Moreover, $\cS(\Gamma)$ always passes injectively to the quotient Cuntz-Krieger algebra $\cO_{\vec{\Gamma}}$.

A further compression yields an honest Hilbert space $\FP{0}$ associated to a planar algebra originally due to Jones-Shlyakhtenko-Walker \cite{MR2645882}. 
The corresponding compression of $\OP{}$ is the Doplicher-Roberts algebra $\cO_\rho$, where $\rho$ corresponds the strand in $\Pro(\cP_\bullet)$, the rigid $C^*$-tensor category of projections of $\cP_\bullet$ \cite{MR2559686,1207.1923,1208.5505}. 
Hence we recover the main result of \cite{MR1183780} on compressing $\cO_{\vec{\Gamma}_\rho}$ to obtain $\cO_\rho$.
The compression of $\TP{}$ is a Toeplitz extension of $\cO_\rho$ by the compacts $\cK$ on $\FP{0}$.

In this case, the compression of the semifinite GJS $C^*$-algebra (and also the free graph algebra $\cS(\Gamma)$) is the zeroth GJS $C^*$-algebra  \cite{MR2732052,MR2645882,GJSCStar}.
When $\cP_\bullet=\NC_\bullet$,  the planar algebra of non-commuting polynomials in $n$ self-adjoint variables (see Definition \ref{defn:NC}), this second compression exactly recovers Voiculescu's free Gaussian functor. 

\begin{figure}[!ht]
\begin{center}
\begin{tabular}{|c||c|c|c|}
\hline
& $\cA=\cO$ & $\cA=\cT$ & $\cA=\cS$
\\\hline\hline
$\cA(\cP_\bullet)$
&
Cuntz-Pimsner
&
Pimsner-Toeplitz
&
semifinite GJS algebra
\\\hline
$\cA(\Gamma)$
&
Cuntz-Krieger $\cO_{\vec{\Gamma}}$
&
Toeplitz-Cuntz-Krieger $\cT_{\vec{\Gamma}}$
&
free graph algebra $\cS(\Gamma)$
\\\hline
$\cA_0(\cP_\bullet)$
&
Doplicher-Roberts $\cO_\rho$
&
Toeplitz extension by $\cK$
&
GJS algebra
\\\hline
$\cA_0(\NC_\bullet)$
&
Cuntz $\cO_n$
&
Toeplitz $\cT_n$
&
Free semicircular system
\\\hline
\end{tabular}
\end{center}
\caption{$C^*$-algebras arising from our $C^*$-Hilbert bimodule $\cX(\cP_\bullet)$.}\label{fig:AllTheAlgebras}
\end{figure}

We immediately see some basic properties of the above algebras.
First, for $\cA\in \{\cO,\cT,\cS\}$, we have $\cA(\cP_{\bullet})\cong \cA(\Gamma)\otimes \cK$, where $\cK$ is the compact operators on a separable, infinite dimensional Hilbert space.
This fact, together with Theorem \ref{thm:MainB} above, yields the following corollary.

\begin{cor*}
The semifinite GJS algebra $\SP{}$ and the free graph algebra $\cS(\Gamma)$ are subnuclear, and thus exact. 
\end{cor*}

In Part II of this series \cite{GJSCStar}, we use our functorial construction to analyze the structure of the GJS $C^*$-algebra $\SP{0}$.
In particular, we show it is simple, has unique tracial state, has stable rank 1, and has weakly unperforated $K_0$-group.
We will also show that $\SP{0}$ is strongly Morita equivalent to $\cS(\Gamma)$ and to $\SP{}$, and thus we determine its $K$-theory from Theorem \ref{thm:MainB}.

\subsection{Outline}

In Section \ref{sec:Background}, we give the background for our construction, including material on $C^*$-Hilbert bimodules, Cuntz-Krieger graph algebras, Doplicher-Roberts algebras, and planar algebras.
In Section \ref{sec:PAFockSpace}, we give the first construction: the GJSW-Doplicher-Roberts system based on work of Jones \cite{JonesSeveral}.
There are slight problems with functoriality of this construction, and it does not allow for the efficient computation of $K_*(\SP{0})$. 

Our operator-valued system in Section \ref{sec:Operator} alleviates these concerns, and we compress it in Section \ref{sec:Compression} to obtain numerous canonical $C^*$-algebras, including (Toeplitz-)Cuntz-Krieger algebras, as well as the GJSW-Doplicher-Roberts system of Section \ref{sec:PAFockSpace}.
Of particular importance is Subsection \ref{sec:FreeGraphAlgebraOfUnoriented} where we introduce the free graph algebra $\cS(\Lambda,\mu)$.

\subsection{Acknowledgements}

The authors would like to thank Ken Dykema, George Elliott, Vaughan Jones, Dima Shlyakhtenko, and Dan Voiculescu for many helpful conversations.
We would like to thank Vaughan Jones again for generously allowing us to develop his ideas from \cite{JonesSeveral}.
David Penneys was partially supported by the Natural Sciences and Engineering Research Council of Canada.
Both authors were supported by DOD-DARPA grant HR0011-12-1-0009.


\section{Background}\label{sec:Background}

In this section, we begin by studying canonical $C^*$-algebras associated to a Hilbert bimodule $\cY$.
We follow \cite{MR1426840}, defining the Pimsner-Toeplitz and Cuntz-Pimsner algebras.
If we have a real subspace $\cY_\bbR\subset \cY$, we also get a free semicircular algebra as in \cite{Germain}.
This construction allows for the computation of $K$-theory via \cite{Germain}.

We then define the Cuntz-Krieger algebras \cite{MR561974} via the Cuntz-Krieger bimodule \cite{MR1722197}.
Next, we look at the Doplicher-Roberts algebras $\cO_\rho$ associated to an object $\rho$ in a rigid $C^*$-tensor category $\sC$ \cite{MR1010160}.

Finally, we discuss planar algebras as the rest of this article will connect the above ideas to shaded and unshaded planar algebras.

\subsection{$C^*$-algebras associated to a Hilbert bimodule}\label{sec:Pimsner}

We now recall the construction of the Pimsner-Toeplitz and Cuntz-Pimsner algebras of a Hilbert $A-A$ bimodule $\cY$ from \cite{MR1426840}. We will assume $A$ acts by compact operators to simplify the definitions.
First, we recall the notion of an $A-B$ Hilbert-bimodule for $C^{*}$-algebras $A$ and $B$.

\begin{defn}
An $A-B$ Hilbert bimodule $\cY$ is a Banach space with an isometric left action of $A$ and a right action of $B$, together with a $B$-valued sesquilinear form $\langle \cdot | \cdot \rangle_{B}$ on $Y$ which is conjugate linear in the first variable.
It satisfies
\be
\item
$\|\xi\|_{\cY} = \|\langle \xi | \xi \rangle_{B}\|_B^{1/2}$ for all $\xi \in \cY$.  In particular, $\langle \cdot | \cdot \rangle_{B}$ is positive definite.
\item
$\langle \xi|\eta \rangle_{B} = \langle \eta| \xi \rangle_{B}^{*}$ for all $\xi, \eta \in \cY$
\item
$\langle x\xi|\eta \rangle_{B} = \langle \xi| x^{*}\eta \rangle$ for all $\xi, \eta \in \cY$ and $x \in A$
\item
$\langle \xi|\eta \rangle_{B}\cdot y = \langle \xi| \eta y \rangle_{B}$ for all $\xi, \eta \in \cY$ and $y \in B$
\ee

We denote by $\cL(\cY)$ the operators $T$ on $\cY$ which are bounded and adjointable, i.e., there exists $T^{*}$ such that $\langle T\xi|\eta \rangle_{B} = \langle \xi|T^{*}\eta \rangle_{B}$ for all $\xi, \eta \in \cY$.

For $\eta,\xi\in \cY$, we define the rank one operator $\ketbra{\eta}{\xi}$ by $\ketbra{\eta}{\xi} \zeta = \eta \langle \xi|\zeta\rangle_B$.
Note that each rank one operator is adjointable, and $\ketbra{\eta}{\xi}^*=\ketbra{\xi}{\eta}$.
The $\cY$-compact operators, denoted $\cK(\cY)$, is the $C^*$-algebra generated by the rank one operators.
Note that $\cK(\cY)$ is an ideal of $\cL(\cY)$.

We say $A$ acts by compact operators on $\cY$ if the left action of $A$ on $\cY$ is by operators in $\cK(\cY)$.
\end{defn}

\begin{defn}
Let $A,B,C$ be $C^{*}$-algebras, with $\cY$ an $A-B$ Hilbert bimodule and $\cZ$ a $B-C$ Hilbert bimodule.
Let $\cY \odot \cZ$ be the algebraic tensor product over $\bbC$.  Define a $C$-valued inner product on $\cY \odot \cZ$ by
$$
\langle \eta_{1} \odot \eta_{2} | \xi_{1} \odot \xi_{2} \rangle_{C} = \langle  \eta_2 | \langle \eta_1  | \xi_{1}\rangle_B \xi_{2} \rangle_C,
$$
and define $\cN = \set{ \xi \in \cY \odot \cZ}{ \langle \xi| \xi \rangle_{C} = 0}$.
Then $\cY \otimes_{B} \cZ$ is defined to be the completion of $(\cY \odot \cZ)/\cN$ with respect to the above $C$-valued inner product.
\end{defn}

\begin{defn}\label{def:Pimsner-Toeplitz}
Let $\cY$ be a Hilbert $A-A$ bimodule, where $A$ acts by compact operators on $\cY$.
The \underline{Pimsner-Fock space of $\cY$} is given by
$$
\cF(\cY) =  \bigoplus_{n=0}^{\infty} \bigotimes_{A}^{n} \cY,
$$
where $\bigotimes_{A}^{0} \cY = A$ and the completion of the direct sum is with respect to the $A$-valued inner product which is the $A$-linear extension of
$$
\langle \eta_1\otimes \cdots \otimes \eta_m | \xi_1 \otimes \cdots \otimes \xi_n\rangle_A = \delta_{m,n}
\langle \eta_n | \langle \eta_{n-1}| \cdots \langle \eta_2|\langle \eta_1| \xi_1\rangle_A \xi_2\rangle_A \cdots  \xi_{n-1}\rangle_A \xi_n\rangle_A.
$$
The \underline{Pimsner-Toeplitz algebra $\cT(\cY)$} is the $C^*$-subalgebra of $\cL(\cF(\cY))$  generated by $A$ and creation operators $L_{+}(x)$ for $x \in \cY$. The action of the creation operators is given by:
$$
L_{+}(x) (y_{1} \otimes \cdots \otimes y_{n}) = x \otimes y_{1} \otimes \cdots \otimes y_{n}.
$$
The operators $L_{+}(x)$ are bounded and adjointable with adjoint
$$
(L_{+}(x))^{*} (y_{1} \otimes \cdots \otimes y_{n}) = \langle x| y_{1}\rangle_{A}\cdot y_{2} \otimes \cdots \otimes y_{n}.
$$
Notice that $L_{+}(y)^{*}L_{+}(x) = \langle y|x \rangle_{A}$ so that $\|L_{+}(x)\|_{\cL(\cF(\cY))} = \|x\|_{A}$.

The \underline{Cuntz-Pimsner algebra $\cO(\cY)$} is given by $\cT(\cY)/\cK(\cF(\cY))$.
(Clearly $\cK(\cF(\cY))\subset \cT(\cY)$).
\end{defn}

\begin{defn}\label{defn:GaugeAction}
The algebra $\cT(\cY)$ has a canonical $C^*$-dynamical system.
There is a gauge action $\gamma : \bbT \to \Aut(\cT(\cY))$ satisfying $\gamma_z(L_+(\xi))=z L_+(\xi)$ for all $\xi\in \cY$,
and there is a dynamics $\sigma : \bbR \to \Aut(\cT(\cY))$ which is lifted from $\gamma$ via the map $t\mapsto e^{it}$.

Note that we have the (unbounded) number operator given by $N\xi = n\xi$ for $\xi\in \otimes_A^n \cY$.
Hence the map $\exp(it N)$ given by $\xi\mapsto e^{itn}\xi$ extends uniquely to a unitary on $\cF(\cY)$.
For $x\in \cT(\cY)$, we have $\sigma_t(x)=\exp(it N) x \exp(-it N)$ for $t\in\bbR$.

Recall that a state $\varphi$ is a $\KMS_\beta$ state for $(\cT(\cY),\sigma)$ for $\beta>0$ if for all $x,y\in \cT(\cY)$ with $y$ entire we have $\varphi(x \sigma_{i\beta}(y)) = \varphi(yx)$ \cite{MR548006,MR835762}.
The element $y\in \cT(\cY)$ is entire if $t\mapsto \sigma_t(y)$ extends to an entire function.
\end{defn}

\begin{defn}
Define the conditional expectation $E: \cT(\cY)\to \cT(\cY)^{\bbT}$ by
$$
E(x)=\int_\bbT \gamma_z(x)\,dz
$$
where we use the normalized Lebesgue measure on $\bbT$.
The image of $E$ is called the \underline{core of $\cT(\cY)$}.

Since every $A-A$ bilinear isomorphism of $\cY$ induces an automorphism of $\cO(\cY)$, we get a gauge action on $\cO(\cY)$ as well.
Again, we have a conditional expectation $E : \cO(\cY)\to \cO(\cY)^{\bbT}$, and the image is called the \underline{core of $\cO(\cY)$}.
\end{defn}

\begin{rem}\label{rem:KMSKernel}
It is of great interest to study KMS states on the Pimsner-Toeplitz and Cuntz-Pimsner algebras (e.g., see \cite{MR2056837}).
Recall that if $\varphi$ is a KMS state on $A$, then $\fN_\varphi=\set{x\in A}{\varphi(x^*x)=0}$ is a closed 2-sided ideal in $A$.
Since $\varphi$ is a state, $\fN_\varphi$ is a closed left ideal, as $\varphi(x^*x)=0$ implies $\varphi(xy^*yx)\leq \|y^*y\|\varphi(x^*x)=0$.
Using the KMS-condition and the Cauchy-Schwarz inequality,
for $y^*$ entire in $\cT(\cY)$,
$$
\varphi(y^*x^*xy)
=\varphi(x^*xy\sigma_{i\beta}(y^*))
\leq \|x^*x\|_2\|\|y\sigma_{i\beta}(y^*)\|_2
\leq \|x\|_\I \varphi(x^*x)^{1/2}\|y\sigma_{i\beta}(y^*)\|_2
=0.
$$
Hence by continuity, we see that $\fN_\varphi$ is a right ideal.
In many cases, $\fN_\varphi=\cK(\cY)$, and we will see such examples in Subsections \ref{sec:PACuntz} and 
\ref{sec:OperatorCuntz} (see also Subsection \ref{sec:KMS}).
\end{rem}

\begin{defn}
Let $\cY_{\R}$ be a closed real subspace of $\cY$ such that $\cY_{\R}\cdot A = \cY$.
The \underline{semicircular algebra $\cS(\cY_\bbR)$} of $\cY_\bbR$ is the $C^*$-subalgebra of $\cL(\cF(\cY))$ generated by $A$ together with the operators $\set{L_{+}(\eta) + L_{+}(\eta)^{*}}{ \eta \in \cY_{\R}}$.
\end{defn}

We have the following theorem from \cite{Germain}.

\begin{thm}\label{thm:Germain}
Suppose $A$ is a unital $C^{*}$ algebra, and $\cY$ is an $A-A$ Hilbert bimodule, with inner product $\langle \cdot | \cdot \rangle_{A}$.
Let $\cY_{\R}$ be a closed real subspace of $\cY$ such that $\cY_{\R}\cdot A = \cY$.
The canonical inclusions $i: A \rightarrow \cS(\cY_{\R})$ and $j: \cS(\cY_{\R}) \rightarrow \cT(\cY)$ are $KK$-equivalences.
\end{thm}

\begin{rem}
The proof that $j\circ i : A \to \cT(\cY)$ is a KK-equivalence is originally in \cite{MR1426840}.
\end{rem}

\begin{ex}[Voiculescu's free Gaussian functor \cite{MR799593}]\label{ex:FreeGaussian}
Let $H_\bbR$ be a real Hilbert space, and let $H$ be its complexification.
Consider $H$ as a $\C-\C$ Hilbert bimodule in the obvious way.
In this case, we get Voiculescu's free Gaussian functor.

The full Fock space is
$$
\cF(H) =\bigoplus_{n\geq 0} H^{\otimes n},
$$
where $H^{\otimes 0}$ denotes a one dimensional Hilbert space spanned by the vector $\Omega$.
The left creation and annihilation operators $L_+(\xi),L_+(\xi)^*$ for $\xi \in H$ are the usual left creation and annihilation operators on full Fock space.

We may also define right creation and annihilation operators, which commute with the left creation and annihilation operators up to the compacts. 
For all $\eta,\xi\in H$, we have the relations $[L_+(\eta),R_+(\xi)]=0$ and $[R_+(\eta)^*,L_+(\xi)]=\langle \xi,\eta\rangle p_\Omega$ and their adjoints.

When $\dim(H_\bbR)=n<\infty$, the Pimsner-Toeplitz and Cuntz-Pimsner algebras are the Toeplitz algebra $\cT_n$ and the Cuntz algebra $\cO_n$ respectively.
On the full Fock space $\cF(H)$,
\be
\item
$L_+(\xi)^*L_+(\xi)=\|\xi\|^2$ for all $\xi\in H_\bbR$, and
\item
if $\{\xi_1,\dots, \xi_n\}$ is an orthonormal basis for $H_\bbR$, then $\sum_{i=1}^n L_+(\xi_i)L_+(\xi_i)^*=1-p_\Omega$.
\ee
The representation of $\cT_n$ on $\cF(H)$ is irreducible, so $\cT_n$ contains the compact operators $\cK$ on $\cF(H)$. Moreover, $\cT_n/\cK\cong \cO_n$.
If we set $S_i=L_+(\xi_i)$ for $i=1, \dots, n$ for an orthonormal basis $\{\xi_1,\dots, \xi_n\}$ of $H_\bbR$, the $S_i$'s satisfying the Cuntz algebra relations
$$
\sum_{i=1}^n S_iS_i^* = 1 = S_j^*S_j \text{ for all }j=1,\dots, n.
$$
It is well known that $\cO_n$ is simple and purely infinite, and $K_0(\cO_n)=\bbZ/(n-1)\bbZ$ and $K_1(\cO_n)=(0)$ \cite{MR604046}.

The point of starting with a real Hilbert space is that we get Voiculescu's free semicircular $C^*$-algebra $\cS_n$ as the $C^*$-algebra generated by the elements $L_+(\xi)+L_+(\xi)^*$ for $\xi\in H_\bbR$.
The von Neumann completion of $\cS_n$ is isomorphic to $L(\F_n)$.

In this case, the $\bbR$-action satisfies $\sigma_t(S_j)=e^{it} S_j$ for all $j=1,\dots, n$, and the core of $\cO_n$ is the AF algebra $\bigotimes^\infty M_n(\bbC)$, which has a unique tracial state.
By \cite{MR500150}, there is a unique $\KMS$ state on $\cO_n$ for $2\leq n<\I$, where the only admissible $\beta$-value is $\ln(n)$.
The $\KMS_{\ln(n)}$-state $\varphi$ is given by
$$
\varphi(S_{i_k}\cdots S_{i_1} S_{j_1}^*\cdots S_{j_\ell}^*) =e^{-\ln(n)k}\delta_{k,\ell}\delta_{i_1,j_1}\cdots \delta_{i_k,j_k}.
$$
As $\fN_\varphi$ must be an ideal of $\cO_n$, we must have $\fN_\varphi=(0)$, i.e., $\varphi$ is faithful.

Moreover, there is a unique $\KMS_{\ln(n)}$-state on $\cT_n$ which factors through $\varphi$, so taking the GNS representation of $\cT_n$ implements the canonical surjection $\cT_n \to \cO_n$.
(However, note that for each $\beta>\ln (n)$, there is a $\KMS_\beta$-state on $\cT_n$ by \cite{MR3061018} which does not factor through $\cO_n$.)

The $K$-theory of $\cO_n$ was computed in \cite{MR604046}, and the $K$-theory of $\cS_n$ was computed in \cite{MR1220422}.
Note that by \cite{Germain}, we have KK-equivalences $i:\bbC \to \cS_n$ and $j:\cS_n\to \cT_n$.
\end{ex}

\subsection{Cuntz-Krieger graph algebras and Toeplitz extensions}

We define the Toeplitz-Cuntz-Krieger and Cuntz-Krieger \cite{MR561974} algebras associated to a directed graph $\Lambda$ as the Pimsner-Toeplitz and Cuntz-Pimsner algebras associated to the Cuntz-Krieger bimodule \cite{MR1722197} respectively.
This spatial approach is parallel to Voiculescu's free Gaussian functor \cite{MR799593} and to our operator valued system in Section \ref{sec:Operator}.

\begin{nota}
A directed graph $\Lambda=(V(\Lambda),E(\Lambda),s,t)$ consists of a countable vertex set $V(\Lambda)$, a countable edge set $E(\Lambda)$, and source and target maps $s,t:E(\Lambda)\to V(\Lambda)$.

We use the characters $\Lambda,\Gamma$ for graphs, and we will specify whether they are directed.
Usually the character $E$ is used in the graph algebra literature, but we reserve the character $E$ for our conditional expectations.
\end{nota}

\begin{defn}\label{defn:CuntzKriegerBimodule}
Given a directed graph $\Lambda=(V(\Lambda), E(\Lambda), s, t)$, let $\cC_\Lambda=C_0(V(\Lambda))$.
We define the \underline{Cuntz-Krieger bimodule $Y(\Lambda)$} (called $X(\Lambda)$ in \cite[Example 1.2]{MR1722197}, see also \cite[p. 193]{MR1426840}) as 
$$
Y(\Lambda) = \set{f: E(\Lambda) \to \bbC}{V(\Lambda)\ni \alpha \mapsto \sum_{t(\e)=\alpha} |f(\e)|^2 \text{ is in }\cC_\Lambda}
$$
together with the following $\cC_\Lambda-\cC_\Lambda$ Hilbert bimodule structure.

For $\epsilon \in E(\Lambda)$, let $\chi_{\epsilon}$ be the indicator function at $\e\in E(\Lambda)$, i.e., $\chi_{\epsilon}(\epsilon') = \delta_{\epsilon=\epsilon'}$.
Then $Y(\Lambda)$ is the completion of the space of finitely supported functions on $E(\Lambda)$ under the $\cC_\Lambda$-valued inner product
$$
\langle \chi_{\e'}| \chi_{\e} \rangle_{\cC_\Lambda} = \delta_{\e=\e'} p_{t(\epsilon)}
$$
where $p_{\alpha}$ is the indicator function at $\alpha \in V(\Lambda)$.
Note $Y(\Lambda)$ is naturally a $\cC_\Lambda-\cC_\Lambda$ Hilbert bimodule under the actions $p_{\alpha}\chi_{\epsilon} = \delta_{\alpha = s(\epsilon)}\chi_{\epsilon}$ and $\chi_{\epsilon}p_{\alpha} = \delta_{t(\epsilon) = \alpha}\chi_{\epsilon}$.

The \underline{Toeplitz-Cuntz-Krieger algebra $\cT_\Lambda$} is the Pimsner-Toeplitz algebra $\cT(Y(\Lambda))$.
The \underline{Cuntz-Krieger algebra $\cO_\Lambda$} is the Cuntz-Pimsner algebra $\cO(Y(\Lambda))$.
\end{defn}

\begin{rem}
If $\Lambda=(V(\Lambda),E(\Lambda),s,t)$ is a locally finite directed graph with no sinks, then these algebras can also be defined as universal $C^*$-algebras via generators and relations as follows.
\be
\item
The Toeplitz-Cuntz-Krieger algebra $\cT_\Lambda$ is the universal $C^*$-algebra generated by partial isometries $S_\epsilon$ for $\epsilon\in E(\Lambda)$ and projections $p_\alpha$ for $\alpha\in V(\Lambda)$ satisfying the relations
$$
S_{\e'}^{*}S_{\e} = \delta_{\epsilon,\epsilon'}p_{t(\epsilon)}
\text{ and }
\sum_{s(\e) = \alpha} S_{\e}S_{\e}^{*} \leq p_{\alpha}.
$$
\item
If we replace the `$\leq$' with `$=$' in the second relation above, then this universal $C^*$-algebra is the Cuntz-Krieger algebra $\cO_\Lambda$.

\ee
\end{rem}

Given a countable directed graph $\Lambda$, there are two matrices associated to it.

\begin{defn}
The \underline{edge matrix of $\Lambda$} is the $\{0,1\}$-matrix $A$ indexed by the edges of $\Lambda$ such that $A(\epsilon,\epsilon')=\delta_{t(\epsilon)=s(\epsilon')}$.

The \underline{vertex matrix of $\Lambda$} (also known as the \underline{adjacency matrix of $\Lambda$}) is the $\Z_{\geq 0}$-matrix $B$ indexed by the vertices of $\Lambda$ such that $B(\alpha,\beta)$ is the number of edges from $\alpha$ to $\beta$.
\end{defn}

\begin{facts}\label{facts:CuntzKrieger}
Suppose $\Lambda$ is locally finite, strongly connected, and the edge matrix is not a permutation matrix.
\be
\item
The $K$-theory of $\cO_\Lambda$ is given by $K_0(\cO_\Lambda)=\coker(1-A^T)\cong \coker(1-B^T)$ and $K_1(\cO_\Lambda)=\ker(1-A^T)\cong \ker(1-B^T)$ where $1-A^T$ acts on $\Z^{E(\Lambda)}$ and  $1-B^T$ acts on $\Z^{V(\Lambda)}$ \cite{MR608527,MR1409796}, \cite[Theorem 3.2]{MR2020023}.

\item
 $\cO_\Lambda$ is separable, nuclear, simple, purely infinite,
\cite{MR561974,MR1432596,MR1626528},
and in the bootstrap class
\cite[Proposition 2.6]{MR1738948} (the crossed product $\cO_\Lambda\rtimes_\alpha \bbT$ is AF), so it satisfies the UCT \cite{MR894590}.
Hence by \cite{MR1745197,MR1780426}, the stable isomorphism class of $\cO_\Lambda$ is determined by the $K$-theory of $\cO_\Lambda$, and the isomorphism class is determined by the $K$-theory and the class of the unit.
\ee
\end{facts}

\begin{rems}\label{rem:PAExamples}
\mbox{}
\be
\item
The examples in this article are inspired by (sub)factor planar algebras $\cP_\bullet$.
Hence each directed graph $\vec{\Lambda}$ in this article is obtained from an undirected graph $\Lambda$ by a simple procedure (see Definition \ref{defn:Directed}).
In the (sub)factor planar algebra case, this undirected graph is the principal graph $\Gamma$ of $\cP_\bullet$.

When our directed graphs come from principal graphs, they are pointed (there is a base vertex $\star$), locally finite, and strongly connected (any vertex may be reached from $\star$ along directed edges).
Moreover, since we assume $\delta>1$, $\Gamma$ is not the graph with one vertex and one loop nor the Coxeter-Dynkin diagram $A_2$, so the edge and vertex matrices will always be irreducible and not permutation matrices.
Hence we are in the fortunate situation of Facts \ref{facts:CuntzKrieger}.

\item
The Cuntz-Krieger algebras can also be realized as a quotient of the Toepltiz extensions acting on a Fock space of finite paths due to \cite{MR620461,MR650194}. 
In general, this will not yield the universal Toeplitz-Cuntz-Krieger algebra as the defect $p_\Omega$ must always be the same for each vertex, since their Fock space has only one vacuum vector $\Omega$.

\item
It might be interesting to extend our results beyond the finite dimensional case (allowing $\dim(\cP_n)=\infty$), where truly infinite matrices could appear.
In particular, one might want to work with a planar algebra for which $\cP_0$ is some infinite dimensional commutative $C^*$-algebra.
The work of Exel-Laca \cite{MR1703078} might be highly useful in this situation.

\item
Many authors have studied KMS states on Cuntz-Krieger algebras.
If $\Lambda$ is finite and the edge matrix is irreducible,  by \cite{MR759450}, $\cO_\Lambda$ has a unique KMS state, and the only admissible $\beta$-value is $\ln(\lambda)$, where $\lambda$ is the Frobenius-Perron eigenvalue of $\Lambda$.
See \cite{MR3061018} for the KMS states on $\cT_\Lambda$ in this case.
For infinite locally finite graphs, see also \cite{1306.5118}.

\ee
\end{rems}

\subsection{Doplicher-Roberts algebras}

In \cite{MR1010160}, Doplicher-Roberts obtain a duality theory for compact subgroups $G$ extending the Tannaka-Krein theory \cite{MR1173027}.
In \cite[Section 4]{MR1010160}, they associate a $C^*$-algebra $\cO_\rho$ to an object $\rho$ in a rigid $C^*$-tensor category $\sC$.
The category may be recovered from the algebra with some additional structure in special situations.

These algebras are similar to the Cuntz-Krieger algebras and Ocneanu's path/string algebras \cite{MR996454}.
Indeed, in \cite{MR1183780}, it was shown that $\cO_\rho$ is a corner of the Cuntz-Krieger algebra associated to an auxiliary bipartite graph associated to the fusion graph of $\rho$.
One can also view paths on this auxiliary bipartite graph as paths on a certain directed graph (see Definition \ref{defn:Directed} below).
Using the string algebra approach in \cite{MR1604162}, Izumi found inclusions of generalized Doplicher-Roberts algebras with finite Watatani indices.

Because we want undirected fusion graphs to define our free graph algebras in Subsection \ref{sec:FreeGraphAlgebraOfUnoriented}, we work with symmetrically self-dual objects. 
Note that although the $C^*$-algebra construction of \cite{MR1010160} does not require symmetrically self-dual objects, the main theorems therein require `special objects' in symmetric rigid $C^*$-tensor categories.
Theorem 3.4 of \cite{MR1010160}, which ensured the existence of enough `special objects', constructed objects of the form $X\oplus \overline{X}$, which are always symmetrically self-dual.

\begin{nota}
Let $\sC$ be a strict rigid $C^*$-tensor category, where we denote the objects of $\sC$ by $\rho,\sigma,\tau$.
We denote the tensor product of $\sigma$ and $\tau$ by $\sigma\tau$.
We denote the finite dimensional Hilbert space of maps from $\sigma$ to $\tau$ by $\Hom(\sigma,\tau)$.

Fix a simple symmetrically self-dual object $\rho\in \sC$.
Let $\Gamma_\rho$ be the fusion graph with respect to $\rho$, i.e.,
the vertices are the isomorphism classes of irreducible objects in $\sC$, and given simples $\sigma, \tau$, there exactly $\dim(\Hom(\sigma \rho, \tau))$ edges from $\sigma$ to $\tau$.
Since our distinguished object $\rho$ is symmetrically self-dual, our fusion graph is undirected.
\end{nota}

\begin{defn}
The \underline{Doplicher-Roberts algebra $\cO_\rho$} is the $C^*$-algebra defined as follows.
First, there is a canonical inclusion $\Hom(\rho^m,\rho^n)\hookrightarrow \Hom(\rho^{m+1},\rho^{n+1})$ given by $x\mapsto x\otimes 1_\rho$.
For $k\in\bbZ$, let $\sp{o}\cO_\rho^k = \varinjlim \Hom(\rho^n, \rho^{k+n})$, and let $\sp{o}\cO_\rho = \bigoplus_{k\in \bbZ} {\sp{o}\cO}_\rho^k$.
The multiplication on $\sp{o}\cO_\rho$ is induced from the composition in $\sC$.
If $y\in \Hom(\rho^m, \rho^{n+m})$ and $x\in \Hom(\rho^{n+m},\rho^{k+n+m})$, then $xy \in \Hom(\rho^m, \rho^{k+n+m})$, and
$xy\otimes 1_\rho = (x\otimes 1_\rho)(y\otimes 1_\rho)$.
We let $\cO_\rho$ be the universal enveloping algebra of $\sp{o}\cO_\rho$, which exists by \cite{MR1010160,MR1183780}.
\end{defn}

\begin{rem}
The Doplicher-Roberts algbera $\cO_\rho$ comes equipped with a canonical $\bbT$-action $\gamma$ and a canonical endomorphism $\widehat{\rho}$ implemented by $\rho$. 
They define $\widehat{\rho}$ on $\bigoplus_{k\in \bbZ} {\sp{o}\cO}_\rho^k$ as the map induced by 
$$
\Hom(\rho^m, \rho^{n+m})\ni x\mapsto 1_\rho\otimes x \in \Hom(\rho^{m+1},\rho^{n+m+1}),
$$
and $\widehat{\rho}$ extends to an endomorphism of $\cO_\rho$.
We get a faithful representation of the full subcategory $\sT_\rho\subset \sC$ whose objects are $\rho^n$ for $n\in \bbZ_{\geq 0}$ in $\sEnd(\cO_\rho)$, the $C^*$-tensor category of endomorphisms of $\cO_\rho$.
Moreover, the assignment $(\sT_\rho,\rho)$ to $(\cO_\rho,\gamma,\widehat{\rho})$ is functorial.

We give a diagrammatic construction of the Doplicher-Roberts algebra $\cO_\rho$ via Subsection \ref{sec:PACuntz} using the factor planar algebra $\cP_\bullet$ of the rigid $C^*$-tensor subcategory $\langle \rho\rangle \subset \sC$ and the choice of object $\rho$.
See also Remark \ref{rem:PAandTC} below.
\end{rem}

\begin{defn}\label{defn:Directed}
Suppose $\Lambda$ is an undirected graph.
The directed graph $\vec{\Lambda}=(V(\vec{\Lambda}),E(\vec{\Lambda}),s,t)$ is the graph with the following properties:
\be
\item
$V(\Lambda) = V(\vec{\Lambda})$.
\item
For each $\epsilon\in E(\Lambda)$ which is a loop, i.e., $\epsilon$ only connects to one vertex $\alpha$, we have a single edge $\epsilon\in E(\vec{\Lambda})$ such that $s(\epsilon)=t(\epsilon)=\alpha$.
\item
For each $\epsilon\in E(\Lambda)$ which is not a loop, i.e., $\epsilon$ connects to two distinct vertices $\alpha,\beta$, we have two edges $\epsilon',\epsilon''\in E(\vec{\Lambda})$ such that $s(\epsilon')=t(\epsilon'')=\alpha$ and $t(\epsilon')=s(\epsilon'')=\beta$.
\ee
There is an involution $\op$ on $E(\vec{\Lambda})$ given by $\epsilon^{\op}=\epsilon$ if $\epsilon$ is a loop, and $(\epsilon')^{\op}=\epsilon''$ if $\epsilon$ is not a loop.
\end{defn}

\begin{rem}\label{rem:OrientLambda}
As long as $\Lambda$ is not the graph with one vertex and one loop or the $A_{2}$ Coxeter-Dynkin diagram, $\vec{\Lambda}$ is strongly connected, and the vertex and edge matrices of $\vec{\Lambda}$ are not permutation matrices.
Moreover, $\vec{\Lambda}$ is locally finite if $\Lambda$ is.
\end{rem}

\begin{rem}\label{rem:DRIsCornerOfCK}
By \cite{MR1183780}, $\cO_\rho$ is isomorphic to a full corner of the Cuntz-Krieger algebra $\cO_{\vec{\Gamma}_\rho}$, where $\vec{\Gamma}_\rho$ is obtained from $\Gamma_\rho$ as in Definition \ref{defn:Directed}.
They showed how to realize $\cO_{\vec{\Gamma}_\rho}$ as a path/string algebra acting on a Hilbert space of infinite paths.
In this infinite path representation, $\cO_\rho$ is isomorphic to the cutdown $P_\star\cO_{\vec{\Gamma}} P_\star$ where $P_\star$ is the projection onto the paths that start at the distinguished vertex $\star$, which corresponds to the trivial object $1\in\cC$.
In particular, Facts \ref{facts:CuntzKrieger} all hold for $\cO_\rho$.
(See also \cite{MR1432596} on Cuntz-Krieger algebras of pointed graphs.)

We note that the method of \cite{MR1183780} using a bipartite graph is equivalent to looking at infinite paths on our directed graph $\vec{\Gamma}_\rho$ obtained from $\Gamma_\rho$ as in Definition \ref{defn:Directed}.
This approach is analogous to the finite path representation due to \cite{MR650194,MR759450}.
\end{rem}

\begin{rem}\label{rem:Izumi}
In \cite[Section 2.5]{MR1604162}, Izumi studies the Doplicher-Roberts algebras $\cO_{\rho\overline{\rho}}$ and $\cO_{\overline{\rho}\rho}$, where $\rho$ is the standard sector of a finite index, finite depth type III subfactor.
In the type II$_1$ language, $\rho$ corresponds to the standard bimodule $\sb{N}L^2(M)_M$, which is a 1-morphism in a 2-category.
Thus $\rho\overline{\rho}$ and $\overline{\rho}\rho$ are generating bimodules of the principal even half and the dual even half of the standard invariant respectively, and the Doplicher-Roberts construction can be performed in these tensor categories.
In this case, there are two canonical endomorphisms $\widehat{\rho}$ and $\widehat{\overline{\rho}}$ which switch back and forth between the two algebras $\cO_{\rho\overline{\rho}}$ and $\cO_{\overline{\rho}\rho}$.

Izumi notes in \cite[Proposition 3.1]{MR1604162} that $\cO_{\rho\overline{\rho}}$ is stably isomorphic to $\cO_{\Lambda\Lambda^T}$ where $\Lambda$ is the (bipartite!) adjacency matrix of $\Gamma$.
This follows from Remark \ref{rem:DRIsCornerOfCK} above, since $\Lambda\Lambda^T$ is the adjacency matrix of the fusion graph with respect to $\rho\overline{\rho}$.
\end{rem}

\subsection{Planar algebras}\label{sec:PAs}

Planar algebras have proven to be a useful tool for analyzing subfactors and tensor categories.

\begin{defn}
A planar algebra is a sequence of vector spaces $\cP_\bullet=(\cP_n)_{n\geq 0}$ together with an action of the planar operad, i.e., every planar tangle with $k_1,\dots, k_r$ points on the input disks and $k_0$ points on the output disk corresponds to a multilinear map
$\cP_{k_1}\times \cdots \times \cP_{k_r}\to \cP_0$. For example,
$$
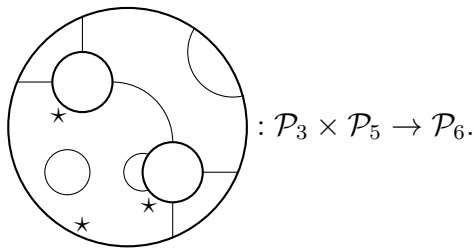
\begin{tikzpicture}[baseline = .5cm]
	\clip (-.6,.6) circle (1.6cm);
	\draw (.8,1.6) circle (.6cm);
	\draw (0,.4) arc (0:90:.8cm);
	\draw (-.4,0) circle (.25cm);
	\draw (-1.2,2.2)--(-1.2,1.2)--(-2.2,1.2);
	\draw (0,-1)--(0,0)--(1,0);
	\draw (-1.4,0) circle (.3cm);
	\draw[ultra thick] (-.6,.6) circle (1.6);
	\ncircle{unshaded}{(-1.2,1.2)}{.4}{235}{}
	\ncircle{unshaded}{(0,0)}{.4}{235}{}
	\node at (-1.2,-.7) {$\star$};
\end{tikzpicture}
:\cP_{3}\times \cP_{5} \to \cP_{6}.
$$
We require
\begin{itemize}
\item \underline{isotopy invariance:} isotopic tangles produce the same multilinear maps,
\item \underline{identity:} the identity tangle (which only has radial strings) acts as the identity transformation, and
\item \underline{naturality:} the gluing of tangles corresponds to composition of multilinear maps.
When we glue tangles, we match up the points along the boundary disks making sure the distinguished intervals marked by $\star$ align.
\end{itemize}
\end{defn}

\begin{nota}
We will always draw our planar tangles in rectangular form in what follows.
$$
\begin{tikzpicture}[baseline = .5cm]
	\nbox{unshaded}{(-.6,.6)}{1.4}{0}{0}{}
	\draw (.8,1.4) arc (-90:-180:.6cm);
	\draw (0,.4) arc (0:90:.8cm);
	\draw (-.4,-.2) arc (270:90:.2cm);
	\draw (-1.2,2)--(-1.2,1.2)--(-2,1.2);
	\draw (0,-.8)--(0,0)--(.8,0);
	\draw (-1.4,0) circle (.3cm);
	\nbox{unshaded}{(-1.2,1.2)}{.4}{0}{0}{}
	\nbox{unshaded}{(0,0)}{.4}{0}{0}{}
	\node at (-.5,-.5) {$\star$};
	\node at (-1.6,.6) {$\star$};
	\node at (-2.2,-.6) {$\star$};
\end{tikzpicture}
$$
When we draw a planar tangle, we will often suppress the external rectangle, which is assumed to be large.
If we omit a $\star$, it is always assumed to be in the \underline{lower left corner}.
Finally, we draw one string labelled $k$ rather than $k$ parallel strings.
\end{nota}

\begin{defn}
 $P_\bullet$ is called a \underline{factor planar algebra} if $P_\bullet$ is
\begin{itemize}
\item
\underline{evaluable:} $\dim(P_n)<\infty$ for all $n\geq 0$ and $P_0\cong \C$ via the map that sends the empty diagram to $1\in\C$. 
By naturality, there is a scalar $\delta$ such that any labelled diagram containing a closed loop is equal to $\delta$ times the same diagram without the closed loop. 

\item
\underline{involutive:} for each $n\geq 0$, there is a map $*\colon P_n\to P_n$ with $*\circ *=\id$ which is compatible with the reflection of tangles, i.e., if $T$ is a planar tangle with $r$ input disks, then $T(x_1,\dots, x_r)^*=T^*(x_1^*,\dots,x_r^*)$ where $T^*$ is the reflection of $T$.

\item
\underline{spherical:} the value of a closed diagram is invariant under spherical isotopy. Equivalently, for all $x\in \cP_2$,  
$$
\tr(x)
=
\begin{tikzpicture}[baseline = -.1cm]
	\draw (0,.3) arc (180:0:.3cm)--(.6,-.3) arc (0:-180:.3cm);
	\nbox{unshaded}{(0,0)}{.3}{0}{0}{$x$}
\end{tikzpicture}
=
\begin{tikzpicture}[baseline = -.1cm, xscale=-1]
	\draw (0,.3) arc (180:0:.3cm)--(.6,-.3) arc (0:-180:.3cm);
	\nbox{unshaded}{(0,0)}{.3}{0}{0}{$x$}
\end{tikzpicture}\,.
$$

\item
\underline{unitary/positive:} for every $n\geq 0$, the map $\langle\cdot,\cdot\rangle \colon P_n\times P_n\to P_0\cong\C$ given by
$$
\langle x,y\rangle =
\begin{tikzpicture}[baseline=-.1cm]
	\draw (0,0)--(1,0);
	\node at (.5,.2) {{\scriptsize{$n$}}};
	\nbox{unshaded}{(0,0)}{.3}{0}{0}{$x$}
	\nbox{unshaded}{(1,0)}{.3}{0}{0}{$y^*$}
\end{tikzpicture}
$$
is a positive definite inner product. 
Hence by \cite{MR696688}, $\delta\in \set{2\cos(\pi/n)}{n\geq 3}\cup[2,\I)$.

\end{itemize}
A \underline{subfactor planar algebra} is a 2-shaded factor planar algebra. 
We refer the reader to \cite{MR2979509,MR2972458} for more details on subfactor planar algebras.
\end{defn}

\begin{exs}
Some examples of factor planar algebras are the Temperley-Lieb planar algebra $\TL_\bullet(\delta)$ and the planar algebra of non-commutative polynomials $\NC_\bullet$ (see Definition \ref{defn:NC}).
\end{exs}

\begin{defn}
The \underline{principal graph $\Gamma$} of $\cP_\bullet$ is the graph whose vertices are the isomorphism classes of simple projections in $\cP_\bullet$, and if $p\in \cP_{2n}$ and $q\in \cP_{2(n+1)}$ are simple projections, then the vertices $[p]$ and $[q]$ are connected by $\dim(q \cP_{2(n+1)}(p\otimes \id))$ edges.

Given a vertex $[p]$ of $\Gamma$, the number $\tr(p)$ is independent of the choice of representative of $[p]$. 
The vector $(\tr(p))_{[p]\in V(\Gamma)}$ defines the \underline{quantum dimension weight vector} on the vertices of $\Gamma$ satisfying the Frobenius-Perron condition:
$$
\delta \tr(p) = \sum_{[q]\in V(\Gamma)} N_{[p],[q]} \tr(q)
$$
where $N_{[p],[q]}$ is the number of edges connecting $[p]$ and $[q]$ in $\Gamma$.
Note that the quantum dimension weight vector is normalized so that the trace of the empty diagram is 1.
\end{defn}

\begin{rem}\label{rem:PAandTC}
In \cite{MR2559686,1208.5505}, it was shown how to associate to a factor planar algebra $\cP_\bullet$ its rigid $C^*$-tensor category of projections $\Pro(\cP_\bullet)$, and from a rigid $C^*$-tensor category $\cC$ and a choice of symmetrically self-dual object $\rho\in \cC$, how to obtain a factor planar algebra $\PA(\cC,\rho)_\bullet$.
In this case, $\PA(\Pro(\cP_\bullet),\id)_\bullet=\cP_\bullet$, and $\Pro(\PA(\cC,\rho)_\bullet)$ is equivalent to the subcategory of $\cC$ generated by $\rho$.
Moreover, the fusion graph $\Gamma_\rho$ corresponds to the principal graph $\Gamma$.
Also see \cite{1207.1923} for more details.
\end{rem}


\section{The GJSW-Doplicher-Roberts system}\label{sec:PAFockSpace}

In this subsection, following Guionnet-Jones-Shlyakhtenko-Walker \cite{MR2732052,MR2645882}, we define a canonical Hilbert space $\cF_0(\cP_\bullet)$ associated to a planar algebra $\cP_\bullet$ which shares some properties of a full Fock space.
Following Jones \cite{JonesSeveral}, we define the Toeplitz and Cuntz algebras in this scenario, where the latter is isomorphic to the Doplicher-Roberts algebra $\cO_\rho$ where $\rho$ is the strand in $\cP_\bullet$.

The work of Jones \cite{JonesSeveral} and Shelly \cite{Shelly} defined $\cO_\rho$ as the image of the Toeplitz algebra in the GNS representation of a KMS$_{\ln(\delta)}$-state.
We show this is equivalent to taking the quotient by the compact operators $\cK\subset B(\FP{0})$ in Subsection \ref{sec:KMS}.

We will see the GJS algebra $\SP{0}$ arises as the analog of Voiculescu's free semicircular algebra in Subsection \ref{sec:SP}.

\subsection{The Fock space and Toeplitz algebra}\label{sec:PAToeplitz}

\begin{defn}
Let $\FP{0}$, the \underline{full Fock space of $\cP_\bullet$}, be the Hilbert space of $\Gr_0 = \bigoplus_{k=0}^{\infty}\cP_{k}$ with the usual inner product, i.e., the completion of $\Gr_0$ under the inner product given by the extension of
$$
\langle x, y\rangle =
\begin{cases}
0&\text{if }m\neq n\\
\begin{tikzpicture}[baseline=-.1cm]
	\draw (0,0)--(1,0);
	\node at (.5,.2) {\scriptsize{$n$}};
	\nbox{unshaded}{(0,0)}{.3}{0}{0}{$x$}
	\nbox{unshaded}{(1,0)}{.3}{0}{0}{$y^*$}
\end{tikzpicture}
& \text{if }m=n
\end{cases}
$$
for $x\in\cP_n$ and $y\in \cP_m$.

For $x\in \cP_n$ and $0\leq \ell,\err$ with $\ell+\err=n$, we define the creation-annihilation operator $L_\err(x)$ by its action on a $y\in\cP_m$:  
$$
L_r(x)y =
\left(
\begin{tikzpicture}[baseline=.1cm]
	\draw (.2,0)--(.2,.8);
	\draw (-.2,0)--(-.2,.8);
	\node at (-.4,.6) {\scriptsize{$\ell$}};
	\node at (.4,.6) {\scriptsize{$\err$}};
	\nbox{unshaded}{(0,0)}{.4}{0}{0}{$x$}
	\draw[fill=red] (0,.4) circle (.05cm);
\end{tikzpicture}
\right)
\begin{tikzpicture}[baseline=.1cm]
	\draw (0,0)--(0,.8);
	\node at (.2,.6) {\scriptsize{$m$}};	
	\nbox{unshaded}{(0,0)}{.4}{0}{0}{$y$}
\end{tikzpicture}
=
\begin{cases}
0&\text{if }m< \err\\
\begin{tikzpicture}[baseline=-.1cm]
	\draw (-.2,0)--(-.2,.8);
	\draw (1.2,0)--(1.2,.8);
	\draw (.2,.4) arc (180:0:.3cm);
	\node at (-.4,.6) {\scriptsize{$\ell$}};
	\node at (.5,.8) {\scriptsize{$\err$}};
	\node at (1.7,.6) {\scriptsize{$m-\err$}};
	\nbox{unshaded}{(0,0)}{.4}{0}{0}{$x$}
	\nbox{unshaded}{(1,0)}{.4}{0}{0}{$y$}
\end{tikzpicture}
& \text{if }\err\leq m.
\end{cases}
$$
We use the dot on the top of the $n$-box $x$ to denote $x\in \cP_n$ as the operator $L_{\err}(x)$ to distinguish it from the vector $x\in\FP{0}$.
Note that the omitted $\star$ is still on the bottom left corner.
The operator $L_{\err}(x)$ is bounded by \cite{MR2645882}.

We define the \underline{$\cP_\bullet$-Toepltiz algebra $\TP{0}$} as the unital $C^*$-algebra acting on $\FP{0}$ generated by
$$
\set{L_{\err}(x)}{x\in \cP_n,\, n\geq 0,\text{ and }0\leq \err \leq n}.
$$

It is straightforward to show that the product $L_{\err_1}(x)\cdot L_{\err_2}(y)\in \TP{0}$ for $x\in \cP_n$ and $y\in \cP_m$ is given by
$$
L_{r_1}(x) \cdot L_{r_2}(y) =
\left(
\begin{tikzpicture}[baseline=.1cm]
	\draw (.2,0)--(.2,.8);
	\draw (-.2,0)--(-.2,.8);
	\node at (-.4,.6) {\scriptsize{$\ell_1$}};
	\node at (.4,.6) {\scriptsize{$\err_1$}};
	\nbox{unshaded}{(0,0)}{.4}{0}{0}{$x$}
	\draw[fill=red] (0,.4) circle (.05cm);
\end{tikzpicture}
\right)
\cdot
\left(
\begin{tikzpicture}[baseline=.1cm]
	\draw (.2,0)--(.2,.8);
	\draw (-.2,0)--(-.2,.8);
	\node at (-.4,.6) {\scriptsize{$\ell_2$}};
	\node at (.4,.6) {\scriptsize{$\err_2$}};
	\nbox{unshaded}{(0,0)}{.4}{0}{0}{$y$}
	\draw[fill=red] (0,.4) circle (.05cm);
\end{tikzpicture}
\right)
=
\begin{cases}
\begin{tikzpicture}[baseline=.2cm]
	\nbox{unshaded}{(.5,.2)}{.8}{.4}{.6}{}
	\draw (-.2,0)--(-.2,1.4);
	\draw (1,0)--(1,1.4);
	\draw (1.4,0)--(1.4,1.4);
	\draw (.2,.4) arc (180:0:.3cm);
	\node at (-.4,1.2) {\scriptsize{$\ell_1$}};
	\node at (.1,.6) {\scriptsize{$\err_1$}};
	\node at (.5,1.2) {\scriptsize{$\ell_2-\err_1$}};
	\node at (1.6,1.2) {\scriptsize{$\err_2$}};
	\nbox{unshaded}{(0,0)}{.4}{0}{0}{$x$}
	\nbox{unshaded}{(1,0)}{.4}{0}{.2}{$y$}
	\draw[fill=red] (1.2,1) circle (.05cm);
\end{tikzpicture}
&\text{if }r_1\leq \ell_2\\
\begin{tikzpicture}[baseline=.2cm,xscale=-1]
	\nbox{unshaded}{(.5,.2)}{.8}{.4}{.6}{}
	\draw (-.2,0)--(-.2,1.4);
	\draw (1,0)--(1,1.4);
	\draw (1.4,0)--(1.4,1.4);
	\draw (.2,.4) arc (180:0:.3cm);
	\node at (-.4,1.2) {\scriptsize{$\err_2$}};
	\node at (.05,.6) {\scriptsize{$\ell_2$}};
	\node at (.5,1.2) {\scriptsize{$\err_1-\ell_2$}};
	\node at (1.6,1.2) {\scriptsize{$\ell_1$}};
	\nbox{unshaded}{(0,0)}{.4}{0}{0}{$x$}
	\nbox{unshaded}{(1,0)}{.4}{0}{.2}{$y$}
	\draw[fill=red] (1.2,1) circle (.05cm);
\end{tikzpicture}
&\text{if }\ell_2< r_1,
\end{cases}
$$
and the adjoint of $L_\err(x)$ is $L_\ell(x^*)$:
$$
\left(
\begin{tikzpicture}[baseline=.1cm]
	\draw (.2,0)--(.2,.8);
	\draw (-.2,0)--(-.2,.8);
	\node at (-.4,.6) {\scriptsize{$\ell$}};
	\node at (.4,.6) {\scriptsize{$\err$}};
	\nbox{unshaded}{(0,0)}{.4}{0}{0}{$x$}
	\draw[fill=red] (0,.4) circle (.05cm);
\end{tikzpicture}
\right)^*
=
\begin{tikzpicture}[baseline=.1cm,xscale = -1]
	\draw (.2,0)--(.2,.8);
	\draw (-.2,0)--(-.2,.8);
	\node at (-.4,.6) {\scriptsize{$\ell$}};
	\node at (.4,.6) {\scriptsize{$\err$}};
	\nbox{unshaded}{(0,0)}{.4}{0}{0}{$x^*$}
	\draw[fill=red] (0,.4) circle (.05cm);
\end{tikzpicture}.
$$
Define the \underline{Toeplitz core of $\cP_\bullet$}, denoted $\coreTP{0}$, as the subalgebra of $\TP{0}$ generated by $\set{L_n(x)}{x\in \cP_{2n},\, n\geq 0}$.
\end{defn}

\begin{rem}
The core of $\TP{0}$ with the multiplication defined above is \underline{not} the usual core of the standard invariant defined by Popa in \cite{MR1278111}, i.e., the inductive limit algebra $\varinjlim \cP_{2n}$ under the right inclusion. See Remark \ref{rem:CuntzCore} for the appearance of the usual core, and see Example \ref{ex:ToeplitzAF} for the AF structure of the core of $\TP{0}$.
\end{rem}

The Toeplitz algebra $\TP{0}$ appears briefly in \cite{JonesSeveral}, but without a distinguished name.
Jones observes that $\TP{0}$ contains the compact operators $\cK$ as follows from the next two lemmas.

\begin{lem}[\cite{JonesSeveral}]\label{lem:Irreducible}
The action of $\TP{0}$ on $\FP{0}$ is irreducible.
\end{lem}
\begin{proof}
Let $x\in \cP_\err\setminus\{0\}$ and $y\in \cP_\ell$. Then the operator
$$
\frac{1}{\|x\|_2^2} L_\err(y\wedge x)=
\frac{1}{\|x\|_2^2}\,
\begin{tikzpicture}[baseline=.1cm]
	\nbox{unshaded}{(0,0)}{.5}{.4}{.4}{}
	\draw (.4,0)--(.4,.9);
	\draw (-.4,0)--(-.4,.9);
	\node at (-.6,.7) {\scriptsize{$\ell$}};
	\node at (.6,.7) {\scriptsize{$\err$}};
	\nbox{unshaded}{(-.4,0)}{.3}{0}{0}{$y$}
	\nbox{unshaded}{(.4,0)}{.3}{0}{0}{$x$}
	\draw[fill=red] (0,.5) circle (.05cm);
\end{tikzpicture}
$$
maps $x\in\FP{0}$ to $y\in \FP{0}$. The rest is straightforward.
\end{proof}

\begin{lem}[\cite{JonesSeveral}]\label{lem:Compact}
The Toeplitz algebra $\TP{0}$ contains a finite rank operator, and thus all the compact operators $\cK$ by Lemma \ref{lem:Irreducible}.
\end{lem}
\begin{proof}
Note that the projection onto $\oplus_{k\geq n} \cP_k$ is given by
$$
L_n(\cup_{n})
=
\begin{tikzpicture}[baseline=.1cm]
	\draw (-.2,.7)--(-.2,.1) arc (-180:0:.2cm)--(.2,.7);
	\node at (-.4,.5) {\scriptsize{$n$}};
	\node at (.4,.5) {\scriptsize{$n$}};
	\nbox{}{(0,0)}{.3}{.1}{.1}{}
	\draw[fill=red] (0,.3) circle (.05cm);
\end{tikzpicture}.
$$
Hence $1-L_n(\cup_{n})$ is the finite rank projection onto $\bigoplus_{k\leq n} \cP_k$.
\end{proof}

\subsection{A useful result for weights on Toeplitz algebras}\label{sec:KMS}

Before we study the $\bbT$-action, the KMS state, and the Cuntz algebra of $\cP_\bullet$, we take a brief interlude to prove a useful result for weights on Toeplitz algebras.

Suppose $T$ is a $C^*$-algebra together with:
\begin{itemize}
\item a closed 2-sided ideal $K\subset T$ of ``compact" operators,
\item a $\|\cdot\|$-continuous action $\sigma: \bbT\to \Aut(T)$,
\item a conditional expectation $E: T\to C$ given by $E(x)=\int_\bbT \gamma_z(x) \, dz$, where the measure is normalized Lebesgue measure,
\item the core $C$ is AF, with distinguished finite dimensional algebras $C_k$ for $k\geq 0$ such that $C=\overline{\bigcup_{k\geq 0} C_k}^{\|\cdot\|}$ where the inclusion $C_k\hookrightarrow C_{k+1}$ is not necessarily unital,
\item
a lower semi-continuous tracial weight $\Tr : C^+\to [0,\infty]$, which induces a lower semi-continuous weight $\varphi= \Tr \circ E$ on $T$,
\item
the set $\fN_\varphi=\set{x\in T}{\varphi(x^*x)=0}$ is a 2-sided ideal of $T$.
\end{itemize}

\begin{rems}
\mbox{}
\be
\item
$\fN_\varphi$ is automatically closed by lower semi-continuity.
\item
If $\varphi$ is a KMS state, then $\fN_\varphi$ is automatically a 2-sided ideal by Remark \ref{rem:KMSKernel}.
\ee
\end{rems}

\begin{lem}\label{lem:ExpectCompact}
Suppose $x\in T^+$. Then $x\in K$ if and only if $E(x)\in K$.
\end{lem}
\begin{proof}
Suppose $x\in K$. Then $E(x)$ is a limit of Riemann sums of elements of $K$, so $E(x)\in K$.
Now suppose $E(x)\in K$.
Since $K$ is hereditary, if $E(x)\in K$, then so is
$$
E_s(x)=\int_0^s \sigma_t(x)\, dt,
$$
and thus so is the difference quotient
\begin{align*}
x&=\lim_{h\to 0} \frac{E_{0+h}(x)-E_0(x)}{h}=\lim_{h\to 0} \frac{1}{h}\int_0^h \sigma_t(x)\,dt.
\qedhere
\end{align*}
\end{proof}

We now prove the following main theorem of this subsection.
\begin{thm}\label{thm:KMS}
Suppose that $K\subseteq \fN_\varphi$ and $C_k^+\cap \fN_\varphi \subset K$.
Then $\fN_\varphi=K$.
\end{thm}

\begin{proof}
First, we note that $E(\fN_\varphi)=C\cap \fN_\varphi$ since $\varphi=\varphi\circ E$.
We prove the following claim:
\begin{claim}
If $p\in C\cap \fN_\varphi$ is a projection, then $p\in K$.
\end{claim}
\begin{proof}[Proof of Claim]
Suppose $p\in C$ is a projection with $\varphi(p)=0$.
Since $C$ is AF, there is an $n>0$ and a projection $q\in C_n$ such that $\|p-q\|<1$.
Hence $p$ is homotopic in $C$ to $q$, and thus there is a unitary $u\in C$ such that $p=uqu^*$ \cite[Section 2.2]{MR1783408}.
Now $\varphi(q)=\varphi(u^*pu)=0$ since $\fN_\varphi$ is an ideal, and thus $q\in C_n^+\cap \fN_\varphi \subset K$.
We conclude $p\in K$.
\end{proof}

Now we know the ideal $E(\fN_\varphi)=C\cap \fN_\varphi$ is a hereditary subalgebra of the AF algebra $C$, and thus has real rank zero.
By the claim, every projection in $E(\fN_\varphi)$ is in $K$, so $E(\fN_\varphi)\subseteq K$.
Using Lemma \ref{lem:ExpectCompact}, we see every positive element in $\fN_\varphi$ is also in $K$.
Hence $\fN_\varphi\subseteq K$.
\end{proof}

\subsection{The $\bbT$-action, the KMS state, and the Cuntz algebra}\label{sec:PACuntz}

We now give an action of $\bbT$ on $\TP{0}$, and we construct a $\KMS_{\ln(\delta)}$-state $\varphi$.

\begin{defn}[\cite{JonesSeveral}]
Consider the number operator $N : \Gr_0 \to \FP{0}$ by
$$
\begin{tikzpicture}[baseline=-.1cm]
	\draw (0,0)--(0,.8);
	\node at (.2,.6) {\scriptsize{$n$}};	
	\nbox{unshaded}{(0,0)}{.4}{0}{0}{$x$}
\end{tikzpicture}
\longmapsto
n\,
\begin{tikzpicture}[baseline=-.1cm]
	\draw (0,0)--(0,.8);
	\node at (.2,.6) {\scriptsize{$n$}};	
	\nbox{unshaded}{(0,0)}{.4}{0}{0}{$x$}
\end{tikzpicture}\,.
$$
Note that $N$ is an unbounded, closable operator which preserves the grading of $\FP{0}$.
There is a unique bounded operator $\exp(it N)$ given by the extension of
$$
\begin{tikzpicture}[baseline=-.1cm]
	\draw (0,0)--(0,.8);
	\node at (.2,.6) {\scriptsize{$n$}};	
	\nbox{unshaded}{(0,0)}{.4}{0}{0}{$x$}
\end{tikzpicture}
\longmapsto
e^{i t n} \,
\begin{tikzpicture}[baseline=-.1cm]
	\draw (0,0)--(0,.8);
	\node at (.2,.6) {\scriptsize{$n$}};	
	\nbox{unshaded}{(0,0)}{.4}{0}{0}{$x$}
\end{tikzpicture}\,.
$$
Define an $\bbR$-action $\sigma : \bbR \to \Aut(\TP{0})$ by $\sigma_t = \Ad(\exp(itN))$:
$$
\sigma_t\left(
\begin{tikzpicture}[baseline=.1cm]
	\draw (.2,0)--(.2,.8);
	\draw (-.2,0)--(-.2,.8);
	\node at (-.4,.6) {\scriptsize{$\ell$}};
	\node at (.4,.6) {\scriptsize{$\err$}};
	\nbox{unshaded}{(0,0)}{.4}{0}{0}{$x$}
	\draw[fill=red] (0,.4) circle (.05cm);
\end{tikzpicture}
\right) = e^{i t (\ell-\err)}
\begin{tikzpicture}[baseline=.1cm]
	\draw (.2,0)--(.2,.8);
	\draw (-.2,0)--(-.2,.8);
	\node at (-.4,.6) {\scriptsize{$\ell$}};
	\node at (.4,.6) {\scriptsize{$\err$}};
	\nbox{unshaded}{(0,0)}{.4}{0}{0}{$x$}
	\draw[fill=red] (0,.4) circle (.05cm);
\end{tikzpicture}.
$$
Note that the $\bbR$-action induces a $\bbT$-action $\gamma$ via the map $t\mapsto e^{it}$.
\end{defn}

\begin{defn}[\cite{JonesSeveral,Shelly}]
For each $x\in \TP{0}$, the map $t\mapsto \sigma_t(x)$ is norm continuous, so we define the conditional expectation $E : \TP{0} \to \coreTP{0}$ by
$$
E(x) = \int_\bbT \gamma_z(x)\, dz,
$$
where we use the normalized Lebesgue measure on $\bbT$.
We define a normalized (non-faithful) trace $\tau$ on $\coreTP{0}$ by the extension of
$$
\tau\left(
\begin{tikzpicture}[baseline=-.1cm]
	\draw (.2,0)--(.2,.8);
	\draw (-.2,0)--(-.2,.8);
	\node at (-.4,.6) {\scriptsize{$n$}};
	\node at (.4,.6) {\scriptsize{$n$}};
	\nbox{unshaded}{(0,0)}{.4}{0}{0}{$x$}
	\draw[fill=red] (0,.4) circle (.05cm);
\end{tikzpicture}
\right)
=
\delta^{-n}\,
\begin{tikzpicture}[baseline=-.1cm]
	\draw (-.2,0)--(-.2,.4) arc (180:0:.2cm) -- (.2,0);
	\node at (0,.7) {\scriptsize{$n$}};
	\nbox{unshaded}{(0,0)}{.4}{0}{0}{$x$}
	\draw[fill=red] (0,.4) circle (.05cm);
\end{tikzpicture}\,.
$$
Finally, define the state $\varphi : \TP{0} \to \C$ by $\varphi= \tau\circ E$.
\end{defn}

\begin{lem}[\cite{JonesSeveral,Shelly}]\label{lem:KMSstate}
$\varphi$ is a $\KMS_{\ln(\delta)}$ state.
\end{lem}
\begin{proof}
For $x\in\cP_{\ell_1+r_1}$ and $y\in \cP_{\ell_2+r_2}$, where without loss of generality $\err_1\leq \ell_2$,
\begin{align*}
\varphi
\left(
\begin{tikzpicture}[baseline=-.1cm]
	\draw (.2,0)--(.2,.8);
	\draw (-.2,0)--(-.2,.8);
	\node at (-.4,.6) {\scriptsize{$\ell_1$}};
	\node at (.4,.6) {\scriptsize{$\err_1$}};
	\nbox{unshaded}{(0,0)}{.4}{0}{0}{$x$}
	\draw[fill=red] (0,.4) circle (.05cm);
\end{tikzpicture}
\cdot
\begin{tikzpicture}[baseline=-.1cm]
	\draw (.2,0)--(.2,.8);
	\draw (-.2,0)--(-.2,.8);
	\node at (-.4,.6) {\scriptsize{$\ell_2$}};
	\node at (.4,.6) {\scriptsize{$\err_2$}};
	\nbox{unshaded}{(0,0)}{.4}{0}{0}{$y$}
	\draw[fill=red] (0,.4) circle (.05cm);
\end{tikzpicture}
\right)
&=
\delta_{\ell_1-\err_1=\err_2-\ell_2}
\delta^{-\ell_1}\,
\begin{tikzpicture}[baseline=.2cm]
	\nbox{unshaded}{(.5,.2)}{.8}{.4}{.8}{}
	\draw (-.2,0)--(-.2,1) arc (180:90:.5cm) -- (1.1,1.5) arc (90:0:.5cm)-- (1.6,0);
	\draw (1,0)--(1,1) arc (180:0:.2cm) -- (1.4,0);
	\draw (.2,.4) arc (180:0:.3cm);
	\node at (-.4,1.2) {\scriptsize{$\ell_1$}};
	\node at (.1,.6) {\scriptsize{$\err_1$}};
	\node at (.5,1.2) {\scriptsize{$\ell_2-\err_1$}};
	\node at (1.8,1.2) {\scriptsize{$\ell_{1}$}};
	\nbox{unshaded}{(0,0)}{.4}{0}{0}{$x$}
	\nbox{unshaded}{(1,0)}{.4}{0}{.4}{$y$}
	\draw[fill=red] (1.2,1) circle (.05cm);
\end{tikzpicture}
\\&=
\delta_{\ell_1-\err_1=\err_2-\ell_2}
\delta^{-\ell_1}\,
\begin{tikzpicture}[baseline=.2cm,xscale=-1]
	\nbox{unshaded}{(.5,.2)}{.8}{.4}{.8}{}
	\draw (-.2,0)--(-.2,1) arc (180:90:.5cm) -- (1.1,1.5) arc (90:0:.5cm)-- (1.6,0);
	\draw (1,0)--(1,1) arc (180:0:.2cm) -- (1.4,0);
	\draw (.2,.4) arc (180:0:.3cm);
	\node at (-.4,1.2) {\scriptsize{$\err_1$}};
	\node at (.05,.6) {\scriptsize{$\ell_1$}};
	\node at (.5,1.2) {\scriptsize{$\ell_2-\err_1$}};
	\node at (1.8,1.2) {\scriptsize{$\err_1$}};
	\nbox{unshaded}{(0,0)}{.4}{0}{0}{$x$}
	\nbox{unshaded}{(1,0)}{.4}{0}{.4}{$y$}
	\draw[fill=red] (1.2,1) circle (.05cm);
\end{tikzpicture}
\\&=
\delta^{\ell_1-\err_1}
\varphi\left(
\begin{tikzpicture}[baseline=.1cm]
	\draw (.2,0)--(.2,.8);
	\draw (-.2,0)--(-.2,.8);
	\node at (-.4,.6) {\scriptsize{$\ell_2$}};
	\node at (.4,.6) {\scriptsize{$\err_2$}};
	\nbox{unshaded}{(0,0)}{.4}{0}{0}{$y$}
	\draw[fill=red] (0,.4) circle (.05cm);
\end{tikzpicture}
\cdot
\begin{tikzpicture}[baseline=.1cm]
	\draw (.2,0)--(.2,.8);
	\draw (-.2,0)--(-.2,.8);
	\node at (-.4,.6) {\scriptsize{$\ell_1$}};
	\node at (.4,.6) {\scriptsize{$\err_1$}};
	\nbox{unshaded}{(0,0)}{.4}{0}{0}{$x$}
	\draw[fill=red] (0,.4) circle (.05cm);
\end{tikzpicture}
\right).\qedhere
\end{align*}
\end{proof}

Recall from Remark \ref{rem:KMSKernel} that $\fN_\varphi$ is a closed 2-sided ideal of $\TP{0}$.
Also note that the finite rank operators from the proof of Lemma \ref{lem:Compact} are in the kernel of the KMS state $\varphi$, since we are using the normalized trace, and $\cup_n$ is the image of $1_0$ inside $\cP_{2n}$ under the right inclusion.
Thus $\cK\subset \fN_\varphi$.
To show that $\cK=\fN_\varphi$, by Theorem \ref{thm:KMS}, it suffices to show the following lemma.

\begin{lem}\label{lem:FiniteRank}
Suppose $x\in \coreTP{0}$ is a finite sum of elements of the form $L_n(x_n)$ where $x_n\in \cP_{2n}$.
If $x\geq 0$ and $\varphi(x)=0$, then $x$ has finite rank.
\end{lem}
\begin{proof}
Suppose
$$
x=\sum_{n=0}^N
\begin{tikzpicture}[baseline=-.1cm]
	\draw (.2,0)--(.2,.8);
	\draw (-.2,0)--(-.2,.8);
	\node at (-.4,.6) {\scriptsize{$n$}};
	\node at (.4,.6) {\scriptsize{$n$}};
	\nbox{unshaded}{(0,0)}{.4}{0}{0}{$x_n$}
	\draw[fill=red] (0,.4) circle (.05cm);
\end{tikzpicture}\,,
$$
$x\geq 0$, and $\varphi(x)=0$.
Consider the element $\widetilde{x}$ given by
$$
\widetilde{x}=\sum_{n=0}^N
\begin{tikzpicture}[baseline=.1cm]
	\draw (.5,0)--(.5,1.4);
	\draw (-.5,0)--(-.5,1.4);
	\draw (-.2,1.4) -- (-.2,.9) arc (-180:0:.2cm) -- (.2,1.4);
	\node at (-.7,1.2) {\scriptsize{$n$}};
	\node at (.7,1.2) {\scriptsize{$n$}};
	\node at (0,.6) {\scriptsize{$N-n$}};
	\nbox{}{(0,.2)}{.8}{.1}{.1}{}
	\nbox{unshaded}{(0,0)}{.4}{.3}{.3}{$x_n$}
	\draw[fill=red] (0,1) circle (.05cm);
\end{tikzpicture}\,.
$$
Note that for all $m\geq N$ and all $y\in \cP_M$, we have
\begin{equation}\label{eq:SameAsX}
\widetilde{x} y=
\sum_{n=0}^N
\begin{tikzpicture}[baseline=.2cm]
	\nbox{unshaded}{(.5,.2)}{.8}{.4}{.6}{}
	\draw (-.2,0)--(-.2,1.4);
	\draw (1,0)--(1,1.4);
	\draw (1.4,0)--(1.4,1.4);
	\draw (.2,.4) arc (180:0:.3cm);
	\node at (-.4,1.2) {\scriptsize{$n$}};
	\node at (.1,.6) {\scriptsize{$n$}};
	\node at (.5,1.2) {\scriptsize{$N-n$}};
	\node at (2,1.2) {\scriptsize{$m-N$}};
	\nbox{unshaded}{(0,0)}{.4}{0}{0}{$x_n$}
	\nbox{unshaded}{(1,0)}{.4}{0}{.2}{$y$}
\end{tikzpicture}
=xy.
\end{equation}
Hence if $\xi \in H$, and $p$ is the projection onto $\bigoplus_{n\geq N} \cP_n$, we have
$$
\langle \widetilde{x} \xi, \xi\rangle
=
\langle \widetilde{x} p\xi,p\xi\rangle
=
\langle x p\xi,p\xi\rangle
\geq 0,
$$
so $\widetilde{x}\geq 0$.
With a slight abuse of notation, we may identify $\widetilde{x}$ with $L_N(\widetilde{x})$ for $\widetilde{x}\in \cP_{2N}$.
Now even though $\varphi$ is not faithful on $\coreTP{0}$, $\varphi|_{\cP_{2N}}$ is faithful when we identify each $y\in\cP_{2N}$ with $L_N(y)\in\coreTP{0}$.
Thus
$$
0=\varphi(x)=\varphi(\widetilde{x})=\varphi|_{\cP_{2N}}(\widetilde{x}),
$$
and thus $\widetilde{x}=0$.
Thus $x$ is finite rank since it agrees with $\widetilde{x}$ on $\cP_m$ for $m\geq N$.
\end{proof}

\begin{cor}\label{cor:FiniteRank}
Suppose $x\in \TP{0}$ is a finite sum of elements of the form $L_{\err_j}(x_j)$ where $x_j\in \cP_{\err_j+\ell_j}$.
If $x\geq 0$ and $\varphi(x)=0$, then $x\in\cK$.
\end{cor}
\begin{proof}
Note that $E(x)\in \coreTP{0}$ is a finite sum of elements of the form $L_n(x_n)$ where $x_n\in \cP_{2n}$, $E(x)\geq 0$, and $\varphi(E(x))=\varphi(x)=0$.
By Lemma \ref{lem:FiniteRank}, $E(x)$ has finite rank, so by Lemma \ref{lem:ExpectCompact}, $x\in\cK$.
\end{proof}

By Theorem \ref{thm:KMS}, where $C_n$ is the $C^*$-algebra generated by $\cP_{2k}$ for $0\leq k\leq n$, we now have the following proposition.

\begin{prop}
$\cK=\fN_\varphi$.
Hence $\TP{0}/\cK$ is isomorphic to the $C^*$-algebra generated by $\TP{0}$ in the GNS representation of the $\KMS_{\ln(\delta)}$-state $\varphi$.
\end{prop}

\begin{defn}
Define the \underline{$\cP_\bullet$-Cuntz algebra} $\OP{0}=\TP{0}/\cK = \TP{0}/\fN_\varphi$.
Define the \underline{(Cuntz) core of $\cP_\bullet$} as the fixed points of $\OP{0}$ under the $\bbT$-action. 
Note that for each $z\in \bbT$, the action of $\gamma_z$ descends to $\OP{0}=\TP{0}/\cK$ since $\gamma_z$ preserves $\cK$ as it is unitarily implemented.
\end{defn}

\begin{rems}\label{rem:CuntzCore}
\mbox{}
\be
\item
The core of $\OP{0}$ is the usual core which appears in \cite{MR1278111}, i.e., $\coreOP{0} = \varinjlim \cP_{2k}$ under the right inclusion $\cP_{2k}\hookrightarrow \cP_{2k+2}$ given by
$$
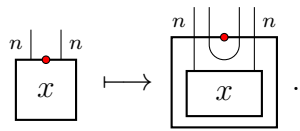
\begin{tikzpicture}[baseline=0cm]
	\draw (.2,0)--(.2,.8);
	\draw (-.2,0)--(-.2,.8);
	\node at (-.4,.6) {\scriptsize{$n$}};
	\node at (.4,.6) {\scriptsize{$n$}};
	\nbox{unshaded}{(0,0)}{.4}{0}{0}{$x$}
	\draw[fill=red] (0,.4) circle (.05cm);
\end{tikzpicture}
\longmapsto
\begin{tikzpicture}[baseline=.1cm]
	\draw (.4,0)--(.4,1.2);
	\draw (-.4,0)--(-.4,1.2);
	\draw (-.2,1.2) -- (-.2,.7) arc (-180:0:.2cm) -- (.2,1.2);
	\node at (-.6,1) {\scriptsize{$n$}};
	\node at (.6,1) {\scriptsize{$n$}};
	\nbox{}{(0,.2)}{.6}{.1}{.1}{}
	\nbox{unshaded}{(0,.05)}{.3}{.2}{.2}{$x$}
	\draw[fill=red] (0,.8) circle (.05cm);
\end{tikzpicture}\,.
$$
Hence the AF structure of the core of $\OP{0}$ is the usual one encoded by the principal graph $\Gamma$.

\item
Jones \cite{JonesSeveral} and Shelly \cite{Shelly} use the $\bbT$-action to show that the conditional expectation from $\OP{0}$ onto $\coreOP{0}$ is positive and faithful, since it is just integration with respect to Haar measure.
The KMS state then arises naturally as the composite of this faithful conditional expectation along with the faithful trace on $\coreOP{0}$ inherited from $\cP_\bullet$.

\ee
\end{rems}

\begin{facts}
\mbox{}
\be
\item
The algebra $\OP{0}$ is isomorphic to the Doplicher-Roberts algebra $\cO_\rho$ \cite{MR1010160} where $\rho$ is the strand in the rigid $C^*$-tensor category $\Pro(\cP_\bullet)$ \cite{1208.5505}.
Thus $\OP{0}$ is isomorphic to a full corner of the Cuntz-Krieger algebra $\cO_{\vec{\Gamma}}$ (see Remark \ref{rem:DRIsCornerOfCK}).

In Remark \ref{rem:OP}, we provide another proof that $\OP{0}$ is a compression of $\cO_{\vec{\Gamma}}$.
Our approach has the added benefit that it realizes $\TP{0}$ and $\SP{0}$ as reduced compressions of the universal Toeplitz-Cuntz-Krieger algebra $\cT_{\vec{\Gamma}}$ and the newly introduced free graph algebra $\cS(\Gamma)$ respectively.
(See Subsection \ref{sec:SP} for the definition of $\SP{0}$ and Subsection \ref{sec:FreeGraphAlgebraOfUnoriented} for the definition of the free graph algebra $\cS(\Gamma)$.)

\item
When the strand in $\cP_\bullet$ is $\rho\overline{\rho}$ where $\rho={\sb{N}L^2(M)}_M$ is the standard bimodule for a finite index II$_1$-subfactor $N\subset M$, $\OP{0}$ is isomorphic to $\cO_{\rho\overline{\rho}}$ from \cite[Section 2.5]{MR1604162}.

\item
Jones \cite{JonesSeveral} and Shelly \cite{Shelly} originally define the Cuntz algebra $\OP{0}$ as the image of $\TP{0}$ in the GNS representation of the KMS state $\varphi$.
In fact, by \cite[Propositions 3.1.3 and 3.3.2]{Shelly}, for finite depth $\cP_\bullet$, $\varphi$ is the unique KMS state for $\OP{0}$, and the only admissible $\beta$-value is $\ln(\delta)$.
(This is not the case for infinite depth.
For example, there is a 1-parameter family of traces on the AF core of $\cO_0(\TL_\bullet)$.)
\ee
\end{facts}

\subsection{The generalized semicircular algebra of $\cP_\bullet$}\label{sec:SP}

We now see that the algebras arising from Voiculescu's free Gaussian functor arise from the factor planar algebra $\NC_\bullet$ of non-commuting polynomials.
This motivates the definition of a generalized semicircular algebra, which is also known as the (zeroth) Guionnet-Jones-Shlyakhtenko $C^*$-algebra.

We first recall the construction of the factor planar algebra $\NC_\bullet$.

\begin{defn}\label{defn:NC}
Consider the algebra $\C\langle X_1,\dots, X_n\rangle$ of non-commuting polynomials in self-adjoint variables.
The factor planar algebra $\NC_\bullet$ is uniquely characterized by the following properties:
\begin{itemize}
\item
$\NC_n$ is the $\C$-span of the monomials of degree $n$, and $(X_{i_1}\cdots X_{i_n})^*=X_{i_n}\cdots X_{i_1}$.
\item
$\NC_\bullet$ is generated by the 1-boxes
$
\begin{tikzpicture}[baseline=0cm]
	\draw (0,0)--(0,.6);
	\nbox{unshaded}{(0,0)}{.3}{0}{0}{$X_i$}
\end{tikzpicture}
$\,, which satisfy
$
\begin{tikzpicture}[baseline=-.1cm]
	\draw (0,0)--(1,0);
	\nbox{unshaded}{(0,0)}{.3}{0}{0}{$X_i$}
	\nbox{unshaded}{(1,0)}{.3}{0}{0}{$X_j$}
\end{tikzpicture}
=
\delta_{i,j}.
$
\item
The strand is equal to $\sum_{i=1}^n X_i^2$. In diagrams,  
$
\D
\begin{tikzpicture}[baseline=-.1cm]
	\draw (0,-.9) -- (0,.9);
\end{tikzpicture}
=
\sum_{i=1}^n
\begin{tikzpicture}[baseline=-.1cm]
	\draw (0,.4) -- (0,.9);
	\draw (0,-.4) -- (0,-.9);
	\nbox{unshaded}{(0,.4)}{.3}{0}{0}{$X_i$}
	\nbox{unshaded}{(0,-.4)}{.3}{0}{0}{$X_i$}
\end{tikzpicture}\,.
$
\end{itemize}
It is not hard to show that the modulus of $\NC_\bullet$ is $\delta=n$, and
the principal graph $\Gamma$ is the $n$-bouquet, i.e., one vertex with $n$ self-loops.

In more detail, the planar operad acts on $\NC_\bullet$ by contracting indices.
For a vector of indices $\vec{i}=(i_1,\dots, i_n)$, let $X_{\vec{i}}=X_{i_1}\cdots X_{i_n}$.
For example, if $\vec{i}=(i_1,i_2,i_3)$ and $\vec{j}=(j_1,j_2,j_3,j_4, j_5)$, then
\begin{align*}
\begin{tikzpicture}[baseline = .5cm]
	\nbox{unshaded}{(-.6,.6)}{1.4}{0}{0}{}
	\draw (.8,1.4) arc (-90:-180:.6cm);
	\draw (0,.4) arc (0:90:.8cm);
	\draw (-.4,-.2) arc (270:90:.2cm);
	\draw (-1.2,2)--(-1.2,1.2)--(-2,1.2);
	\draw (0,-.8)--(0,0)--(.8,0);
	\draw (-1.4,0) circle (.3cm);
	\nbox{unshaded}{(-1.2,1.2)}{.4}{0}{0}{}
	\nbox{unshaded}{(0,0)}{.4}{0}{0}{}
	\node at (-.5,-.5) {$\star$};
	\node at (-1.6,.6) {$\star$};
	\node at (-2.2,-.6) {$\star$};
\end{tikzpicture}
\,(X_{\vec{i}},X_{\vec{j}})
&=
\begin{tikzpicture}[baseline = .5cm]
	\nbox{unshaded}{(-.6,.6)}{1.4}{0}{0}{}
	\draw  (.8,1.4) arc (-90:-180:.6cm);
	\draw  (0,.4) arc (0:90:.8cm);
	\draw  (-.4,-.2) arc (270:90:.2cm);
	\draw (-1.2,2)--(-1.2,1.2)--(-2,1.2);
	\draw (0,-.8)--(0,0);
	\draw (0, 0)--(.8,0);
	\draw (-1.4,0) circle (.3cm);
	\nbox{unshaded}{(-1.2,1.2)}{.4}{0}{0}{}
	\nbox{unshaded}{(0,0)}{.4}{0}{0}{}
	\node at (-.5,-.5) {$\star$};
	\node at (-1.6,.6) {$\star$};
	\node at (-2.2,-.6) {$\star$};
	\node at (-1,1.8) {{\scriptsize{$i_{2}$}}};
	\node at (-1.8,1) {{\scriptsize{$i_{1}$}}};
	\node at (.4,1.4) {{\scriptsize{$k$}}};
	\node at (-.6,1.35) {{\scriptsize{$i_{3}$}}};
	\node at (.2,.6) {{\scriptsize{$j_{3}$}}};
	\node at (-.6,.35) {{\scriptsize{$j_{2}$}}};
	\node at (-.65,-.25) {{\scriptsize{$j_{1}$}}};
	\node at (.6,.2) {{\scriptsize{$j_{4}$}}};
	\node at (.2,-.6) {{\scriptsize{$j_{5}$}}};
\end{tikzpicture} 
\,(X_{\vec{i}},X_{\vec{j}})
\\
&= n\cdot \delta_{i_{3},j_{3}}\cdot \delta_{j_{1}, j_{2}}\cdot X_{i_{1}}X_{i_{2}}\left(\sum_{k=1}^{n}X_{k}^{2}\right)X_{j_{4}}X_{j_{5}}.
\end{align*}
\end{defn}

\begin{fact}
When $\delta=n$, $\cT_0 (\NC_\bullet)$ is isomorphic to $\cT_n$ and $\cO_0 (\NC_\bullet)$ is isomorphic to $\cO_n$ (see Example \ref{ex:FreeGaussian}).
\end{fact}

Voiculescu's free semicircular algebra $\cS_n$ naturally arises in this situation.
The generators $X_i$ are in $\cP_1$, and thus can be viewed as creation or annihilation operators depending upon the dot placement:
$$
\begin{tikzpicture}[baseline=-.1cm]
	\draw (-.2,0)--(-.2,.8);
	\node at (-.3,.6) {\scriptsize{$1$}};
	\nbox{unshaded}{(0,0)}{.4}{0}{0}{$X_i$}
	\draw[fill=red] (0,.4) circle (.05cm);
\end{tikzpicture}
\text{ and }
\begin{tikzpicture}[baseline=-.1cm]
	\draw (.2,0)--(.2,.8);
	\node at (.3,.6) {\scriptsize{$1$}};
	\nbox{unshaded}{(0,0)}{.4}{0}{0}{$X_i$}
	\draw[fill=red] (0,.4) circle (.05cm);
\end{tikzpicture}\,.
$$
Hence summing over all placements of the dot above gives the free semi-circular generators:
$$
L_+(X_i)+L_+(X_i)^*=
\begin{tikzpicture}[baseline=-.1cm]
	\draw (-.2,0)--(-.2,.8);
	\node at (-.3,.6) {\scriptsize{$1$}};
	\nbox{unshaded}{(0,0)}{.4}{0}{0}{$X_i$}
	\draw[fill=red] (0,.4) circle (.05cm);
\end{tikzpicture}
+
\begin{tikzpicture}[baseline=-.1cm]
	\draw (.2,0)--(.2,.8);
	\node at (.3,.6) {\scriptsize{$1$}};
	\nbox{unshaded}{(0,0)}{.4}{0}{0}{$X_i$}
	\draw[fill=red] (0,.4) circle (.05cm);
\end{tikzpicture}\,.
$$
In the general (sub)factor planar algebra setting, it has proven instrumental to look at the algebra generated by summing over all placements of the dot.

\begin{defn}
The \underline{$\cP_\bullet$-semicircular algebra $\SP{0}$} is the $C^*$-algebra generated by
$$
\set{
\sum_{\ell+\err=n}\,
\begin{tikzpicture}[baseline=.1cm]
	\draw (.2,0)--(.2,.8);
	\draw (-.2,0)--(-.2,.8);
	\node at (-.4,.6) {\scriptsize{$\ell$}};
	\node at (.4,.6) {\scriptsize{$\err$}};
	\nbox{unshaded}{(0,0)}{.4}{0}{0}{$x$}
	\draw[fill=red] (0,.4) circle (.05cm);
\end{tikzpicture}
}{x\in \cP_{n} \text{ and } n\geq 0}.
$$
\end{defn}

\begin{facts}
\mbox{}
\be
\item
$\cS_0(\NC_\bullet)$ is isomorphic to Voiculescu's free semicircular algebra $\cS_n$.
\item
In \cite{GJSCStar}, we will see that $\SP{0}$ is isomorphic to the (zeroth) Guionnet-Jones-Shlyakhtenko $C^*$-algebra \cite{MR2732052,MR2645882}.
Thus $\SP{0}$ completes to an interpolated free group factor \cite{MR2807103,MR3110503}, just as $\cS_n$ completes to $L(\bbF_n)$ \cite{MR799593}.
\item
Since $\SP{0}\cap \cK=(0)$, the inclusion $\SP{0} \hookrightarrow \TP{0}$ descends to an injection $\SP{0} \hookrightarrow \OP{0}$.
\ee
\end{facts}

\begin{rems}\label{rems:NotFunctor}
\mbox{}
\be
\item
Germain's Theorem \ref{thm:Germain} does not apply to give us the $K$-theory of $\TP{0}$ or $\SP{0}$.
The Fock space $\FP{0}$ is not generated by $\cP_1$, so $\TP{0}$ is not generated by elements of degree 1.

\item
This construction is not functorial.
Consider the canonical planar algebra inclusion $\Phi:\TL_\bullet \hookrightarrow \NC_\bullet$.
While $\Phi$ induces an inclusion of Hilbert spaces $\cF_0(\TL_\bullet)\hookrightarrow \cF_{0}(\NC_\bullet)$, $\Phi$ does not induce a map of Toeplitz algebras.
Consider the element
$$
x_{\cP_\bullet}=\,
\begin{tikzpicture}[baseline=.1cm]
	\draw (-.2,.7)--(-.2,.1) arc (-180:0:.2cm)--(.2,.7);
	\nbox{}{(0,0)}{.3}{.1}{.1}{}
	\draw[fill=red] (0,.3) circle (.05cm);
\end{tikzpicture}
\,-
\begin{tikzpicture}[baseline=.1cm]
	\draw (-.2,.7)--(-.2,.1) arc (-180:0:.2cm)--(.2,.7);
	\node at (-.4,.5) {\scriptsize{$2$}};
	\node at (.4,.5) {\scriptsize{$2$}};
	\nbox{}{(0,0)}{.3}{.1}{.1}{}
	\draw[fill=red] (0,.3) circle (.05cm);
\end{tikzpicture}
\in \TP{0}.
$$
Note that $x_{\TL_\bullet}=0$ in $\cT_{0}(\TL_\bullet)$ since $\TL_1=(0)$, but $x_{\NC_\bullet}\neq 0$ in $\cT_{0}(\NC_\bullet)$, as it is the projection onto $\NC_1$.

However, we will see in Theorem \ref{thm:OSfunctor} that the assignments $\cP_\bullet$ to $\OP{0}$ and $\SP{0}$ are functorial.
\ee
\end{rems}


\section{The operator-valued system}\label{sec:Operator}

The Fock space analogy of the previous section, originally due to Jones, brought together work of Doplicher-Roberts \cite{MR1010160}, Guionnet-Jones-Shlyakhtenko-Walker \cite{MR2732052,MR2645882,JonesSeveral}, Izumi \cite{MR1604162}, Shelly \cite{Shelly}, and Voiculescu \cite{MR799593}.
However, we saw some problems with functoriality.

In this section, we construct a canonical $C^*$-Hilbert bimodule $\cX=\cX(\cP_\bullet)$ associated to a (sub)factor planar algebra.
We will see in Subsection \ref{sec:FurtherCompression} that the Cuntz and free semicircular algebras of $\cP_\bullet$ of the previous subsection are full hereditary subalgebras of the Cuntz-Pimsner and $\cB$-valued free semicircular algebras associated to $\cX$.
In general, the $\cP_\bullet$-Toeplitz algebra may be a reduced compression of the Pimsner-Toeplitz algebra of $\cX$.

This $C^*$-Hilbert bimodule construction is functorial.
Additionally, it allows for the efficient computation of the $K$-theory of the free semicircular algebra $\SP{}$ via Theorem \ref{thm:Germain} from \cite{Germain}.
We postpone the proof to Corollary \ref{cor:KTheory}, but the interested reader can perform the calculation at the end of Subsection \ref{sec:PimsnerToeplitz}.

For this section, $\cP_\bullet$ is a fixed factor planar algebra.

\subsection{Operator-valued Fock space}\label{sec:FockOp}

\begin{defn}
For $n\geq 0$, let $\cB_{n}=\cB_n(\cP_\bullet)$ be the external direct sum
$$
\cB_n=\bigoplus_{l, r = 0}^{n} \cP_{l, r},
$$
where $\cP_{l, r} = \cP_{l + r}$, and we picture an element of $b\in \cP_{l, r}$ as
$
\begin{tikzpicture}[baseline = -.1 cm]
	\draw (-.6, 0)--(.6, 0);
	\nbox{unshaded}{(0,0)}{.3}{0}{0}{$b$}
	\node at (-.5, .2) {{\scriptsize{$l$}}};
	\node at (.5, .2) {{\scriptsize{$r$}}};
\end{tikzpicture}
$.

We define the following multiplication and (non-normalized) trace $\Tr$ on $\cB_{n}$ by
$$
\begin{tikzpicture}[baseline = 0cm]
	\draw (-.8, 0)--(.8, 0);
	\nbox{unshaded}{(0,0)}{.4}{0}{0}{$a$}
	\node at (-.6, .2) {{\scriptsize{$l$}}};
	\node at (.6, .2) {{\scriptsize{$r$}}};
\end{tikzpicture}
\, \cdot \,
\begin{tikzpicture}[baseline = 0cm]
	\draw (-.8, 0)--(.8, 0);
	\nbox{unshaded}{(0,0)}{.4}{0}{0}{$b$}
	\node at (-.6, .2) {{\scriptsize{$l'$}}};
	\node at (.6, .2) {{\scriptsize{$r'$}}};
\end{tikzpicture}\,
= \, \delta_{r, l'} \,
\begin{tikzpicture}[baseline = 0cm]
	\draw (-.8, 0)--(2, 0);
	\nbox{unshaded}{(0,0)}{.4}{0}{0}{$a$}
	\node at (-.6, .2) {{\scriptsize{$l$}}};
	\node at (.6, .2) {{\scriptsize{$r$}}};
    \nbox{unshaded}{(1.2,0)}{.4}{0}{0}{$b$}
	\node at (1.8, .2) {{\scriptsize{$r'$}}};
\end{tikzpicture}
\, \, \text{ and } \, \,
\Tr(b) = \delta_{l, r}\,
\begin{tikzpicture}[baseline = -.2cm]
	\draw (-.4, 0)--(.4, 0) arc(90:-90:.4cm)--(-.4, -.8) arc(270:90:.4cm);
	\nbox{unshaded}{(0,0)}{.4}{0}{0}{$b$}
	\node at (-.6, .2) {{\scriptsize{$l$}}};
	\node at (.6, .2) {{\scriptsize{$r$}}};
\end{tikzpicture}\,.
$$
We equip $\cB_{n}$ with involution $\dagger$ defined as follows.
If $b \in \cP_{l, r}$, then
$$
b^{\dagger} =
\begin{tikzpicture}[baseline = 0cm]
	\draw (-.8, 0)--(.8, 0);
	\nbox{unshaded}{(0,0)}{.4}{0}{0}{$b^{*}$}
	\node at (-.6, .2) {{\scriptsize{$r$}}};
	\node at (.6, .2) {{\scriptsize{$l$}}};
\end{tikzpicture}\,
\in  \cP_{r, l},
$$
where the $*$ is the standard involution on $\cP_\bullet$.
Notice that $\cB_{n}$ is a finite dimensional $C^{*}$-algebra under $\dagger$, so $\cB_n$ has a unique $C^{*}$-norm.
\end{defn}

\begin{defn}
Let $\Gamma(n)$ be the truncation of $\Gamma$ to depth $n$, and let $V(\Gamma(n))$ be its vertices.
\end{defn}

\begin{rem}\label{rem:Bnmatrix}
Note that
$$
\cB_{n} \cong \bigoplus_{\alpha \in V(\Gamma_{n})} M_{n(\alpha)}(\C),
$$
where $n(\alpha)$ represents the multiplicity of $p_{\alpha}$ in $\cB_{n}$, and $p_\alpha$ is a representative of the vertex $\alpha\in V(\Gamma(n))$.
\end{rem}

Observe that $\cB_{n}$ is canonically a non-unital, hereditary subalgebra of $\cB_{n+1}$.
Furthermore, the uniqueness of the $C^{*}$-norm on each $\cB_{n}$ implies that there is a unique $C^{*}$-norm on $\cup_{n} \cB_{n}$.

\begin{defn}
Define the $C^*$-algebra $\cB=\cB(\cP_\bullet) = \varinjlim \cB_n = \overline{\bigcup_{n\geq 0} \cB_{n}}^{\|\cdot\|}$.
\end{defn}

\begin{rem}
From the preceding discussion, as a $C^{*}$-algebra, $\cB \cong \bigoplus_{\alpha \in V(\Gamma)} \cK$.
\end{rem}

In what follows, we construct $\cB-\cB$ Hilbert bimodules. 

\begin{defn}
Consider the external direct sum
$$
X_n=\bigoplus_{l, r = 1}^{\infty} \cP_{l, n, r},
$$
where diagrammatically, an element $x \in \cP_{l, n, r}$ can be seen as 
$
x = \begin{tikzpicture}[baseline = -.1 cm]
	\draw (-.6, 0)--(.6, 0);
	\draw (0, 0)--(0, .6);
	\nbox{unshaded}{(0,0)}{.3}{0}{0}{$x$}
	\node at (-.5, .2) {{\scriptsize{$l$}}};
	\node at (.5, .2) {{\scriptsize{$r$}}};
	\node at (.15, .45) {{\scriptsize{$n$}}};
\end{tikzpicture}
$.
We use the convention that $X=X_1$.
We have an involution $\dagger$ on $X_n$ defined as follows.
If $x \in \cP_{l, n, r}$, then
$$
x^\dagger = \begin{tikzpicture}[baseline = 0cm]
	\draw (-.8, 0)--(.8, 0);
	\draw (0, 0)--(0, .8);
	\nbox{unshaded}{(0,0)}{.4}{0}{0}{$x^{*}$}
        \node at (.2, .6) {{\scriptsize{$n$}}};
	\node at (-.6, .2) {{\scriptsize{$r$}}};
	\node at (.6, .2) {{\scriptsize{$l$}}};
\end{tikzpicture} \in \cP_{r, n, l}\,.
$$
Moreover, $X_n$ has a $\cB$-valued inner product given by the sesquilinear extension of
$$
\langle x| y \rangle_{\cB}
= \delta_{l, l'}\,
\begin{tikzpicture}[baseline = 0cm]
	\draw (-.8, 0)--(2, 0);
	\draw (0, .4) arc (180:90:.3cm) -- (.9,.7) arc (90:0:.3cm);
	\node at (.6, .9) {\scriptsize{$n$}};
	\nbox{unshaded}{(0,0)}{.4}{0}{0}{$x^{*}$}
	\node at (-.6, .2) {{\scriptsize{$r$}}};
	\node at (.6, .2) {{\scriptsize{$l$}}};
	\nbox{unshaded}{(1.2,0)}{.4}{0}{0}{$y$}
	\node at (1.8, .2) {{\scriptsize{$r'$}}};
\end{tikzpicture}
$$
for $x\in \cP_{l,n,r}$ and $y\in\cP_{l',n,r'}$.
However, note that so far, the $\cB$-valued inner product only takes values in $\cup_{n\geq 0}\cB_n$.

Define a norm on $X_n$ by $\|x\|_{X_n}^2=\| \langle x| x\rangle_\cB\|_\cB$.  
\end{defn}

Before we continue, we need to show that the involution $\dagger$ is continuous on $X_{n}$.  
The following two propositions accomplish just that.

\begin{prop}\label{prop:expectation}
Consider the capping map $R_{k}:\cup_{n\geq 0}\cB_n \rightarrow \cup_{n\geq 0}\cB_n$ which is given by linear extension of
$$
R_{k}\left(
\begin{tikzpicture}[baseline = -.1 cm]
	\draw (-.8, 0)--(.8, 0);
	\nbox{unshaded}{(0,0)}{.4}{0}{0}{$b$}
	\node at (-.6, .2) {{\scriptsize{$l$}}};
	\node at (.6, .2) {{\scriptsize{$r$}}};
\end{tikzpicture}
\right)
\,= \, 
\begin{cases}
\begin{tikzpicture}[baseline = -.1 cm]
	\draw (-.8, 0)--(.8, 0);
	\draw (-.4, .2) arc(270:90:.2cm) -- (.4, .6) arc(90:-90:.2cm);
	\nbox{unshaded}{(0,0)}{.4}{0}{0}{$b$}
	\node at (-1, .2) {{\scriptsize{$l - k$}}};
	\node at (1, .2) {{\scriptsize{$r - k$}}};
	\node at (0, .8) {{\scriptsize{$k$}}};
\end{tikzpicture} & \text{ if } l, r \geq k \\
0 & \text{ otherwise.} 
\end{cases}
$$
Then $\|R_{k}(b)\|_\cB \leq \delta^{k}\|b\|_\cB$, so $R_{k}$ extends continuously to $\cB$.
\end{prop}

\begin{proof}
Let $b \in \cB_{n}$ for some $n$.  We may assume that $b \in \bigoplus_{l, r \geq k } \cP_{l, r}$.  Let 
$$
\iota_{n,k}: \cB_{n} \rightarrow \left(\sum_{j=k}^{n+k} 1_{j} \right)\cB\left(\sum_{j = k}^{n+k} 1_{j} \right)
\text{ by }
a\mapsto 
\begin{tikzpicture}[baseline = .1 cm]
	\draw (-.8, 0)--(.8, 0);
	\draw (-.8, .6)--(.8, .6);
	\nbox{unshaded}{(0,0)}{.4}{0}{0}{$a$}
	\node at (0, .8) {{\scriptsize{$k$}}}; 	
\end{tikzpicture}\,.
$$
Note that $\iota_{n, k}$ is a unital inclusion.  Let 
$$
E_{n, k}: \left(\sum_{j=k}^{n+k} 1_{j} \right)\cB\left(\sum_{j = k}^{n+k} 1_{j} \right) \rightarrow \iota_{n, k}(\cB_{n})
$$ 
be the $\Tr$-preserving conditional expectation.  
Then $E_{n, k}$ has norm 1.  
It is a straightforward diagrammatic argument to see for all $b$ as above, $R_{k}(b) = \delta^{k}E_{n, k}(b)$.
This gives the desired bound.
\end{proof}

\begin{prop}\label{prop:cont.dagger}
For $x \in X_{n}$, $\|x\|_{X_n} \leq \delta^{n/2}\|x^{\dagger}\|_{X_n}$.
\end{prop}

\begin{rem}
Note that this proposition also implies $\|x\|_{X_n} \geq \delta^{-n/2}\|x^{\dagger}\|_{X_n}$.
\end{rem}

\begin{proof}
Let $x \in X_{n}$, and write $x$ as a finite sum
$$
x = \sum_{l, r} \begin{tikzpicture}[baseline = 0cm]
	\draw (-.8, 0)--(.8, 0);
	\draw (0, 0)--(0, .8);
	\nbox{unshaded}{(0,0)}{.4}{0}{0}{$x_{l, r}$}
        \node at (.2, .6) {{\scriptsize{$n$}}};
	\node at (-.6, .2) {{\scriptsize{$l$}}};
	\node at (.6, .2) {{\scriptsize{$r$}}};
\end{tikzpicture}\, .
$$
Let $\tilde{x}_{l, r}$ be the following element of $\cB$:
$$
\tilde{x}_{l, r} = \begin{tikzpicture}[baseline = 0cm]
	\draw (-.8, 0)--(.8, 0);
	\nbox{unshaded}{(0,0)}{.4}{0}{0}{$x_{l, r}$}
        \node at (-.9, .2) {{\scriptsize{$l + n$}}};
	\node at (.6, .2) {{\scriptsize{$r$}}};
\end{tikzpicture}\, ,
$$
and let $\tilde{x} = \sum_{l, r}\tilde{x}_{l, r}$.  
By drawing the right diagrams, $\|x\|_{X_n} = \|\tilde{x}^{\dagger}\cdot\tilde{x}\|_{\cB}^{1/2}$, and
$$
\|x^{\dagger}\|_{X_n} 
= 
\|R_{n}(\tilde{x}\cdot\tilde{x}^{\dagger})\|_\cB^{1/2} 
\leq 
\delta^{n/2} \|\tilde{x}\cdot\tilde{x}^{\dagger}\|_\cB^{1/2} 
= 
\delta^{n/2}\cdot\|\tilde{x}^{\dagger}\cdot\tilde{x}\|_\cB^{1/2} 
= 
\delta^{n/2}\|x\|_{X_n}.
$$
In the string of inequalities, we used that $\cB$ is a $C^{*}$-algebra under the involution $\dagger$.
\end{proof}

\begin{defn}
Let $\cX_n=\cX_n(\cP_\bullet)$ be the completion of $X_n$ with respect to $\|\cdot\|_{X_n}$,
where again, we use the convention that $\cX=\cX_1$.
We also use the convention that $\cX_0=\cB$.  
By Proposition \ref{prop:cont.dagger}, the involution $\dagger$ extends continuously to each $\cX_{n}$.

For each $m$, there are natural left and right actions of $\cB_m$ on $\cX_n$, namely the extensions of the actions on $X_n$ given by
\begin{align*}
\begin{tikzpicture}[baseline = -.1 cm]
	\draw (-.8, 0)--(.8, 0);
	\nbox{unshaded}{(0,0)}{.4}{0}{0}{$b$}
	\node at (-.6, .2) {{\scriptsize{$l$}}};
	\node at (.6, .2) {{\scriptsize{$r$}}};
\end{tikzpicture}
\,\cdot \,
\begin{tikzpicture}[baseline = -.1 cm]
	\draw (-.8, 0)--(.8, 0);
	\draw (0, 0)--(0, .8);
	\nbox{unshaded}{(0,0)}{.4}{0}{0}{$x$}
	\node at (-.6, .2) {{\scriptsize{$l'$}}};
	\node at (.6, .2) {{\scriptsize{$r'$}}};
\end{tikzpicture}\,
&= \delta_{r, l'} \cdot
\begin{tikzpicture}[baseline = -.1 cm]
	\draw (-.8, 0)--(2, 0);
	\draw (1.2, 0)--(1.2, .8);
	\nbox{unshaded}{(0,0)}{.4}{0}{0}{$b$}
	\node at (-.6, .2) {{\scriptsize{$l$}}};
	\node at (.6, .2) {{\scriptsize{$r$}}};
	\node at (1.8, .2) {{\scriptsize{$r'$}}};
	\nbox{unshaded}{(1.2,0)}{.4}{0}{0}{$x$};
\end{tikzpicture}\, \text{ and }
\\
\begin{tikzpicture}[baseline = -.1 cm]
	\draw (-.8, 0)--(.8, 0);
    \draw (0, 0)--(0, .8);
	\nbox{unshaded}{(0,0)}{.4}{0}{0}{$x$}
	\node at (-.6, .2) {{\scriptsize{$l'$}}};
	\node at (.6, .2) {{\scriptsize{$r'$}}};
\end{tikzpicture}\, \cdot \, \begin{tikzpicture}[baseline = -.1 cm]
	\draw (-.8, 0)--(.8, 0);
	\nbox{unshaded}{(0,0)}{.4}{0}{0}{$b$}
	\node at (-.6, .2) {{\scriptsize{$l$}}};
	\node at (.6, .2) {{\scriptsize{$r$}}};
\end{tikzpicture}\,
&=
\delta_{r', l} \cdot
\begin{tikzpicture}[baseline = -.1 cm]
	\draw (-.8, 0)--(2, 0);
	\draw (0, 0)--(0, .8);
	\nbox{unshaded}{(1.2,0)}{.4}{0}{0}{$b$}
	\node at (-.6, .2) {{\scriptsize{$l'$}}};
	\node at (.6, .2) {{\scriptsize{$l$}}};
	\node at (1.8, .2) {{\scriptsize{$r$}}};
	\nbox{unshaded}{(0,0)}{.4}{0}{0}{$x$};
\end{tikzpicture}\, .
\end{align*}
Moreover, the operator norm of the left action of $b\in \cB_m$ on $\cX_n$ is equal to $\|b\|_\cB$ since $\cB_m$ has a unique $C^*$-norm, so we get an isometric embedding of $\cB$ into $\cL(\cX_n)$.
\end{defn}

\begin{prop}\label{prop:OperatorValuedFock}
$\cX_n \cong \bigotimes_{\cB}^n \cX$.
\end{prop}
\begin{proof}
We first prove by induction on $n$ that the $\cB$ valued inner product on $\bigotimes_{\cB}^n \cX$ is given on simple tensors by
$$ \langle x_{1} \otimes \cdots \otimes x_{n} | y_{1} \otimes \cdots \otimes y_{n} \rangle_{\cB} =
 \begin{tikzpicture}[baseline = 0cm]
	\draw (-1.2, 0)--(.4, 0);
	\draw (-.4, .4) arc(180:90:.6cm) -- (4.6, 1) arc(90:0:.6cm);
	\nbox{unshaded}{(-.4,0)}{.4}{0}{0}{$x_{n}^{*}$}
    \node at (.7, 0) {$\scriptsize{\cdots}$};
	\draw (1, 0)--(2.6, 0);
	\draw (1.8, .4) arc(180:90:.2cm) -- (2.8, .6) arc(90:0:.2cm);
	\nbox{unshaded}{(1.8,0)}{.4}{0}{0}{$x_{1}^{*}$}
    \draw (2.6, 0)--(3.8, 0);
	\nbox{unshaded}{(3,0)}{.4}{0}{0}{$y_{1}$}
    \node at (4.1, 0) {$\scriptsize{\cdots}$};
	\draw (4.4, 0)--(6, 0);
	\nbox{unshaded}{(5.2,0)}{.4}{0}{0}{$y_{n}$}
\end{tikzpicture} \, .
$$
For $n=1$, this is a tautology.  Assume the result holds for $\bigotimes_{\cB}^{n-1} \cX$.  Then on $\bigotimes_{\cB}^n \cX$, we have
$$
\langle x_{1} \otimes x_{2} \otimes \cdots \otimes x_{n} | y_{1} \otimes y_{2} \otimes \cdots \otimes y_{n} \rangle_{\cB} = \langle x_{2} \otimes \cdots \otimes x_{n}|\,  \langle x_{1}|y_{1} \rangle_{\cB} \cdot y_{2} \otimes \cdots \otimes y_{n}\rangle_{\cB}.
$$
By the induction hypothesis, this inner product is
$$
\begin{tikzpicture}[baseline = 0cm]
	\draw (-1.2, 0)--(.4, 0);
	\draw (-.4, .4) arc(180:90:.6cm) -- (4.6, 1) arc(90:0:.6cm);
	\nbox{unshaded}{(-.4,0)}{.4}{0}{0}{$x_{n}^{*}$}
	\node at (.7, 0) {$\scriptsize{\cdots}$};
	\draw (1, 0)--(2.6, 0);
	\draw (1.8, .4) arc(180:90:.2cm) -- (2.8, .6) arc(90:0:.2cm);
	\nbox{unshaded}{(1.8,0)}{.4}{0}{0}{$x_{1}^{*}$}
	\draw (2.6, 0)--(3.8, 0);
	\nbox{unshaded}{(3,0)}{.4}{0}{0}{$y_{1}$}
	\node at (4.1, 0) {$\scriptsize{\cdots}$};
	\draw (4.4, 0)--(6, 0);
	\nbox{unshaded}{(5.2,0)}{.4}{0}{0}{$y_{n}$}
\end{tikzpicture} \, .
$$
We now define a map $U_{n}: \bigotimes_{\cB}^n \cX \rightarrow \cX_{n}$ by
$$
U_{n}(x_{1} \otimes \cdots \otimes x_{n}) =
\begin{tikzpicture}[baseline = 0cm]
	\draw (-1.2, 0)--(.4, 0);
	\draw (-.4, .0)  -- (-.4, .8);
	\nbox{unshaded}{(-.4,0)}{.4}{0}{0}{$x_{1}$}
	\node at (.7, 0) {$\scriptsize{\cdots}$};
	\draw (1, 0)--(2.6, 0);
	\draw (1.8, 0) -- (1.8, .8);
	\nbox{unshaded}{(1.8,0)}{.4}{0}{0}{$x_{n}$}
\end{tikzpicture}\, .
$$
Note that $U_{n}$ is a $\cB-\cB$ bimodule map which is $\cB$ middle linear and preserves the $\cB$ valued inner product.  Therefore it is an isometry.  Let $x \in \cP_{l, n, r}$, and define $\widehat{x}$ and $c_{k} \in \cP_{k+1, k}$ as follows:
$$
\widehat{x} =
\begin{tikzpicture}[baseline = -.1 cm]
	\draw (-.8, 0)--(.8, 0);
	\nbox{unshaded}{(0,0)}{.4}{0}{0}{$x$}
	\node at (-.6, .2) {{\scriptsize{$l$}}};
	\node at (.8, .2) {{\scriptsize{$r+n$}}};
\end{tikzpicture}\,
\text{ and } c_{k} =
\begin{tikzpicture}[baseline = -.1cm]
	\nbox{unshaded}{(0,0)}{.4}{0}{0}{}
	\draw (-.4, -.1)--(.4, -.1);
	\draw (-.4, .1) arc(-90:0:.3cm);
	\node at (0, -.25) {{\scriptsize{$k$}}};
\end{tikzpicture}\,.
$$
Then $\widehat{x} \in \cB$, $c_{k} \in \cX$, and $U_{n}(\widehat{x}c_{n + r - 1} \otimes c_{n+r-2} \otimes \cdots \otimes c_{r}) = x$.
Therefore, $U_{n}$ has dense range, and since it is an isometry, it is surjective.
\end{proof}

\begin{defn}
The \underline{Pimsner-Fock space} is given by $\FP{}=\bigoplus_{n\geq 0}\cX_n\cong \bigoplus_{n\geq 0} \bigotimes_{\cB}^{n} \cX$.
\end{defn}

\begin{rem}
The involution $\dagger$ defined on each $\cX_{n}$ does \emph{not} extend continuously to $\FP{}$.
Indeed, if we set $C_{k}$ to be the diagram
$
\begin{tikzpicture}[baseline = -.1cm]
	\nbox{unshaded}{(0,0)}{.4}{0}{0}{}
	\draw (-.4, 0) arc(-90:0:.4cm);
	\node at (0, 0) {{\scriptsize{$k$}}};
\end{tikzpicture}\,,
$
then $\|C_{k}\|_{\cB} = \delta^{n/2}$ and $\|C_{k}^{\dagger}\|_{\cB} = 1$.  An involution on $\FP{}$ is not required in the sections that follow.  
We will only use that $\cX$ has a well defined involution.
\end{rem}

Before we define our Pimsner-Toeplitz algebra, we need to know that $\cB$ acts on $\cX$ by compact operators. 

\begin{prop}
Let $p_{\alpha}$ be a representative of the vertex $\alpha \in V(\Gamma)$.  
Then $p_{\alpha}$ is a linear combination of rank 1 operators. 
Hence $\cB \subset \cK(\cX)$.
\end{prop}

\begin{proof}
We write $\beta \sim \alpha$ if $\alpha$ and $\beta$ are the endpoints of at least one edge in $E(\Gamma)$.   
For each $\beta \sim \alpha$, we observe that $\dim(p_{\alpha}\cX p_{\beta})$ is $n_{\beta}$ where there are $n_{\beta}$ edges with endpoints $\alpha$ and $\beta$.  
Pick an orthogonal basis $\{g(\beta)_{1}, \dots, g(\beta)_{n_{\beta}}\}$ for $p_{\alpha}\cX p_{\beta}$ under the inner product $\Tr ( \langle \cdot|\cdot \rangle_{\cB})$.  
We note that as $p_{\beta}$ is minimal in $\cB$, $\langle g(\beta)_{i}| g(\beta')_{j} \rangle_{\cB} = k\delta_{i, j}\delta_{\beta, \beta'} p_{\beta}$ for $k$ a nonzero constant. 
Therefore, we can choose the $g$'s so that $k = 1$.  
(See Subsection \ref{sec:Induced} for more details about these bases.) 

We claim
$$
p_{\alpha} = \sum_{\beta \sim \alpha} \sum_{i=1}^{n_{\beta}} |g(\beta)_{i}\rangle\langle g(\beta)_{i}|.
$$ 
Indeed, if $x \in p_{\gamma}\cX$ and $\gamma \neq \alpha$, then both sides of the equation send $x$ to zero.  
Otherwise, we may assume $x \in p_{\alpha}\cX$ is of the form
$$
x = 
\begin{tikzpicture}[baseline = -.1cm]
	\draw (-.8, 0)--(2, 0);
	\draw (0, 0)--(0, .8);
	\nbox{unshaded}{(0,0)}{.4}{0}{0}{\scriptsize{$g(\beta)_{i}$}}
	\nbox{unshaded}{(1.2,0)}{.4}{0}{0}{$b$}
\end{tikzpicture}
$$
for $\beta \sim \alpha$ and $b \in \cB$.  Clearly, $p_{\alpha}\cdot x = x$.  As for $y = \sum_{i=1}^{n_{\beta}} |g_{i}\rangle\langle g_{i}|$, the orthogonality relations give
$$
y\cdot x = g(\beta)_{i}\langle g(\beta)_{i}|g(\beta)_{i}\rangle_{\cB}\cdot b = (g(\beta)_{i}p_{\beta})b = g(\beta_{i})\cdot b = x,
$$
so we have the desired equality.
\end{proof}

\subsection{The Pimsner-Toeplitz and free semicircular algebras}\label{sec:PimsnerToeplitz}

On $\FP{}$, we have creation and annihilation operators as follows.

\begin{defn}\label{defn:CreationDiagram}
Let $x \in \cP_{l,1,r}\subset \cX$.
We define the $x$ creation operator $L_+(x)$ on $\FP{}$ by the linear extension of its action on $y\in \cP_{l',n,r'}$:
$$
L_+(x)y=
\left(
\begin{tikzpicture}[baseline=-.1cm]
	\draw (-.2,0)--(-.2,.8);
	\draw (-.8, 0)--(.8, 0);
	\node at (-.6,.2) {\scriptsize{$l$}};
	\node at (.6,.2) {\scriptsize{$r$}};
	\nbox{unshaded}{(0,0)}{.4}{0}{0}{$x$}
	\draw[fill=red] (0,.4) circle (.05cm);
\end{tikzpicture}
\right)
\begin{tikzpicture}[baseline=-.1cm]
	\draw (0,0)--(0,.8);
	\draw (-.8, 0)--(.8, 0);
	\node at (-.6,.2) {\scriptsize{$l'$}};
	\node at (.6,.2) {\scriptsize{$r'$}};
	\node at (.2,.6) {\scriptsize{$n$}};	
	\nbox{unshaded}{(0,0)}{.4}{0}{0}{$y$}
\end{tikzpicture}
=
\delta_{r,l'}\,
\begin{tikzpicture}[baseline = -.1cm]
	\draw (-.8, 0)--(2, 0);
	\draw (0, 0)--(0, .8);
	\draw (1.2, 0)--(1.2, .8);
	\node at (-.6,.2) {\scriptsize{$l$}};
	\node at (.6,.2) {\scriptsize{$r$}};
	\node at (1.8,.2) {\scriptsize{$r'$}};
	\node at (1.4,.6) {\scriptsize{$n$}};
	\nbox{unshaded}{(0,0)}{.4}{0}{0}{$x$}
	\nbox{unshaded}{(1.2,0)}{.4}{0}{0}{$y$}
\end{tikzpicture}\,.
$$
We define the $x$ annihilation operator $L_-(x)$ on $\FP{}$ by the extension of
$$
L_-(x)y=
\left(
\begin{tikzpicture}[baseline=-.1cm]
	\draw (.2,0)--(.2,.8);
	\draw (-.8, 0)--(.8, 0);
	\node at (-.6,.2) {\scriptsize{$l$}};
	\node at (.6,.2) {\scriptsize{$r$}};
	\nbox{unshaded}{(0,0)}{.4}{0}{0}{$x$}
	\draw[fill=red] (0,.4) circle (.05cm);
\end{tikzpicture}
\right)
\begin{tikzpicture}[baseline=-.1cm]
	\draw (0,0)--(0,.8);
	\draw (-.8, 0)--(.8, 0);
	\node at (-.6,.2) {\scriptsize{$l'$}};
	\node at (.6,.2) {\scriptsize{$r'$}};
	\node at (.2,.6) {\scriptsize{$n$}};	
	\nbox{unshaded}{(0,0)}{.4}{0}{0}{$y$}
\end{tikzpicture}
=
\delta_{r,l'}\,
\begin{tikzpicture}[baseline = -.1cm]
	\draw (-.8, 0)--(2, 0);
	\draw (0.2, .4) arc (180:0:.4cm);
	\draw (1.2, 0)--(1.2, .8);
	\node at (-.6,.2) {\scriptsize{$l$}};
	\node at (.6,.2) {\scriptsize{$r$}};
	\node at (1.8,.2) {\scriptsize{$r'$}};
	\node at (1.6,.6) {\scriptsize{$n-1$}};
	\nbox{unshaded}{(0,0)}{.4}{0}{0}{$x$}
	\nbox{unshaded}{(1.2,0)}{.4}{0}{0}{$y$}
\end{tikzpicture}\,.$$
Note that $(L_{+}(x))^{*} = L_{-}(x^{\dagger})$.
\end{defn}

The following proposition shows these operators generate a Pimsner-Toeplitz algebra:

\begin{prop}\label{prop:OperatorValuedToeplitz}
The operators $L_{+}(x)$ and $L_{-}(x)$ acting on $\FP{}$ are bounded.  
Furthermore, together with $\cB$, the $C^{*}$-algebra they generate is isomorphic to the Pimsner-Toeplitz algebra $\cT(\cX)$. \end{prop}

\begin{proof}
Consider the map $U: \bigoplus_{n\geq 0} \bigotimes_{\cB}^{n} \cX \rightarrow \FP{}$ which is defined by $U(y_{1} \otimes \cdots \otimes y_{n}) = U_{n}(y_{1} \otimes \cdots \otimes y_{n})$ with $U_{n}$ as in the proof of Proposition \ref{prop:OperatorValuedFock}.  
The operator $U$ is unitary, and it is easy to verify that $U$ intertwines the action of the two definitions of $L_+(x)$ from Definitions \ref{def:Pimsner-Toeplitz} and \ref{defn:CreationDiagram}.
\end{proof}

\begin{defn}
We refer to the $C^{*}$-algebra generated by $\cB$ and $\set{L_{+}(x)}{x \in \cX}$ as $\TP{}$.
\end{defn}

\begin{rem}\label{rem:dot}
One can think of $\TP{}$ as being generated by operators of the form
$$
L_{\err}(x) =
\begin{tikzpicture}[baseline = -.1cm]
	\draw (-.8, 0)--(.8, 0);
	\draw (-.2, 0)--(-.2, .8);
	\draw (.2, 0)--(.2, .8);
	\node at (-.6, .2) {{\scriptsize{$l$}}};
	\node at (.6, .2) {{\scriptsize{$r$}}};
	\node at (-.4, .6) {{\scriptsize{$\ell$}}};
	\node at (.4, .6) {{\scriptsize{$\err$}}};
	\nbox{unshaded}{(0,0)}{.4}{0}{0}{$x$}
	\draw[fill=red] (0,.4) circle (.05cm);
\end{tikzpicture}
$$
where
$$
\begin{tikzpicture}[baseline = -.1cm]
	\draw (-.8, 0)--(.8, 0);
	\draw (-.2, 0)--(-.2, .8);
	\draw (.2, 0)--(.2, .8);
    \node at (-.6, .2) {{\scriptsize{$l$}}};
	\node at (.6, .2) {{\scriptsize{$r$}}};
	\node at (-.4, .6) {{\scriptsize{$\ell$}}};
	\node at (.4, .6) {{\scriptsize{$\err$}}};
	\nbox{unshaded}{(0,0)}{.4}{0}{0}{$x$}
	\draw[fill=red] (0,.4) circle (.05cm);
\end{tikzpicture}
\left(
\begin{tikzpicture}[baseline = -.1cm]
	\draw (-.8, 0)--(.8, 0);
	\draw (0, 0)--(0, .8);
	\node at (-.2, .6) {{\scriptsize{$k$}}};
	\nbox{unshaded}{(0,0)}{.4}{0}{0}{$\xi$}
\end{tikzpicture}
\right)
=
\begin{cases}
0 &\text{ if } k < \err \\
\begin{tikzpicture}[baseline = -.1cm]
	\draw (-.8, 0)--(2, 0);
	\draw (-.2, 0)--(-.2, .8);
	\draw (.2, 0)--(.2, .4) arc(180:0:.4cm);
	\draw (1.4, 0)--(1.4, .8);
    \node at (-.6, .2) {{\scriptsize{$l$}}};
	\node at (.6, .2) {{\scriptsize{$r$}}};
	\node at (-.4, .6) {{\scriptsize{$\ell$}}};
	\node at (.6, .6) {{\scriptsize{$\err$}}};
	\node at (1.8, .6) {{\scriptsize{$k-\err$}}};
	\nbox{unshaded}{(0,0)}{.4}{0}{0}{$x$}
	\nbox{unshaded}{(1.2,0)}{.4}{0}{0}{$\xi$}
\end{tikzpicture}
& \text{if } \err \geq k.
\end{cases}
$$
Indeed, if $c_{k}$ is as in Proposition \ref{prop:OperatorValuedFock}, then
$$
L_{\err}(x) =L_{+}(c_{\ell}^{\dagger}) \cdots L_{+}(c_{l + \ell - 2}^{\dagger})L_{+}(c_{l + \ell - 1}^{\dagger})(b)L_{-}(c_{r + \err - 1})L_{-}(c_{r + \err - 2})\cdots L_{-}(c_{\err})
$$
where
$
b=
\begin{tikzpicture}[baseline = -.1 cm]
	\draw (-.6, 0)--(.6, 0);
	\nbox{unshaded}{(0,0)}{.3}{0}{0}{$x$}
	\node at (-.7, .2) {{\scriptsize{$l+\ell$}}};
	\node at (.7, .2) {{\scriptsize{$r+\err$}}};
\end{tikzpicture}
\in\cB
$.
\end{rem}

\begin{defn}
Let $\cX_{\R}$ be the closed real subspace $\set{x \in \cX}{x = x^{\dagger}}\subset \cX$.  
We define the \underline{$\cB$-valued semicircular algebra $\SP{}$} to be the $C^{*}$-subalgebra of $\TP{}$ spanned by the (self-adjoint) elements $L_{+}(\xi) + L_{-}(\xi)$ for $\xi \in \cX_{\R}$.
\end{defn}

Note that $\SP{}$ is isomorphic to the algebra generated by $\set{L_{+}(x) + L_{+}(x)^{*}}{x \in \cX_{\R}}$ and $\cB$  as operators on $\bigoplus_{n\geq 0} \bigotimes_{\cB}^{n} \cX$.  
We may also view generators of $\SP{}$, just as in the presentation of $\SP{0}$ is Subsection \ref{sec:SP}, where we sum over all placements of the dot on the top.

\begin{prop}\label{prop:cyclic}
$\SP{}$ is generated by $\cB$ and 
$\D
\set{
\sum_{\ell+ \, \err \, = n}
\begin{tikzpicture}[baseline = .1cm]
	\draw (-.8, 0)--(.8, 0);
	\draw (-.2, 0)--(-.2, .8);
	\draw (.2, 0)--(.2, .8);
	\node at (-.6, .2) {{\scriptsize{$l$}}};
	\node at (.6, .2) {{\scriptsize{$r$}}};
	\node at (-.4, .6) {{\scriptsize{$\ell$}}};
	\node at (.4, .6) {{\scriptsize{$\err$}}};
	\nbox{unshaded}{(0,0)}{.4}{0}{0}{$x$}
	\draw[fill=red] (0,.4) circle (.05cm);
\end{tikzpicture}
}{
n\geq 0 \text{ and }x\in\cP_{l,n,r}
}$.
\end{prop}

\begin{proof}
Let $\cA$ be the $C^*$-algebra generated be the sums above.  It is clear that $\SP{} \subset \cA$.  
To show the other inclusion, will prove by induction on $n$ that each sum is in $\SP{}$.

If $n = 1$, we see that 
$$
\begin{tikzpicture}[baseline = -.1 cm]
	\draw (-.8, 0)--(.8, 0);
	\draw (.2, 0)--(.2, .8);
	\nbox{unshaded}{(0,0)}{.4}{0}{0}{$x$}
	\draw[fill=red] (0,.4) circle (.05cm);
\end{tikzpicture}
\, + \, 
\begin{tikzpicture}[baseline = -.1 cm]
	\draw (-.8, 0)--(.8, 0);
	\draw (-.2, 0)--(-.2, .8);
	\nbox{unshaded}{(0,0)}{.4}{0}{0}{$x$}
	\draw[fill=red] (0,.4) circle (.05cm);
\end{tikzpicture} 
= 
\frac{1}{2}[L_{+}(x + x^{\dagger}) + L_{-}(x + x^{\dagger}) + i(L_{+}(-i(x - x^{\dagger})) + L_{-}(-i(x - x^{\dagger})))].
$$

We now assume that the result holds for up to $n-1$ marked points at the top of the box.  We need only show that 
$$
\sum_{m + k = n}
\begin{tikzpicture}[baseline = .1cm]
	\draw (-.8, 0)--(.8, 0);
	\draw (-.2, 0)--(-.2, .8);
	\draw (.2, 0)--(.2, .8);
	\node at (-.4, .6) {{\scriptsize{$\ell$}}};
	\node at (.4, .6) {{\scriptsize{$\err$}}};
	\nbox{unshaded}{(0,0)}{.4}{0}{0}{$x$}
	\draw[fill=red] (0,.4) circle (.05cm);
\end{tikzpicture}\, \in \SP{}
$$ 
when $x$ is of the form
$$
x = 
\begin{tikzpicture}[baseline = 0cm]
	\draw (-1.2, 0)--(1.6, 0);
	\draw (-.4, 0)--(-.4, .8);
	\nbox{unshaded}{(-.4,0)}{.4}{0}{0}{$x_{1}$}
	\draw (.8, 0)--(.8, .8);
	\nbox{unshaded}{(.8,0)}{.4}{0}{0}{$y$}
	\node at (1.3, .6) {{\scriptsize{$n-1$}}};
\end{tikzpicture}\,.
$$
We see that 
$$
\left(
\begin{tikzpicture}[baseline = .1 cm]
	\draw (-.8, 0)--(.8, 0);
	\draw (.2, 0)--(.2, .8);
	\nbox{unshaded}{(0,0)}{.4}{0}{0}{$x_1$}
	\draw[fill=red] (0,.4) circle (.05cm);
\end{tikzpicture}
\, + \, 
\begin{tikzpicture}[baseline = .1 cm]
	\draw (-.8, 0)--(.8, 0);
	\draw (-.2, 0)--(-.2, .8);
	\nbox{unshaded}{(0,0)}{.4}{0}{0}{$x_1$}
	\draw[fill=red] (0,.4) circle (.05cm);
\end{tikzpicture}
\right) 
\cdot \sum_{\ell \,+ \, \err \, = n-1}\, 
\begin{tikzpicture}[baseline = .1cm]
	\draw (-.8, 0)--(.8, 0);
	\draw (-.2, 0)--(-.2, .8);
	\draw (.2, 0)--(.2, .8);
	\node at (-.4, .6) {{\scriptsize{$\ell$}}};
	\node at (.4, .6) {{\scriptsize{$\err$}}};
	\nbox{unshaded}{(0,0)}{.4}{0}{0}{$y$}
	\draw[fill=red] (0,.4) circle (.05cm);
\end{tikzpicture} = \sum_{\ell \, + \, \err\, = n-2} 
\begin{tikzpicture}[baseline=.1cm]
	\nbox{unshaded}{(.5,.1)}{.6}{.4}{.6}{}
	\draw (.9,0)--(.9,1.1);
	\draw (-.8,0) -- (2,0);
	\draw (1.3,0) -- (1.3,1.1);
	\draw (.1,.3) arc (180:0:.3cm);
	\node at (.7,.9) {\scriptsize{$\ell$}};
	\node at (1.45,.9) {\scriptsize{$\err$}};
	\nbox{unshaded}{(0,0)}{.3}{0}{0}{$x$}
	\nbox{unshaded}{(.9,0)}{.3}{0}{.2}{$y$}
	\draw[fill=red] (1.1,.7) circle (.05cm);
\end{tikzpicture}
+ \sum_{\ell \, + \, \err \, = n}
\begin{tikzpicture}[baseline = .1cm]
	\draw (-.8, 0)--(.8, 0);
	\draw (-.2, 0)--(-.2, .8);
	\draw (.2, 0)--(.2, .8);
	\node at (-.4, .6) {{\scriptsize{$\ell$}}};
	\node at (.4, .6) {{\scriptsize{$\err$}}};
	\nbox{unshaded}{(0,0)}{.4}{0}{0}{$x$}
	\draw[fill=red] (0,.4) circle (.05cm);
\end{tikzpicture}\,.
$$
By the inductive hypothesis, the left hand side and the first summation on the right hand side is in $\SP{}$.  Therefore, all terms are in $\SP{}$.
\end{proof}

It is important to note that this proposition proves the following:
\begin{prop}\label{prop:cyclic2}
$\SP{}$ is generated by $\cB$ and $\set{L_{+}(x) + L_{-}(x)}{x \in \cX}$.
\end{prop}

\begin{rem}\label{rem:GJS}
Proposition \ref{prop:cyclic} shows that we may regard $\SP{}$ as generated by operators $x \in \cP_{l, n, r}$.  The action of $x$ on $y \in \FP{}$ is given by:
$$
x \cdot 
\left(
\begin{tikzpicture}[baseline = .1cm]
	\draw (-.8, 0)--(.8, 0);
	\draw (0, 0)--(0, .8);
    \node at (-.2, .6) {\scriptsize{$m$}};
	\nbox{unshaded}{(0,0)}{.4}{0}{0}{$\xi$}
\end{tikzpicture}
\right) = \sum_{k=0}^{\min\{n, m\}}\, \begin{tikzpicture} [baseline = .1cm]
\draw (-.8, 0)--(2, 0);
\draw (-.2, 0)--(-.2, .8);
\draw (.2,.4) arc(180:0:.4cm);
\draw (1.4, 0)--(1.4, .8);
\node at (.6, .6) {\scriptsize{$k$}};
\nbox{unshaded}{(0,0)}{.4}{0}{0}{$x$}
\nbox{unshaded}{(1.2,0)}{.4}{0}{0}{$y$}
\end{tikzpicture}
$$
This diagram also shows how to multiply $x$ and $y$ when they are viewed in the algebra $\SP{}$.  
This is exactly the product that appears in \cite{MR2645882}.   
We will show $\SP{}$ is isomorphic to the semifinite GJS algebra in \cite[Lemma 3.3]{GJSCStar}.
\end{rem}


\subsection{The Cuntz-Pimsner algebra}\label{sec:OperatorCuntz}

As $\TP{}$ is isomorphic the the Pimsner-Toeplitz algebra acting on $\FP{}$, it contains $\KP{} = \cK(\FP{})$.  
This allows us to define our Cuntz-Pimsner algebra over $\cX$.

\begin{defn}\label{defn:XCuntz}
The Cuntz-Pimnser algebra $\OP{}$ is the $C^{*}$-algebra $\TP{}/\KP{}$.
\end{defn}

\begin{defn}\label{defn:Xcore}
From Definition \ref{defn:GaugeAction}, there is a canonical $\bbR$-action $\sigma: \bbR\to \Aut(\TP{})$ given by the extension of
$$
\sigma_{t}\left( 
\begin{tikzpicture}[baseline = .1cm]
	\draw (-.8, 0)--(.8, 0);
	\draw (-.2, 0)--(-.2, .8);
	\draw (.2, 0)--(.2, .8);
	\node at (-.4, .6) {{\scriptsize{$\ell$}}};
	\node at (.4, .6) {{\scriptsize{$\err$}}};
	\nbox{unshaded}{(0,0)}{.4}{0}{0}{$x$}
	\draw[fill=red] (0,.4) circle (.05cm);
\end{tikzpicture}
\right)
=
e^{i(\ell-\err)t} 
\begin{tikzpicture}[baseline = .1cm]
	\draw (-.8, 0)--(.8, 0);
	\draw (-.2, 0)--(-.2, .8);
	\draw (.2, 0)--(.2, .8);
	\node at (-.4, .6) {{\scriptsize{$\ell$}}};
	\node at (.4, .6) {{\scriptsize{$\err$}}};
	\nbox{unshaded}{(0,0)}{.4}{0}{0}{$x$}
	\draw[fill=red] (0,.4) circle (.05cm);
\end{tikzpicture}\, ,
$$  
which again induces a $\bbT$-action.
We define the \underline{core of $\TP{}$} as the fixed points under the $\bbT$-action, i.e., $\TP{}^\bbT$.
Note that just as in Subsection \ref{sec:PAToeplitz}, $\sigma_{t}$ is norm continuous in $t$, so we define the conditional expectation $E: \TP{} \rightarrow \coreTP{}$ by
$$
E(x) = \int_{\T} \gamma_{z}(x)dz
$$
where the measure is normalized Lebesgue measure.
\end{defn}

We have a $\cB$-valued expectation $\phi$ on $\coreTP{}$ given by
$$
\phi(L_n(x))
=
\phi\left(
\begin{tikzpicture}[baseline = .1cm]
	\draw (-.8, 0)--(.8, 0);
	\draw (-.2, 0)--(-.2, .8);
	\draw (.2, 0)--(.2, .8);
	\node at (-.4, .6) {{\scriptsize{$n$}}};
	\node at (.4, .6) {{\scriptsize{$n$}}};
	\nbox{unshaded}{(0,0)}{.4}{0}{0}{$x$}
	\draw[fill=red] (0,.4) circle (.05cm);
\end{tikzpicture}
\right) 
= \delta^{-n}\cdot 
\begin{tikzpicture}[baseline = .1cm]
	\draw (-.8, 0)--(.8, 0);
	\draw (-.2, 0)--(-.2, .4) arc(180:0:.2cm);
	\node at (0, .8) {{\scriptsize{$n$}}};
	\nbox{unshaded}{(0,0)}{.4}{0}{0}{$x$}
\end{tikzpicture}\,.
$$
Note that $\phi$ extends continuously to $\coreTP{}$ and is faithful on elements of the form $\phi(L_n(x))$ as above.
We therefore obtain a $\cB$-valued conditional expectation $\psi = \phi\circ E$ defined on $\TP{}$.
 
The map $\psi$ induces the weight $\Psi = \Tr\circ \psi$ which is defined on $\TP{}^{+}$ and takes values in $[0, \infty]$.  
The techniques of Subsections \ref{sec:KMS} and  \ref{sec:PACuntz} can be used to show that $\Phi$ is KMS with inverse temperature $\ln(\delta)$, and the ideal generated by
$$
\set{x \in \TP{}^{+}}{\Psi(x) = 0}
$$
is $\cK(\cP_{\bullet})$.  
Therefore, $\Psi$ drops to a faithful KMS weight on $\OP{}$.

\subsection{Functoriality of the construction}\label{sec:Functorial}

In this subsection, we will show that if $\cP_{\bullet}$ and $\cQ_{\bullet}$ are factor planar algebras, and $i: \cQ_{\bullet} \rightarrow \cP_{\bullet}$ is a planar algebra homomorphism, then $i$ indues a map between the corresponding Pimnser-Toeplitz, Cuntz-Pimsner, and free semicircular algebras.  
Note that any such $i$ must be injective due the fact that the inner product tangle is positive definite, and $\cP_{0} \cong \C \cong \cQ_{0}$.

To begin, recall that
\begin{align*}
\cB_{n}(\cP_\bullet) &=  \bigoplus_{l, r = 1}^{n} \cP_{l, r}
\text{ and } 
\cB(\cP_\bullet) = \varinjlim \cB_n(\cP_\bullet)\\
\cB_{n}(\cQ_\bullet) &= \bigoplus_{l, r = 1}^{n} \cQ_{l, r} 
\text{ and }
\cB(\cQ_\bullet) = \varinjlim \cB_n(\cQ_\bullet)
\end{align*}
with $C^*$-algebra structure as described in Subsection \ref{sec:FockOp}.
Note that $i$ induces a $C^{*}$-algebra inclusion $\cB_{n}(\cQ_\bullet) \hookrightarrow \cB_{n}(\cP_\bullet)$ and hence an inclusion $i_\cB: \cB(\cQ_\bullet)\hookrightarrow \cB(\cP_\bullet)$.   
Recall that 
$$
\cX(\cP_\bullet) = \overline{ \bigoplus_{l, r = 1}^{\I} \cP_{l ,1, r}}^{\|\cdot\|} 
\text{ and }
\cX(\cQ_\bullet) = \overline{ \bigoplus_{l, r = 1}^{\I} \cQ_{l, 1, r}}^{\|\cdot\|} 
$$
with $\cB(\cP_\bullet)-\cB(\cP_\bullet)$ (respectively $\cB(\cQ_\bullet)-\cB(\cQ_\bullet)$) bimodule structure as described in Subsection \ref{sec:FockOp}.
Furthermore we see that the map $i$ induces a bimodule map from $i_\cX: \cX(\cQ_\bullet)\to \cX(\cP_\bullet)$ which satisfies $i_\cX(b_{1}\cdot\xi\cdot b_{2}) \mapsto i_\cB(b_{1})\cdot i_\cX(\xi) \cdot i_\cB(b_{2})$.

We are now ready for the main theorem of this subsection.

\begin{thm}\label{thm:functorial}
The assignment $\cP_{\bullet}$ to $\TP{}$, $\SP{}$, and $\OP{}$ is functorial in the following sense.
If $i: \cQ_{\bullet} \rightarrow \cP_{\bullet}$ is an inclusion of factor planar algebras, then $i$ induces an injective $C^{*}$-algebra homomorphism $i_{\cA}: \cA(\cQ_{\bullet}) \rightarrow \cA(\cP_{\bullet})$ with $\cA \in \{\cT, \cO, \cS\}$. 
The map $i_{\cT}$ satisfies and is uniquely determined by
\begin{itemize}
\item 
$i_{\cT}(b) = i_\cB(b)$ for $b \in \cB(\cQ_\bullet)$ and

\item 
$i_{\cT}(L_{+}(\xi)) = L_{+}(i_\cX(\xi))$.

\end{itemize}
\end{thm}

\begin{proof}
We consider the $C^*$-subalgebra of $\TP{}$ which is generated by $i_\cB(\cB(\cQ_\bullet))$ and the elements $L_{+}(\xi)$ for $\xi \in i_\cX(\cX(\cQ_\bullet))$.  
We note that the following relations hold:
\be
\item 
$a_{1}L_{+}(i_\cX(\xi_{1})) + a_{2}L_{+}(i_\cX(\xi_{2})) = L_{+}(a_{1}i_\cX(\xi_{1}) + a_{2}i_\cX(\xi_{2}))$ for $a_{i} \in \C$ and $\xi_{i} \in \cX(\cQ_\bullet)$.

\item 
$i_\cB(a) \cdot L_{+}(i_\cX(\xi))\cdot i_\cB(b) = L_{+}(i_\cB(a) \cdot \xi \cdot i_\cB(b))$ for $\xi \in \cX(\cQ_{\bullet})$ and $a,b \in \cB(\cQ_\bullet)$.

\item 
$L_{+}(i_\cX(\eta))^{*}L_{+}(i_\cX(\xi)) = \langle i_\cX(\eta) | i_\cX(\xi) \rangle_{\cB(\cP_\bullet)} = i_\cB(\langle \eta | \xi \rangle_{\cB(\cQ_\bullet)})$.

\ee
The last equality holds since the $\cB(\cP_\bullet)$ and $\cB(\cQ_\bullet)$ valued inner products are induced from the planar operad.   
Since $i_\cB : \cB(\cQ_\bullet) \rightarrow \cB(\cP_\bullet)$ is an injection, the universality of Pimsner-Toeplitz algebras \cite{MR1426840} implies that $i$ induces a $C^{*}$-algebra homomorphism $i_{\cT}: \cT(\cQ_{\bullet}) \rightarrow \TP{}$.  
The map $i_{\cT}$ satisfies and is uniquely determined by the equations claimed above.

We now show that $i_{\cT}$ is injective.  
Let $\cX_{\cQ_{\bullet}}$ be the closure of the subspace of $\FP{}$ spanned by boxes in $i(\cQ_{\bullet})$, and note that $\cX_{\cQ_{\bullet}}$ is naturally an $i_\cB(\cB(\cQ_{\bullet}))-i_\cB(\cB(\cQ_{\bullet}))$ Hilbert bimodule.  
We define a map $U : \cF(\cQ_{\bullet}) \rightarrow \cX_{\cQ_{\bullet}}$ which is initially defined  on diagrams and is given by linear extension of the formula
$$
U( \xi) = i(\xi) \text{ for } \xi \in \cQ_{l, n, r}.
$$ 
We see that $U(b_{1} \cdot \xi \cdot b_{2}) = i_\cB(b_{1})\cdot U(\xi)\cdot i_\cB(b_{2})$, and $\langle U(\eta)| U(\xi) \rangle_{\cB(\cP_\bullet)} = i_\cB(\langle \eta| \xi\rangle_{\cB(\cQ_\bullet)}),$ and by construction, $U$ has dense range.  
Therefore, $U$ extends to be a unitary satisfying
$$
UL_{+}(\xi)U^{*} = L_{+}(i_\cX(\xi)).
$$
This shows that $U$ intertwines the action of $\cT(\cQ_\bullet)$ on $\FP{}$ and the action of $i_{\cT}(\cT(\cQ_\bullet))$ on $\cX_{\cQ_{\bullet}}$.  
Therefore $i_{\cT}$ is injective.

The map $i_{\cT}$ sends $\cS(\cQ_{\bullet})$ into $\SP{}$, so the assignment $\cP_{\bullet}$ to $\SP{}$ is functorial as well.
Finally, note that if $x \in \cK(\cQ_{\bullet})$ then $i_{\cT}(x) \in \cK(\cP_{\bullet})$.  
This shows that $i_{\cT}$ induces a map $\cO(\cQ_{\bullet}) \rightarrow  \OP{}$.  Theorem \ref{thm:CuntzGraph} and Lemma \ref{lem:CompactTensor1} below show that $\cO(\cQ_{\bullet})$ is simple, implying that the map is injective.
\end{proof}

\begin{ex}
A common example of an inclusion of planar algebras is the inclusion of the Temperley-Lieb planar subalgebra $\TL_\bullet \hookrightarrow \cP_\bullet$.
\end{ex}

\begin{ex}
Another common example arises from equivariantization, i.e., when $\cQ_\bullet$ is the fixed points of $\cP_\bullet$ under the action of some finite group $G$ of planar algebra automorphisms.
In this case, one can show that each $g\in G$ induces an automorphism of $\cA(\cP_\bullet)$, and $\cA(\cP_\bullet)^G \cong \cA(\cQ_\bullet)$ for $\cA\in \{\cT,\cO,\cS\}$.
\end{ex}

\begin{rem}
Theorem \ref{thm:OSfunctor} below will show that the assignments $\cP_{\bullet}$ to $\OP{0}$ and $\SP{0}$ are functorial.
\end{rem}


\section{Compressions of the operator-valued system}\label{sec:Compression}

In this section, we investigate a compression of $\cB$ which induces a compression of $\cX$, leading to subsequent compressions of $\FP{}$, $\TP{}$, $\OP{}$, and $\SP{}$.
While the compressing projections from $\Gamma$ are not in general canonical, the resulting compressions are isomorphic to canonical objects associated to an oriented graph $\vec{\Gamma}$ obtained from $\Gamma$.

First, we get a spatial isomorphism from the compression of $\cX$ to the Cuntz-Krieger bimodule of $\vec{\Gamma}$ from \cite[Example 1.2]{MR1722197}.
This spatial isomorphism implements isomorphisms between
\begin{itemize}
\item
the compression of $\FP{}$ and the Fock space of the Cuntz-Krieger bimodule \cite[Example 1.4]{MR1722197},
\item
the compression of $\TP{}$ and the universal Toepltiz-Cuntz-Krieger algebra $\cT_{\vec{\Gamma}}$ \cite{MR1722197},
\item
the compression of $\OP{}$ and the Cuntz-Krieger algebra $\cO_{\vec{\Gamma}}$ \cite{MR561974}, and
\item
the compression of $\SP{}$ and the free graph algebra $\cS(\Gamma)$.
\end{itemize}
The free graph algebra $\cS(\Gamma)$ is analogous to the free graph von Neumann algebras associated to arbitrary unoriented weighted graphs appearing in \cite{MR2807103, MR3016810, MR3110503}.
In particular, in Subsection \ref{sec:FreeGraphAlgebraOfUnoriented}, we discuss the free semicircular graph algebra $\cS(\Lambda,\mu)$ associated to an arbitrary unoriented weighted graph $(\Lambda,\mu)$ (which is not necessarily a principal graph) and how it sits inside the (Toeplitz-)Cuntz-Krieger algebra.

In Subsection \ref{sec:FurtherCompression}, we discuss a further compression by $p_\star$ which recovers the GJSW-Doplicher-Roberts system of Section \ref{sec:PAFockSpace}.
Finally, in Subsection \ref{sec:ShadedCase}, we briefly sketch the case of a shaded subfactor planar algebra.

\subsection{Compressing $\cB$ and $\cX$}

For each $\alpha \in V(\Gamma)$, choose a representative $p_{\alpha} \in \cP_{2\depth(\alpha)}\subset 1_{\depth(\alpha)} \cB 1_{\depth(\alpha)}$.
Set $P_{n} = \sum_{\alpha \in V(\Gamma(n))} p_{\alpha}$, and notice that $P_{n}\cB P_{n}$ is naturally a subspace of $P_{n+1}\cB P_{n+1}$.
In fact, $P_n \cB P_n \cong C(V(\Gamma(n)))$, and the inclusion $P_n \cB P_n \hookrightarrow P_{n+1} \cB P_{n+1}$ is the natural inclusion $C(V(\Gamma(n)))\hookrightarrow C(V(\Gamma(n+1)))$.
Let
$$
\cC=\cC(\cP_\bullet)=\varinjlim P_n \cB P_n \cong C_0(V(\Gamma)),
$$
which is our compression of $\cB$.

Note that $\cC$ induces a compression of $\cX$ as follows.
Let
$$
\cY=\cY(\cP_\bullet)=\overline{\cC\cX\cC}^{\|\cdot\|_\cX}=\varinjlim P_n \cX P_n.
$$
Moreover, $\cY$ carries a $\cB$-valued inner product, which is actually just a $\cC$-valued inner product.
Hence $\cY$ is a Hilbert $\cC-\cC$ bimodule.

\begin{rem}
If $\Gamma$ is simply laced and acyclic, then there is a unique choice for each $p_\alpha$, and $\cY$ is canonical.

If $\Gamma$ is not simply laced or not acyclic, it is clear from the above definition that if we chose different representatives $p_\alpha'$ for the vertices $\alpha$, the resulting projections $P_n'$ are equivalent in $\cB$ to the $P_n$.
Such a partial isometry implementing the equivalence is unique up to a choice of phase for each $\alpha$, since minimal projections in a matrix algebra are equivalent by a unique partial isometry up to a phase.
\end{rem}

We now identify $\cY$ with the Cuntz-Krieger bimodule $Y(\vec{\Gamma})$ from Definition \ref{defn:CuntzKriegerBimodule}.

\begin{defn}
For the principal graph $\Gamma$ of $\cP_{\bullet}$, let $\vec{\Gamma}$ be as in Definition \ref{defn:Directed}.
Let
$$
E(\alpha\to \beta)=\set{\epsilon\in E(\vec{\Gamma})}{s(\epsilon)=\alpha\text{ and }t(\epsilon)=\beta}.
$$
\end{defn}

We note that $\cY$ is spanned by elements of the form $p_{\alpha}xp_{\beta}$ for $x \in \cX$.
Diagrammatically, these elements look like
$$
\begin{tikzpicture}[baseline = -.1cm]
	\draw (-.6, 0)--(2.6, 0);
	\draw (1, 0)--(1, .6);
	\nbox{unshaded}{(0,0)}{.3}{0}{0}{$p_{\alpha}$}
	\nbox{unshaded}{(1,0)}{.3}{0}{0}{$x$}
	\nbox{unshaded}{(2,0)}{.3}{0}{0}{$p_{\beta}$}
	\node at (-.5, .2) {{\scriptsize{$l$}}};	
	\node at (.5, .2) {{\scriptsize{$l$}}};	
	\node at (1.5,.2) {{\scriptsize{$r$}}};
	\node at (2.5,.2) {{\scriptsize{$r$}}};
\end{tikzpicture}\,
$$
where there is only one strand on top. 
We note that $p_{\alpha}\cY p_{\beta}$ is finite dimensional, and $\dim(p_{\alpha}\cY p_{\beta}) = \left|E(\alpha \to \beta)\right|$.
This observation will help us define our edge elements.

\begin{defn}[Edge elements]
For each edge $\e \in E(\vec{\Gamma})$, we define an \underline{edge element} $g_{\e}$.
\be
\item
Choose an orthonormal basis $\set{g_{\e}}{\epsilon\in E(\alpha\to \alpha)}$ of $p_{\alpha}\cY p_{\alpha}$ under the inner product
$$
\langle x|y \rangle = \,
\begin{tikzpicture}[baseline = -.1cm]
	\draw (-.3, 0) arc(90:270:.3cm)-- (1.3, -.6) arc(-90:90:.3cm)--(-.3, 0);
	\draw (0, 0.3) arc (180:90:.3cm) --(.7, .6) arc(90:0:.3cm);
	\nbox{unshaded}{(0,0)}{.3}{0}{0}{$x^{\dagger}$}
	\nbox{unshaded}{(1,0)}{.3}{0}{0}{$y$}
\end{tikzpicture}\,,
$$
where each $g_{\e}$ is self-adjoint.

\item
Suppose $\alpha \neq \beta$.
Choose an orthonormal basis $\set{g_{\e}}{\epsilon \in E(\alpha \to \beta)}$ of $p_{\alpha}\cY p_{\beta}$.
We note that $\set{g_{\e}^{\dagger}}{\epsilon\in E(\alpha\to \beta)}$ gives an orthonormal basis of $p_{\beta}\cY p_{\alpha}$.
Thus, for $\e\in E(\beta \rightarrow \alpha)$, we define $g_{\e} = g_{\e^{\op}}^{\dagger}$.
Therefore, if $\e$ is not a loop, then $g_{\e} \in p_{s(\e)}\cY p_{t(\e)}$ and $g_{\e}^{\dagger} = g_{\e^{\op}}$.

\ee
\end{defn}

\begin{rem}
If $\epsilon \in E(\alpha \to \beta)$, we visualize the edge element as
$$
g_\epsilon
=
\begin{tikzpicture}[baseline = -.1cm]
	\draw (-.6, 0)--(2.6, 0);
	\draw (1, 0)--(1, .6);
	\nbox{unshaded}{(0,0)}{.3}{0}{0}{$p_{\alpha}$}
	\nbox{unshaded}{(1,0)}{.3}{0}{0}{$g_\epsilon$}
	\nbox{unshaded}{(2,0)}{.3}{0}{0}{$p_{\beta}$}
	\node at (-.5, .2) {{\scriptsize{$l$}}};	
	\node at (.5, .2) {{\scriptsize{$l$}}};	
	\node at (1.5,.2) {{\scriptsize{$r$}}};
	\node at (2.5,.2) {{\scriptsize{$r$}}};
\end{tikzpicture}
$$
with the source projection on the left and the target projection on the right.
\end{rem}

We now identify $\cY$ with $Y(\vec{\Gamma})$ from Definition \ref{defn:CuntzKriegerBimodule}.
We need the following lemma.

\begin{lem}\label{lem:edgeinnerproduct}
\mbox{}
\be
\item If $\e, \e' \in E(\vec{\Gamma})$, then $\langle g_{\e'}|g_{\e} \rangle_{\cC} = \displaystyle \frac{\delta_{\epsilon,\epsilon'}}{\Tr(p_{t(\e)})}\cdot p_{t(\e)}$.

\item $p_{\alpha}g_{\e} = \delta_{s(\e) = \alpha}g_{\e}$ and $g_{\e}p_{\alpha} = \delta_{t(\e) = \alpha}g_{\e}$.

\ee
\end{lem}

\begin{proof}
\mbox{}
\be
\item
First, we note that the $\cC$-valued inner product must be zero unless $s(\e') = s(\e)$.
Let $\alpha = t(\e') = s((\e')^{\op})$ and $\beta = t(\e)$.
Then we have
$$
\langle g_{\e'}|g_{\e} \rangle_{\cC}
=
\begin{tikzpicture}[baseline = -.1cm]
	\draw (-.6, 0)--(1.6, 0);
	\draw (0, 0.3) arc (180:90:.3cm) --(.7, .6) arc(90:0:.3cm);
	\nbox{unshaded}{(0,0)}{.3}{0}{0}{$g_{\e'}^{\dagger}$}
	\nbox{unshaded}{(1,0)}{.3}{0}{0}{$g_{\e}$}
	\node at (.5, .2) {{\scriptsize{$l$}}};	
	\node at (-.5,.2) {{\scriptsize{$r'$}}};
	\node at (1.5,.2) {{\scriptsize{$r$}}};
\end{tikzpicture}
\in p_\alpha \cC p_\beta
=\begin{cases}
(0) &\text{if }\alpha \neq \beta \\
\spann\{p_\alpha\} & \text{if }\alpha=\beta.
\end{cases}
$$
Assuming $\alpha=\beta$, by the orthonormality of the edge elements in $p_{s(\epsilon)} \cY p_\alpha$,
we observe
$$
\begin{tikzpicture}[baseline = -.1cm]
	\draw (-.3, 0) arc(90:270:.3cm)-- (1.3, -.6) arc(-90:90:.3cm)--(-.3, 0);
	\draw (0, 0.3) arc (180:90:.3cm) --(.7, .6) arc(90:0:.3cm);
	\nbox{unshaded}{(0,0)}{.3}{0}{0}{$g_{\e'}^{\dagger}$}
	\nbox{unshaded}{(1,0)}{.3}{0}{0}{$g_{\e}$}
	\node at (.5, -.4) {{\scriptsize{$r$}}};
	\node at (.5, .2) {{\scriptsize{$l$}}};	
\end{tikzpicture}\,
= \delta_{\epsilon,\epsilon'},
$$
and the result follows.

\item
This is immediate from the definition of $g_{\e}$.
\qedhere
\ee
\end{proof}

\begin{prop}\label{prop:CKBim}
There is a spatial isomorphism $\cY \cong Y(\vec{\Gamma})$ of Hilbert $\cC-\cC$ bimodules.
\end{prop}
\begin{proof}
We construct a unitary $U: Y(\vec{\Gamma}) \rightarrow \cY$.
We initially define $U$ on finitely supported functions by the formula $U(\delta_{\e}) = \sqrt{\Tr(p_{t(\e)})}\cdot g_{\e}$.
Lemma \ref{lem:edgeinnerproduct} shows that $U$ preserves the $\cC$-valued inner products, the $\cC-\cC$ bimodule structure, and has dense range.
Therefore, $U$ can be extended to $Y(\vec{\Gamma})$, giving the isomorphism.
\end{proof}

\subsection{Induced isomorphisms}\label{sec:Induced}

\begin{defn}\label{defn:CompressedFock}
Define $\cF(\Gamma) = \overline{\cC\FP{}\cC}^{\|\cdot\|_{\FP{}}}$ and $\cY_{m} =  \overline{\cC\cX_{m}\cC}^{\|\cdot\|_{\cX_m}}$. 
Using the convention that $\cY = \cY_{1}$, we have $\cF(\Gamma) = \bigoplus_{m=0}^{\infty} \cY_{m}$.
\end{defn}

We now show $\cF(\Gamma)$ (as a $\cC-\cC$ bimodule) retains the Fock space structure of $\FP{}$.

\begin{prop}\label{prop:CompressedFock}
$\cY_{n} \cong \bigotimes_{\cC}^{n}\cY (\cong \bigotimes_{\cC}^{n}Y(\vec{\Gamma}))$, so $\cF(\Gamma) \cong \cF(Y(\vec{\Gamma}))$.
\end{prop}
\begin{proof}
We show $\cY_{2} \cong \cY \otimes_{\cC} \cY$, and the proof of the general statement follows by induction.
Let $x \in p_{\alpha}\cY_{2} p_{\beta}$.
By Proposition \ref{prop:OperatorValuedFock} and the discussion thereafter, we see that
$$
x =
\begin{tikzpicture}[baseline = -.1cm]
	\draw (-.6, 0)--(1.6, 0);
	\draw (0, 0)--(0, .6);
	\draw (1, 0)--(1, .6);
	\nbox{unshaded}{(0,0)}{.3}{0}{0}{$y$}
	\nbox{unshaded}{(1,0)}{.3}{0}{0}{$z$}
	\node at (1.5, .2) {{\scriptsize{$r$}}};
	\node at (0.5, .2) {\scriptsize{$n$}};
	\node at (-.5, .2) {{\scriptsize{$l$}}};	
\end{tikzpicture}
$$
for some (non-unique) $y,z\in \cX$.
There is some projection $P\in \cP_{2n}$ such that
$$
x =
\begin{tikzpicture}[baseline = -.1cm]
	\draw (-.6, 0)--(2.6, 0);
	\draw (0, 0)--(0, .6);
	\draw (2, 0)--(2, .6);
	\nbox{unshaded}{(0,0)}{.3}{0}{0}{$y$}
	\nbox{unshaded}{(1,0)}{.3}{0}{0}{$P$}
	\nbox{unshaded}{(2,0)}{.3}{0}{0}{$z$}
	\node at (2.5, .2) {{\scriptsize{$r$}}};
	\node at (1.5, .2) {\scriptsize{$n$}};
	\node at (0.5, .2) {\scriptsize{$n$}};
	\node at (-.5, .2) {{\scriptsize{$l$}}};	
\end{tikzpicture}
$$
Now $P$ is a sum of projections $P_{\gamma}\in \cP_{2n}$ where the $\gamma$'s are vertices connected to both $\alpha$ and $\beta$.
Thus each $P_\gamma$ is a sum of minimal projections equivalent to $p_\gamma$ in $\cB$.
For each $\gamma$, choose partial isometries $v^{\gamma}_{1}, ..., v^{\gamma}_{N(\gamma)}$ such that
$$
(v^{\gamma}_{j})^{*}v^{\gamma}_{j} = p_{\gamma} \text{ and }
\sum_{j=1}^{N(\gamma)}v^{\gamma}_{j}(v^{\gamma}_{j})^{*}
\left(=
\sum_{j=1}^{N(\gamma)}v^{\gamma}_{j}p_{\gamma}(v^{\gamma}_{j})^{*}
\right)
=
P_{\gamma}.
$$
Inserting the latter expression in for each $P_{\gamma}$ gives the result.
\end{proof}

We now define our compressed Toeplitz algebra.

\begin{defn}
We define the Toeplitz algebra of $\Gamma$ by $\cT(\Gamma) = \overline{\cC \TP{} \cC}^{\|\cdot\|}$.  
Notice that these are operators on $\cF(\Gamma)$.
\end{defn}

The following proposition, whose proof directly follows from Propositions \ref{prop:CKBim} and \ref{prop:CompressedFock}, shows that the definition of $\cT(\Gamma)$ makes sense.

\begin{prop}\label{prop:CompressedToeplitz1}
$\cT(\Gamma)$ is generated as a $C^{*}$-algebra by $\cC$ and $\set{L_{+}(y)}{y \in \cY}$.
\end{prop}

This proposition, together with Proposition \ref{prop:CKBim}, shows that $\cT(\Gamma) \cong \cT(Y(\vec{\Gamma}))$.  Since the edge elements generate $\cY$ as a Banach space, the following proposition immediately follows from Proposition \ref{prop:CompressedToeplitz1}.

\begin{prop}\label{prop:CompressedToeplitz2}
$\cT(\Gamma)$ is generated as a $C^{*}$ algebra by $\cC$ and $\set{L_{+}(g_{\epsilon})}{\epsilon \in E(\vec{\Gamma})}$.
\end{prop}

We will now explore the Toeplitz-structure of $\cT(\Gamma)$.

\begin{defn}
For $\e \in E(\vec{\Gamma})$, set $S_{\e} = \sqrt{\Tr(p_{t(\epsilon)})}L_{+}(g_{\epsilon})$.
\end{defn}

\begin{prop}\label{prop:EdgeRelations}
\mbox{}
\be
\item
If $\e,\e'\in E(\vec{\Gamma})$, then $S_{\e'}^{*}S_{\e} = \delta_{\epsilon,\epsilon'}p_{t(\epsilon)}$.

\item
$\D
\sum_{s(\e) = \alpha} S_{\e}S_{\e}^{*} = p_{\alpha} - | p_{\alpha}\rangle\langle p_{\alpha} |.
$
\ee
\end{prop}

\begin{proof}
Using the identity $L_{+}(y)^{*}L_{+}(x) = \langle y|x \rangle_{\cC}$, (1) follows from (1) of Lemma \ref{lem:edgeinnerproduct}.
For part (2), we first note that if $\xi \in \cC$, then applying either side of the equation to $\xi$ produces zero.  We may now assume that
$$
\xi =
\begin{tikzpicture}[baseline = -.1cm]
	\draw (-.6, 0)--(1.6, 0);
	\draw (0, 0)--(0, .6);
    \draw (1, 0)--(1, .6);
	\nbox{unshaded}{(0,0)}{.3}{0}{0}{$g_{\e'}$}
	\nbox{unshaded}{(1,0)}{.3}{0}{0}{$\eta$}
\end{tikzpicture}\, .
$$
If $\e'$ is not any edge represented in the sum, then $p_{\alpha}g_{\e'} = 0$ and part (1) implies that the sum annihilates $g_{\e'}$ as well.  If $\e'$ is an edge represented in the sum, then by parts (1) and (2) and the proof of Lemma \ref{lem:edgeinnerproduct},
$$
\left(\sum_{s(\e)
=
\alpha}S_{\e}S_{\e}^{*}\right)\cdot \xi
=
S_{\e'}S_{\e'}\xi
=
\begin{tikzpicture}[baseline = -.1cm]
	\draw (-.6, 0)--(3.2, 0);
	\draw (0, 0.3) arc (180:90:.3cm) --(.7, .6) arc(90:0:.3cm);
	\nbox{unshaded}{(0,0)}{.3}{0}{0}{$g_{\e'}$}
	\nbox{unshaded}{(1,0)}{.3}{0}{.6}{$p_{t(\e')}$}
	\nbox{unshaded}{(2.6,0)}{.3}{0}{0}{$\eta$}
\end{tikzpicture}\, = \xi.
$$
We also note that $p_{\alpha}\xi = \xi$ and $p_{\alpha}\langle p_{\alpha} | \xi \rangle = 0$.  This completes the proof.
\end{proof}

Note that $\cT(\Gamma)$ is isomorphic to the universal Toeplitz-Cuntz-Krieger $C^*$-algebra $\cT_{\vec{\Gamma}}$ \cite{MR1722197}.
Let $\cK(\Gamma) = \cK(\cF(\Gamma))$ be the $C^{*}$ algebra generated by the ``rank one" operators $| \xi \rangle\langle \eta |$.  
As in Section \ref{sec:OperatorCuntz}, we see that $\cT(\Gamma)$ contains $\cK(\Gamma)$ as an ideal.  
We define $\cO(\Gamma) = \cT(\Gamma)/\cK(\Gamma)$.  
We have the following theorem:

\begin{thm}\label{thm:CuntzGraph}
$\cO(\Gamma)$ is isomorphic to the Cuntz-Krieger graph $C^{*}$-algebra $\cO_{\vec{\Gamma}}$.
\end{thm}

\begin{proof}
We note that $\cO(\Gamma)$ is generated by the homomorphic images of $\cC$ and the elements $S_{\e}$ for $\e \in E(\vec{\Gamma})$, and we will abuse notation by calling the images $\cC$ and $S_{\e}$ respectively.  No nonzero element of $\cC$ lies in $\cK(\Gamma)$, so $\cC$ is faithfully represented inside $\cO(\Gamma)$.  Also note that by Proposition \ref{prop:EdgeRelations},
$$
S_{\e}^{*}S_{\e} = p_{t(\e)} \text{ and } \sum_{s(\e) = \alpha} S_{\e}S_{\e}^{*} = p_{\alpha}
$$
in $\cO(\Gamma)$.  By \cite[Theorem 2.13]{MR561974} and \cite[Theorem 2.1.8]{MR1409796}, we have $\cO(\Gamma) \cong \cO_{\vec{\Gamma}}$.
\end{proof}

\begin{rem}\label{rem:OP}
We will observe in Theorem \ref{thm:OSfunctor} below that  $\OP{0}$ is naturally isomorphic to  $p_{\star}\cO(\Gamma)p_{\star}$.
As a result, we may now infer many properties of $\OP{0}$ since it is stably isomorphic to $\cO_{\vec{\Gamma}}$.
Hence Facts \ref{facts:CuntzKrieger} all hold for $\OP{0}$, so in particular, $\OP{0}$ is simple.
\end{rem}

\subsection{The free graph algebra of an unoriented graph}\label{sec:FreeGraphAlgebraOfUnoriented}

This subsection defines the third canonical algebra associated to the compressed Fock space.
We show this algebra appears as a subalgebra of Cuntz-Krieger graph algebras of certain directed graphs provided certain natural weighting conditions are satisfied.  
To begin, we provide some useful notation for the remainder of the article.

\begin{nota}
Let $\cA$ be a $C^{*}$-algebra equipped with a lower semicontinuous tracial weight $\tr$ (which may or may not be finite).  We write
$$
\cA = \overset{p_{1}}{\underset{\mu_{1}}{\cA_{1}}} \oplus \cdots \oplus \overset{p_{n}}{\underset{\mu_{n}}{\cA_{n}}} \oplus \cD
$$
if $p_{k}$ is a projection in $\cA$ which serves as the identity of the $C^{*}$-algebra $\cA_{k}$ and $\tr(p_{k}) = \mu_{k}$.  
The algebra $\cD$ may be unital or not.  
\end{nota}

We start with a connected undirected graph $\Lambda$ with weighting $\mu : V(\Lambda)\to \bbR_{>0}$.
Let $\vec{\Lambda}$ be as in Definition \ref{defn:Directed}.
Set $\cC_{\Lambda} = C_{0}(V(\Lambda)) = C_{0}(V(\vec{\Lambda}))$.
We form the $\cC_{\Lambda}-\cC_{\Lambda}$ Hilbert bimodule $Y(\vec{\Lambda})$ and its Fock space $\cF(Y(\vec{\Lambda}))$.
For notational convenience, we will rewrite tensors of the form $\delta_{\e_{1}} \otimes \cdots \otimes\delta_{\e_{n}}$ as $\e_{1} \otimes \cdots \otimes \e_{n}$.
This means that we can think of $Y(\vec{\Lambda})$ being generated by elements $\e$ for $\e \in E(\vec{\Lambda})$ satisfying the relations
$$
p_{\alpha}\e 
= \delta_{s(\e), \alpha}\e, 
\, \, 
\e p_{\alpha} 
= \delta_{t(\e), \alpha}\e, 
\, \,\text{ and } 
\langle \e'|\e \rangle_{\cC_{\Lambda}} 
= \delta_{\e, \e'} p_{t(\e')}.
$$

\begin{defn}\label{defn:FreeGraphAlgebra}
Suppose $\e \in E(\Lambda)$.

\be
\item
If $\e$ is a loop, then we can think of $\e \in E(\vec{\Lambda})$, and we define $T_{\e} = L_{+}(\e) + L_{+}(\e)^{*}$.
\item
If $\e$ is not a loop, let $\e'$ and $\e''$ be the two edges associated to $\e$ in Definition \ref{defn:Directed}.
Set $a = \sqrt[4]{\frac{\mu(s(\e'))}{\mu(t(\e'))}}$.  
We define $T_{\e}$ as follows:
$$
T_{\e} = a L_{+}(\e') + a^{-1} L_{+}(\e'')^{*} + a^{-1} L_{+}(\e'') + a L_{+}(\e')^{*}.
$$
\ee
Notice that the elements $T_{\e}$ are self-adjoint.
\end{defn}

\begin{defn}
The \underline{free graph algebra $\cS(\Lambda,\mu)$} is the $C^{*}$-subalgebra of $\cT_{\vec{\Lambda}}$ generated by $\cC_{\Lambda}$ together with $\set{T_{\e}}{\e \in E(\Lambda)}$.
\end{defn}

We now determine the structure of the free graph algebra.

\begin{thm}\label{thm:freegraphToeplitz}
Define a map $E: \cS(\Lambda,\mu) \rightarrow \cC_\Lambda$ by $E(x) = \sum_{\alpha \in V(\Lambda)} \langle p_{\alpha}| xp_{\alpha} \rangle_{\cC_{\Lambda}}$.
Notice that this sum always converges in norm.

For a fixed $\e \in E(\Lambda)$, let $\cS_{\e}$ denote the $C^{*}$-subalgebra of $\cS(\Lambda,\mu)$ generated by $\cC_{\Lambda}$ and $T_{\e}$.
\be
\item
$E$ is faithful on each $\cS_{\e}$.

\item
The $\cS_{\e}$ are free with amalgamation over $\cC_{\Lambda}$, implying $E$ is faithful on $\cS(\Lambda,\mu)$, and
$$
\cS(\Lambda,\mu) = \underset{\cC_{\Lambda}}{\Asterisk} (\cS_{\e}, E)_{\e \in E(\Lambda)}
$$
where the free product is reduced.
\ee
\end{thm}

\begin{proof}
\mbox{}
\be

\item
Assume first that $\e$ is a loop at $\alpha$.  Let $\Lambda_{\e}$ be the subgraph of $\Lambda$ consisting of the vertex $\alpha$ and the loop $\e$.  We need only show that the mapping $E_{\e}: \cS(\Lambda_{\e},\mu_\epsilon) \rightarrow C^{*}(p_{\alpha}) \cong \C$ given by  $E_{\e}(x) = \langle p_{\alpha}| xp_{\alpha} \rangle_{p_{\alpha}\cC_{\Lambda}p_{\alpha}}$ is faithful.   This immediately follows from Voiculescu's free Gaussian functor construction since $\cS(\Lambda_{\e}) = C^{*}(p_{\alpha}, L_{+}(\e) + L_{-}(\e))$.

Now assume that $\e$ has two distinct edges $\alpha$ and $\beta$ as endpoints, and let $s(\e') = \alpha$ and $t(\e') = \beta$.  Let ${\Lambda}_{\e}$ be the subgraph of ${\Lambda}$ whose vertices are $\alpha$ and $\beta$ and edge $\epsilon$.
Let $\mu_\epsilon=\mu|_{V(\Lambda_\epsilon)}$.

To show that $E$ is faithful on $\cS_{\e}$, it is enough to show that the mapping $E_{\e}: \cS(\Lambda_{\e},\mu_\epsilon) \rightarrow C^{*}(p_{\alpha}, p_{\beta})$ given by
$$
E_{\e}(x) = \langle p_{\alpha}| x p_{\alpha}\rangle + \langle p_{\beta}| xp_{\beta}\rangle
$$
is faithful.
We define bounded (non-adjointable) operators $R_{\e'}$ and $R_{\e''}$ on $\cF(Y(\vec{\Lambda}_\epsilon))$, where $\epsilon',\epsilon''$ are as in Definition \ref{defn:Directed}.
The operators $R_{\e'}$ and $R_{\e''}$ are given on tensors by the formulas
\begin{align*}
R_{\e'}(\e_{1}\otimes \cdots \otimes \e_{n})
&=
a^{-1}\e_{1} \otimes \cdots \otimes \e_{n} \otimes \e' + a\e_{1} \otimes \cdots \otimes \e_{n-1}\langle \e_{n}^{\op}|\e' \rangle
\\
R_{\e''}(\e_{1}\otimes \cdots \otimes \e_{n})
&=
a\e_{1} \otimes \cdots \otimes \e_{n} \otimes \e'' + a^{-1}\e_{1} \otimes \cdots \otimes \e_{n-1}\langle \e_{n}^{\op}|\e'' \rangle.
\end{align*}
A standard induction argument shows that $p_{\alpha} + p_{\beta}$ is cyclic for the algebra $\cR$ generated by $R_{\e'}$, $R_{\e''}$, and $\cC_{\Lambda_{\e}}$ acting on the right, and it is a straightforward argument to show that $\cS(\Lambda_\epsilon,\mu_\epsilon)$ and $\cR$ commute.
Therefore, if there is an $x \in \cS(\Lambda_{\e},\mu_\epsilon)$ so that $xp_{\alpha}$ = 0 and $xp_{\beta} = 0$, then $x(p_{\alpha} + p_{\beta}) = 0$ and since $x$ commutes with $\cR$, $x$ must be 0.
Therefore if $y \geq 0$, $y \in \cS(\Lambda_{\e},\mu_\epsilon)$, and $E(y) = 0$, then
$$
0
=
\langle p_{\alpha}|yp_{\alpha} \rangle_{\cC_{\Lambda_{\e}}} + \langle p_{\beta}|yp_{\beta} \rangle_{\cC_{\Lambda_{\e}}}
=
\langle y^{1/2}p_{\alpha}|y^{1/2}p_{\alpha} \rangle_{\cC_{\Lambda_{\e}}} + \langle y^{1/2}p_{\beta}|y^{1/2}p_{\beta} \rangle_{\cC_{\Lambda_{\e}}}
$$
implying $y = 0$.
This implies $E$ is faithful on $\cS_{\e}$.

\item
Note that the map $E$ is clearly well-defined on all of $\cT(Y(\vec{\Lambda}))$, and the expectationless elements in $\cT(Y(\vec{\Lambda}))$ are densely spanned by terms of the form
    $$
    \sum_{i=1}^{n} a_{i} L_{+}(\e_{1})\cdots L_{+}(\e_{\ell(i)})L_{+}(\e'_{1})^{*}\cdots L_{+}(\e'_{r(i)})^{*}
    $$
for $\ell(i)$ and $r(i)$ not both zero.  Therefore, the expectationless elements in $\cS(\Lambda,\mu)$ are densely spanned by certain sums of this form as well.  Suppose that $x_{1}, ..., x_{n}$ are terms of this form, $x_{i} \in \cS_{\e_{i}}$ with $\e_{i} \neq \e_{i+1}$, and $E(x_{i}) = 0$.  The identity
$$
L_{+}(\e_{i}^{t})^{*}L_{+}(\e_{i+1}^{t}) = 0 \text{ for } s, t \in \{', ''\}
$$
shows that no elements of $\cC$ appear as a summand of the expression for $x_{1}\cdots x_{n}$.  This gives the desired result.
\qedhere
\ee
\end{proof}

Let $\tau_{\mu}$ be the tracial weight on $\cC_{\Lambda}$ which is defined by $\tau_{\mu}(p_{\alpha}) = \mu(\alpha)$, and define $\Tr_{\mu} = \tau_{\mu}\circ E$.
We can now give the explicit structure of the algebras $\cS_\epsilon$.

\begin{thm}\label{thm:FreeGraph}
For each algebra $\cS_{\e}$, $\Tr_{\mu}$ is a lower-semicontinuous tracial weight on $\cS_{\e}$ which is finite if and only if $\sum_{v \in V(\Lambda)} \mu(v)$ is finite.  Furthermore,
\be
\item
If $\e$ is a loop at $\alpha$, then
$$
\cS_{\e} \cong \overset{p_{\alpha}}{\underset{\mu(\alpha)}{C[0, 1]}} \oplus C_{0}(V(\Lambda) \setminus \{\alpha\})
$$
where if $f \in p_{\alpha}\cS_{\e}p_{\alpha}$, $\Tr_\mu(f) = \mu(\alpha)\int_{0}^{1}f(\lambda)\, d\lambda$ with $d\lambda$ Lebesgue measure.

\item
If $\e$ has $\alpha$ and $\beta$ as endpoints with $\alpha \neq \beta$ and $\mu(\alpha) > \mu(\beta)$, then
$$
\cS_{\e} \cong \underset{2\mu(\beta)}{\overset{p_{\beta} + q_{\alpha}}{\big(M_{2}(\C) \otimes C[0, 1] \big)}} \oplus \underset{\mu(\alpha) - \mu(\beta)}{\overset{r_{\alpha}}{\C}} \oplus C_{0}(V(\Lambda) \setminus \{\alpha, \beta\}).
$$
Here, $p_{\alpha} = q_{\alpha} + r_{\alpha}$.
On $(p_{\alpha} + q_{\beta})\cS_{\e}(p_{\alpha} + q_{\beta})$, we have $\Tr_\mu = \mu(\beta) \Tr_{M_{2}(\C)} \otimes \int (\cdot) \,d\lambda$ with $d\lambda$ Lebesgue measure.

\item
Set $\cD$ to be the algebra
$$
\cD = \set{f: [0, 1] \rightarrow M_{2}(\C)}{f \text{ is continuous  and } f(0) \text{ is diagonal}}.
$$
If $\e$ has $\alpha$ and $\beta$ as endpoints with $\alpha \neq \beta$ and $\mu(\alpha) = \mu(\beta)$, then
$$
\cS_{\e} \cong \underset{2\mu(\alpha)}{\overset{p_{\alpha} + p_{\beta}}{\cD}} \oplus C_{0}(V(\Lambda) \setminus \{\alpha, \beta\}),
$$
where $p_{\alpha}$ and $p_{\beta}$ are the canonical matrix units $e_{11}$ and $e_{22}$ respectively.
The trace on $\cD$ is $\mu(\alpha)\Tr_{M_{2}(\C)} \otimes \int (\cdot)\,d\lambda$ with $d\lambda$ Lebesgue measure.

\ee
\end{thm}

\begin{proof}
For (1), note that $\cS_{\e}$ is commutative, and $p_{\alpha}\cS_{\e}p_{\alpha}$ is generated by $p_{\alpha}$ and $L_{+}(\e) + L_{+}(\e)^{*}$.  
Standard arguments show that $L_{+}(\e) + L_{+}(\e)^{*}$ has Voiculescu's semicircular distribution on $[-2, 2]$ in $p_{\alpha}\cS_{\e}p_{\alpha}$ with respect to $\Tr$.

For (2) and (3), we set $\alpha = s(\e')$ and $a = \sqrt[4]{\frac{\mu(\alpha)}{\mu(\beta)}}$ and assume that $a \geq 1$.  
We note that $(p_{\alpha} + p_{\beta})\cS_{\e}(p_{\alpha} + p_{\beta})$ is generated as a $C^{*}$-algebra by
$$
T_{\e'} = a L_{+}(\e') + a^{-1}L_{+}(\e'')^{*} 
\text{ and } 
T_{\e''} = a^{-1}L_{+}(\e'') + aL_{+}(\e').
$$
Note that $T_{\e'}^{*} = T_{\e''}$, $T_{\e'}^{*}T_{\e'} \in p_{\beta}\cS_{\e}p_{\beta}$, and $T_{\e''}^{*}T_{\e''} \in p_{\alpha}\cS_{\e}p_{\alpha}$.  
We will compute the laws of $T_{\e'}^{*}T_{\e'}$ in $p_{\beta}\cS_{\e}p_{\beta}$, and $T_{\e''}^{*}T_{\e''}$ in $p_{\alpha}\cS_{\e}p_{\alpha}$ with respect to $\Tr_{\mu}$.

We let $P_{n}(\e')$ be the coefficient of $p_{\beta}$ in $(T_{\e'}^{*}T_{\e'})^{n}$ and $P_{n}(\e'')$ the coefficient of $p_{\alpha}$ in $(T_{\e''}^{*}T_{\e''})^{n}$. 
We say that $P_{0}(\e') = 1 = P_{0}(\e'')$.  We claim that for $n \geq 1$
$$
P_{n}(\e') = a^{2}\sum_{k=0}^{n-1} P_{k}(\e'')P_{n-k-1}(\e')
\, \, \text{ and } \, \, 
P_{n}(\e'') = a^{-2}\sum_{k=0}^{n-1} P_{k}(\e')P_{n-k-1}(\e'').
$$
Indeed, the $p_{\beta}$ term in $(T_{\e'}^{*}T_{\e'})^{n}$ comes from terms of the form
$$
a^{2} L_{+}(\e')^{*}(T_{\e''}^{*}T_{\e''})^{k}L_{+}(\e')(T_{\e'}^{*}T_{\e'})^{n-k-1},
$$
where in each term, we take the term
$$
a^{2}L_{+}(\e')^{*}(P_{k}(\e'')p_{\alpha})L_{+}(\e')P_{n-k-1}p_{\beta} 
= 
a^{2}P_{k}(\e'')P_{n-k-1}(\e')p_{\beta}
$$
to avoid double-counting.
This establishes the first identity and the second identity is established from similar arguments.

We now write down the moment generating functions
$$
M_{\e'}(z) = \sum_{n=0}^{\infty} P_{n}(\e')z^{n} \text{ and } M_{\e''}(z) = \sum_{n=0}^{\infty} P_{n}(\e'')z^{n}
$$
for $T_{\e'}^{*}T_{\e'}$ and $T_{\e''}^{*}T_{\e''}$ respectively.  From the above recursion relations, we have the following system of equations:
$$
M_{\e'}(z) = a^{2}zM_{\e'}(z)M_{\e''}(z) + 1 \, \, \text{ and }\, \, M_{\e'}(z) = a^{-2}zM_{\e'}(z)M_{\e''}(z) + 1.
$$
Solving this gives
\begin{align*}
&M_{\e'}(z) = \frac{a^{2} - (a^{4} - 1)z + \sqrt{(z(a^{4} - 1) - a^{2})^{2} - 4a^{2}z}}{2z} \text{ and }\\
&M_{\e''}(z) = \frac{a^{-2} - (a^{-4} - 1)z + \sqrt{(z(a^{-4} - 1) - a^{-2})^{2} - 4a^{-2}z}}{2z}.
\end{align*}
The laws with respect to $\Tr_\mu$ can now be recovered from the Cauchy transforms $G_{f}(z) =\mu(t(f))z^{-1} M_{f}(z^{-1})$ for $f \in \{\e', \e''\}$.  The Cauchy transforms are
\begin{align*}
&G_{\e'}(z) = \mu(\beta)\frac{a^{2}z - (a^{4} - 1) + \sqrt{((a^{4} - 1) - a^{2}z)^{2} - 4a^{2}z}}{2z} \text{ and }\\
&G_{\e''}(z) = \mu(\alpha)\frac{a^{-2}z - (a^{-4} - 1) + \sqrt{((a^{-4} - 1) - a^{-2}z)^{2} - 4a^{-2}z}}{2z}
\end{align*}
where the branch on the square root is chosen so that
$$
\lim_{\Im(z) \rightarrow + \infty} G_{\e'}(z) = 0 = \lim_{\Im(z) \rightarrow + \infty} G_{\e''}(z).
$$
The polynomials in the square root both have roots at $a^{2} + a^{-2} \pm 2$.  Therefore they are scalar multiples of each other, and we see that the second polynomial differs from the first by a factor of $a^{-8}$.

To get the law, we compute $\lim_{y \to 0^+} -\frac{1}{\pi}\Im(G(x + iy))$ in the sense of distributions.  Using $a^{4} \geq 1$ and $a^{-4} \leq 1$, this says that the laws $d\mu_{\e'}$ and $d\mu_{\e''}$ of $T_{\e'}^{*}T_{\e'}$ and $T_{\e''}^{*}T_{\e''}$ respectively are:
\begin{align*}
d\mu_{\e'} &=
\mu(\beta)\frac{\sqrt{4a^{2}x - (a^{4} - 1 - a^{2}x)^{2}}}{2\pi x}{\bf 1}_{[a^{2} + a^{-2} - 2, a^{2} + a^{-2} + 2]}\, dx
\text{ and }\\
d\mu_{\e''} &=
\mu(\alpha)(1 - a^{-4})\delta_{0} + \mu(\alpha)a^{-4}\frac{\sqrt{4a^{2}x - (a^{4} - 1 - a^{2}x)^{2}}}{2\pi x} {\bf 1}_{[a^{2} + a^{-2} - 2, a^{2} + a^{-2} + 2]}\, dx.
\end{align*}
Since $a^{-4} = \frac{\mu(\beta)}{\mu(\alpha)}$, this implies that $\Tr_{\mu}((T_{\e'}^{*}T_{\e'})^{n}) = \Tr_{\mu}((T_{\e'}T_{\e'}^{*})^{n})$ for all $n \geq 0$.  Therefore $\Tr_{\mu}$ is tracial on $\cS_{\e}$.

We see from this analysis that both distributions are free-Poisson.  If $a > 1$, then $\mu(\alpha) > \mu(\beta)$, and $a^{-2} + a^{2} > 2$, implying $T_{\e'}^{*}T_{\e'}$ is invertible in $p_{\beta}\cS_{\e}p_{\beta}$ and hence the polar part of $T_{\e'}$ lies in $\cS_{\e}$.  Since the law of $T_{\e'}^{*}T_{\e'}$ in $p_{\beta}\cS_{\e}p_{\beta}$ is absolutely continuous with respect to Lebesgue measure, we have established the isomorphism in (2).

To establish the isomorphism in (3), we have $a = 1$ and hence $\mu(\alpha) = \mu(\beta)$.  Note that the laws of $T_{\e'}^{*}T_{\e'}$ and $T_{\e''}^{*}T_{\e''}$ are both absolutely continuous with respect to Lebesgue measure and are supported on $[0, 4]$.  This gives the isomorphism in (3).
\end{proof}

Since $E$ preserves $\Tr_\mu$ on each $\cS_{\e}$ by definition, we see that $\Tr_\mu$ is a faithful semi-finite trace on $\cS(\Lambda,\mu)$.  We obtain a Hilbert space representation of $\cS(\Lambda, \mu)$ by the induced representation on $\cH = \cF(Y(\vec{\Lambda})) \otimes_{\cC_{\Lambda}} \ell^{2}(V(\Gamma), \mu)$.  
The action of $\cS(\Lambda, \mu)$ on $\cH$ is simply the GNS representation of $\cS(\Lambda, \mu)$ with respect to $\Tr_\mu$.  
(This argument is similar to \cite[Lemma 3.3]{GJSCStar}.)
The von Neumann algebra generated by $\cS(\Lambda, \mu)$ in this representation is $\cM(\Lambda, \mu)$ from \cite{MR3110503}.  
We have the following useful lemma about the structure of $\cM(\Lambda, \mu)$.

\begin{lem}[\cite{MR3110503}]\label{lem:vNGraph}
Write $\alpha \sim \beta$ if $\alpha$ and $\beta$ are the endpoints of at least one edge in $\Lambda$, and $n(\alpha, \beta)$ the number of edges having $\alpha$ and $\beta$ as endpoints.  
Set $a_\mu(\alpha) = \sum_{\beta \sim \alpha} n(\alpha, \beta)\mu(\beta)$ and
$$
A(\Lambda, \mu) = \set{\alpha \in V(\Lambda)}{\mu(\alpha) > a_{\mu}(\alpha)}.
$$
Then
$$
\cM(\Lambda, \mu) \cong \cN \oplus \bigoplus_{\alpha \in A(\Lambda, \mu)} \overset{r_{\alpha}}{\underset{\mu(\alpha) - a_{\mu}(\alpha)}{\C}}.
$$
The algebra $\cN$ is $L(\F_{t})$ for some $t$ if $\sum_{\alpha \not\in A(\Lambda, \mu)} \mu(\alpha) + \sum_{\alpha \in A(\Lambda, \mu)}a_{\mu}(\alpha)$ is finite or $L(\F_{s}) \otimes B(H)$ otherwise.  
The projection $r_{\alpha}$ is a subprojection of  $p_{\alpha}$.
\end{lem}

The above lemma was proven in \cite{MR3110503} for finite weighted graphs.  The standard embedding arguments used in Section 4 of that paper get the infinite case as well.  We also have the following lemma which is surely well known, but we will provide a proof.

\begin{lem}\label{lem:repcompact}
Suppose $x \in \cK(\cF(Y(\vec{\Lambda})))$ and let $\pi$ be the induced representation of $\cL(\cF(Y(\vec{\Lambda})))$ on $\cH = \cF(Y(\vec{\Lambda})) \otimes_{\cC_{\Lambda}} \ell^{2}(V(\Gamma), \mu)$.  Then $\pi(x) \in \cK(\cH)$.

\end{lem}

\begin{proof}
We need only show that $\pi(| \e_{1} \otimes \cdots \otimes \e_{n} \rangle \langle \e'_{1} \otimes \cdots \otimes \e'_{m} |) \in \cK(\cH)$.  We note that on $\cF(Y(\vec{\Lambda}))$, the range of $| \e_{1} \otimes \cdots \otimes\e_{n} \rangle \langle \e'_{1} \otimes \cdots \otimes \e'_{m} |$  is contained in 
$$
(\e_{1} \otimes \cdots \otimes \e_{n})\cC_{\Lambda} = \spann\{(\e_{1} \otimes \cdots \otimes \e_{n})p_{t(\e_{n})}\}.
$$
It follows from this that $\pi(| \e_{1} \otimes \cdots\otimes \e_{n} \rangle \langle \e'_{1} \otimes \cdots \otimes \e'_{m} |)$ is rank at most 1 on $\cH$ and we are finished.
\end{proof}

We now show when $\cS(\Lambda, \mu)$ nontrivially intersects $\cK(\cF(Y(\vec{\Lambda})))$.

\begin{thm}
If $A(\Lambda, \mu)$ is empty, then $\cS(\Lambda, \mu)\cap \cK(\cF(Y(\vec{\Lambda}))) = \{0\}$.  If $\Lambda$ is locally finite and $A(\Lambda, \mu)$ is nonempty, then $\cS(\Lambda, \mu)$ nontrivially intersects $ \cK(\cF(Y(\vec{\Lambda})))$.  
\end{thm}

\begin{proof}
If $A(\Lambda, \mu)$ is empty, Lemma \ref{lem:vNGraph} implies that $\cS(\Lambda, \mu)$ completes to either a II$_{1}$ factor or a II$_{\I}$ factor on $\cH$.  Neither of these contain $\cK(\cH)$ so it follows that $\cS(\Lambda, \mu) \cap \cK(\cF(Y(\vec{\Lambda}))) = \{0\}$ from Lemma \ref{lem:repcompact}.

If $\alpha \in A(\Lambda, \mu)$ and $\Lambda$ is locally finite, let $\e_{1}, \dots \e_{n}$ be the edges satisfying $s(\e_{i}) = \alpha$.  
By Theorem \ref{thm:FreeGraph} and its proof, the support projections $p_{i}$ of $T_{\e_{i}}T_{\e_{i}}^{*}$ are in $\cS(\Lambda, \mu)$.  
Since $a(\alpha) = \sum_{i=1}^{n} \Tr_\mu(p_{i}) < \mu(\alpha)$, Voiculescu's $R$-transform can be used to show that $\sum_{i=1}^{n}p_{i}$ has an atom of size $\mu(\alpha) - a(\alpha)$ at 0 and the rest of the law is supported inside an interval missing the origin.   
From Lemma \ref{lem:vNGraph}, we see that the projection corresponding to the atom at zero is $r_{\alpha}$ and is hence $r_{\alpha} \in \cS(\Gamma, \mu)$.  
The projection $r_{\alpha}$ is rank 1 on $\cH$ so it is compact on $\cF(Y(\vec{\Lambda}))$.
\end{proof}

From this theorem, we get the following corollary:
\begin{cor}
If $A(\Lambda, \mu)$ is empty, then the canonical surjection $\cT(Y(\Lambda)) \rightarrow \cO_{\vec{\Lambda}}$ is injective on $\cS(\Lambda, \mu)$; hence we can view  $\cS(\Lambda, \mu)$ as  a subalgebra of $\cO_{\vec{\Lambda}}$.
\end{cor}

It would be nice to determine more about the general structure of the algebras $\cS(\Lambda, \mu)$ for different weightings $\mu$.  In particular, it would be nice to know which weightings guarantee that $\cS(\Lambda, \mu)$ is simple with unique trace, and more ambitiously, when $\cS(\Lambda, \mu)$ is isomorphic to  $\cS(\Lambda', \mu')$.  Using Theorem \ref{thm:Germain}, we can very quickly compute the $K$-theory of $\cS(\Lambda, \mu)$.

\begin{thm}\label{thm:KTheory}
The inclusion $\cC_{\Lambda} \hookrightarrow \cS(\Lambda, \mu)$ is a $KK$-equivalence for any $\Lambda$.  Consequently,
$$
K_{0}(\cS(\Lambda, \mu)) = \Z\set{[p_{\alpha}]}{\alpha \in V(\Lambda)} \text{ and } K_{1}(\cS(\Lambda, \mu) )= (0).
$$  
\end{thm}

\begin{proof}
We note that on $\cT(Y(\vec{\Lambda}))$, $\cS(\vec{\Lambda})$ is generated by elements $L_{+}(x) + L_{-}(x)$ for $x$ in the closure of the real subspace
$$
\spann_{\R}\set{a\e' + a^{-1}\e''}{\e \in E(\Lambda) \text{ and } a = \sqrt[4]{\frac{\mu(s(\e'))}{\mu(t(\e'))}}}
$$
where we assume $\e' = \e = \e''$ if $\e$ is a loop.  If $\Lambda$ is finite, Theorem \ref{thm:Germain} gives the $K$-theory since $\cC_{\Lambda}$ is unital.  If $\Lambda$ is infinite, then write $\Lambda$ as an increasing union of finite subgraphs $\Lambda_{n}$.  Theorem \ref{thm:Germain} implies that the canonical inclusion $\cC_{\Lambda_{n}} \hookrightarrow \cS(\Lambda_{n}, \mu)$ is a $KK$-equivalence.  Taking inductive limits gives the result.
\end{proof}

\subsection{Compressing $\SP{}$ to the free graph algebra $\cS(\Gamma)$}\label{sec:PAGraph}

Let $\cY_{\R} = \set{y \in \cY}{y = y^{\dagger}}$, and note that $\cY_{\R} = \overline{\cup_{n} P_{n}\cX_{\R}P_{n}}^{\|\cdot\|}$.

\begin{defn}
We define $\cS(\Gamma)$ to be the $C^{*}$-subalgebra of $\cT(\Gamma)$ generated by $\cC$ and $\set{L_{+}(y) + L_{-}(y)}{y \in \cY_{\R}}$.
\end{defn}

The following proposition has the same proof as Propositions \ref{prop:cyclic} and \ref{prop:CompressedFock}.

\begin{prop}\label{prop:cyclic3}
\mbox{}
\be

\item $\cS(\Gamma)$ is generated by $\cC$ and elements
$\D\sum_{\ell\, + \,\err \, = n}\begin{tikzpicture}[baseline = .2cm]
	\draw (-.8, 0)--(.8, 0);
	\draw (-.2, 0)--(-.2, .8);
	\draw (.2, 0)--(.2, .8);
	\node at (-.4, .6) {{\scriptsize{$\ell$}}};
	\node at (.4, .6) {{\scriptsize{$\err$}}};
	\nbox{unshaded}{(0,0)}{.4}{0}{0}{$y$}
	\draw[fill=red] (0,.4) circle (.05cm);
\end{tikzpicture}
$
for $y \in \cY_{n}$.

\item $\cS(\Gamma)$ is generated by $\cC$ and $\set{L_{+}(y) + L_{-}(y)}{y \in \cY}$.

\item $\cS(\Gamma) = \overline{\cC\SP{}\cC}^{\|\cdot\|}$.

\ee

\end{prop}

With the aid of this proposition, we can picture $\cS(\Gamma)$ as being generated by elements $y \in \cY_{n}$ such that if $x \in \cY_{m}$ and $y \in \cY_{n}$, then
$$
x\cdot y = \sum_{k=0}^{\min\{n, m\}}\, \begin{tikzpicture} [baseline = .1cm]
\draw (-.8, 0)--(2, 0);
\draw (-.2, 0)--(-.2, .8);
\draw (.2,.4) arc(180:0:.4cm);
\draw (1.4, 0)--(1.4, .8);
\node at (.6, .6) {\scriptsize{$k$}};
\nbox{unshaded}{(0,0)}{.4}{0}{0}{$x$}
\nbox{unshaded}{(1.2,0)}{.4}{0}{0}{$y$}
\end{tikzpicture}\,.
$$
Note that
$$
L_{+}(g_{\e'}) + L_{-}(g_{\e'}) = \frac{S_{\e'}}{\sqrt{\mu(t(\e'))}} + \frac{S^{*}_{\e''}}{\sqrt{\mu(s(\e'))}}
$$
with $\mu$ the quantum dimension weighting on $\Gamma$ induced from $\cP_{\bullet}$, which satisfies the Frobenius-Perron condition.  
From Section \ref{sec:FreeGraphAlgebraOfUnoriented} we see that $\cS(\Gamma) = \cS(\Gamma, \mu)$, so the name $\cS(\Gamma)$ makes sense.   
As the set $A(\Gamma, \mu)$ is empty for a weighting which satisfies the Frobenius-Perron condition, we immediately deduce that $\cS(\Gamma)$ can be seen as a subalgebra of $\cO(\Gamma)$.  This shows the following:

\begin{cor}
$\cS(\Gamma)$ is subnuclear, and thus exact.
\end{cor}

We could have also deduced exactness of $\cS(\Gamma)$ by \cite{MR2039095}.  From the previous section, we deduce that 
$$
K_{0}(\cS(\Gamma)) = \Z\set{[p_{\alpha}]}{\alpha \in V(\Gamma)} \text{ and } K_{1}(\cS(\Gamma)) = (0).
$$
In \cite[Subsection 5.1]{GJSCStar} we prove that $\cS(\Gamma)$ is simple.  
For now, we end with the following lemma which ties together the structure of $\cB$, $\TP{}$, $\SP{}$ and $\OP{}$  with $\cC$, $\cT(\Gamma)$, $\cS(\Gamma)$, and $\cO(\Gamma)$.

\begin{lem}\label{lem:CompactTensor1} 
Let $\cK$ be the algebra of compact operators on a separable infinite-dimensional Hilbert space.  

\be
\item 
$\cB \cong \cC \otimes \cK$, and

\item 
for $\cA\in \{\cT,\cO,\cS\}$, we have $\cA(\cP_\bullet) \cong \cA(\Gamma) \otimes \cK$.

\ee
\end{lem}

\begin{proof}
(1) is immediate since 
$\cC \cong \bigoplus_{v \in V(\Gamma)} \C$
and
$\cB \cong \bigoplus_{v \in V(\Gamma)}  \cK$
with $\cC$ canonically a hereditary subalgebra  of $\cB$.
To prove (2)-(4), just use the matrix units from (1), since $\cA(\Gamma)\cong \overline{\cC\cA(\cP_{\bullet})\cC}^{\|\cdot\|}$.
\end{proof}

These stable isomorphisms immediately imply the following result.

\begin{cor}\label{cor:KTheory}
The canonical inclusion $i: \cB \hookrightarrow \SP{}$ is a $KK$-equivalence.   As a consequence,
$$
K_{0}(\SP{}) = \Z\set{[p_{\alpha}]}{\alpha \in V(\Gamma)} \text{ and }\, K_{1}(\SP{}) = (0).
$$
The same holds for the canonical inclusion $\SP{} \hookrightarrow \TP{}$. 
\end{cor}

\subsection{A further canonical compression}\label{sec:FurtherCompression}

A planar algebra inclusion $\cQ_\bullet \hookrightarrow \cP_{\bullet}$ does not induce inclusions of (Toeplitz-)Cuntz-Krieger or free graph algberas.
More precisely, we do \underline{not} get a map $\cY(\cQ_\bullet)$ to $\cY(\cP_\bullet)$ since $\depth(\cQ_\bullet) \geq \depth(\cP_\bullet)$ (e.g., see \cite[Corollary 3.12]{1208.1564}). 
In particular, we see this problem when $\delta\geq 2$, $\cQ_\bullet=\TL_\bullet$, and $\cP_\bullet$ is finite depth.

However, we do have functorality of the assignment $\cP_\bullet$ to $\SP{0}$ and $\OP{0}$.
We realize this by a further compression by $p_\star$.
To this end, we will need to represent our compression faithfully on a Hilbert space.

\begin{defn}\label{defn:GJSTrace}
Using the identification of Remark \ref{rem:GJS}, we can define a conditional expectation $E_\infty: \SP{} \rightarrow \cB$ given by the extension of $E_\infty(x) = \delta_{n, 0}x$ for $x \in \cP_{l, n, r}\subset \SP{}$.  
Note that $E_\infty$ is continuous since if $Q_{0}$ is the orthogonal projection from $\FP{}$ onto $\cB$, we have $E_\infty(x) = Q_{0}xQ_{0}$ as an element in $\cB$.  
Furthermore, using the isomorphisms $\cB\cong \cC_\Gamma \otimes \cK$ and $\SP{}\cong \cS(\Gamma)\otimes \cK$ from Lemma \ref{lem:CompactTensor1}, $E_\infty=E\otimes \id_\cK$ for the $E$ given in Theorem \ref{thm:freegraphToeplitz}.
Thus $E_\infty$ is faithful.

The map $\Tr \circ E_\infty$ on $\SP{}$ becomes a faithful lower-semicontinuous tracial weight on $\SP{}$.  
Diagrammatically, we have
$$
\Tr \circ E_\infty(x) =  \delta_{n, 0} \cdot \delta_{l, r}\,
\begin{tikzpicture}[baseline = -.2cm]
	\draw (-.4, 0)--(.4, 0) arc(90:-90:.4cm)--(-.4, -.8) arc(270:90:.4cm);
	\nbox{unshaded}{(0,0)}{.4}{0}{0}{$x$}
	\node at (-.6, .2) {{\scriptsize{$l$}}};
	\node at (.6, .2) {{\scriptsize{$r$}}};
\end{tikzpicture}\,.
$$
\end{defn}

\begin{thm}\label{thm:OSfunctor}
The assignment $\cP_{\bullet}$ to $\SP{0}$ and $\OP{0}$ is functorial in the sense that an inclusion $\cQ_{\bullet} \hookrightarrow \cP_{\bullet}$ of factor planar algebras induces a canonical inclusion $\cA(\cQ_{\bullet}) \hookrightarrow \cA(\cP_{\bullet})$ for $\cA \in \{\cO, \cS\}$.
\end{thm}  

\begin{proof}
We note that the map $i_{\cA}$ in Theorem \ref{thm:functorial} induces a canonical inclusion $p_{\star}\cA(\cQ_{\bullet})p_{\star} \rightarrow p_{\star}\cA(\cP_{\bullet})p_{\star}$ for $\cA \in \{\cO, \cS\}$.  
If we show that $p_{\star}\cA(\cP_{\bullet})p_{\star}$ is naturally isomorphic to $\cA_{0}(\cP)$, we will be finished.

To this end, we note that $p_{\star}\FP{}p_{\star}$ is the Hilbert space $\FP{0}$. 
Let $\rho(p_{\star})$ denote the right action of $p_{\star}$ on $\FP{}$.  
We define a $C^{*}$-algebra homomorphism $\phi: p_{\star}\cL(\FP{})p_{\star} \rightarrow \cB(\FP{0})$ by $\phi(x) = x\rho(p_{\star})$.   
We observe that $\phi(p_{\star}\TP{}p_{\star}) = \TP{0}$ and $\phi(p_{\star}\SP{}p_{\star}) = \SP{0}$.
Diagrammatically, $\phi$ sends $x$ as an operator on $p_{\star}\FP{}$ to the operator $x$ on $\FP{0}$.  

The $C^{*}$-algebra $p_{\star}\SP{}p_{\star} = p_{\star}\cS(\Gamma)p_{\star}$ carries a faithful trace $\tr=\Tr\circ E_\infty (p_{\star}\, \cdot \, p_{\star})$ as in Definition \ref{defn:GJSTrace}.  
We see that the representation of $p_{\star}\SP{}p_{\star}$ on $\FP{0}$ is simply the GNS representation with resect to $\tr$.  
Therefore, $\phi$ is faithful on $p_{\star}\SP{}p_{\star}$.

We note that if $x \in p_{\star}\cK(\cP_{\bullet})$, then $\phi(x) \in \cK(\FP{0})$.  
Therefore, $\phi$ induces a surjection from $p_{\star}\OP{}p_{\star}$ onto $\OP{0}$.  
The algebra $p_{\star}\OP{}p_{\star}$ is simple from Theorem \ref{thm:CuntzGraph} and the simplicity of $\cO_{\vec{\Gamma}}$.  
Therefore the surjection is an isomorphism.
\end{proof}

\begin{rem}
The proof of Theorem \ref{thm:OSfunctor} shows that $\TP{0}$ is a reduced compression of $\TP{}$. 
Note that the map $\phi$ need not be injective on $\TP{}$.
Indeed, if we consider
$$
y_{\cP_\bullet}=\,
\begin{tikzpicture}[baseline=.1cm]
	\draw (-.2,.7)--(-.2,.1) arc (-180:0:.2cm)--(.2,.7);
	\nbox{}{(0,0)}{.3}{.1}{.1}{}
	\draw[fill=red] (0,.3) circle (.05cm);
\end{tikzpicture}
\,-
\begin{tikzpicture}[baseline=.1cm]
	\draw (-.2,.7)--(-.2,.1) arc (-180:0:.2cm)--(.2,.7);
	\node at (-.4,.5) {\scriptsize{$2$}};
	\node at (.4,.5) {\scriptsize{$2$}};
	\nbox{}{(0,0)}{.3}{.1}{.1}{}
	\draw[fill=red] (0,.3) circle (.05cm);
\end{tikzpicture}
\in \TP{}
$$
similar to Remarks \ref{rems:NotFunctor}, then $y_{\TL_\bullet}\neq 0$, but $\phi(y_{\TL_{\bullet}}) = 0$.  
This explains why the assignment $\cP_{\bullet}$ with $\TP{0}$ is not functorial.
\end{rem}

\begin{ex}\label{ex:ToeplitzAF}
Recall the AF structure of Cuntz core of $\OP{0}$ from (1) of Remarks \ref{rem:CuntzCore}.

Now consider $p_\star \TP{} p_\star=p_\star \cT(\Gamma)p_\star$, which acts faithfully on $p_\star \cF(\Gamma)$.
We show that the core of this algebra (the fixed points under the $\bbT$-action) has an AF structure related to that of the core of $\OP{0}$.
We see that the core of $p_\star \TP{} p_\star$ is generated by the subalgebras $C_{n}$ spanned by elements of the form
$$
\begin{tikzpicture}[baseline=0cm]
	\draw (.2,0)--(.2,.8);
	\draw (-.2,0)--(-.2,.8);
	\node at (-.4,.6) {\scriptsize{$k$}};
	\node at (.4,.6) {\scriptsize{$k$}};
	\nbox{unshaded}{(0,0)}{.4}{0}{0}{$x$}
	\draw[fill=red] (0,.4) circle (.05cm);
\end{tikzpicture}
$$
for $k\leq n$. 
The inclusion $C_n \hookrightarrow C_{n+1}$ is the identity.
Notice that for each minimal projection $p\in \cP_{2n}$, we can write
$$
\begin{tikzpicture}[baseline=0cm]
	\draw (.2,0)--(.2,.8);
	\draw (-.2,0)--(-.2,.8);
	\node at (-.4,.6) {\scriptsize{$n$}};
	\node at (.4,.6) {\scriptsize{$n$}};
	\nbox{unshaded}{(0,0)}{.4}{0}{0}{$p$}
	\draw[fill=red] (0,.4) circle (.05cm);
\end{tikzpicture}
=
\begin{tikzpicture}[baseline=.1cm]
	\draw (.4,0)--(.4,1.2);
	\draw (-.4,0)--(-.4,1.2);
	\draw (-.2,1.2) -- (-.2,.7) arc (-180:0:.2cm) -- (.2,1.2);
	\node at (-.6,1) {\scriptsize{$n$}};
	\node at (.6,1) {\scriptsize{$n$}};
	\nbox{}{(0,.2)}{.6}{.1}{.1}{}
	\nbox{unshaded}{(0,.05)}{.3}{.2}{.2}{$p$}
	\draw[fill=red] (0,.8) circle (.05cm);
\end{tikzpicture}
+
\left(
\underbrace{
\begin{tikzpicture}[baseline=0cm]
	\draw (.2,0)--(.2,.8);
	\draw (-.2,0)--(-.2,.8);
	\node at (-.4,.6) {\scriptsize{$n$}};
	\node at (.4,.6) {\scriptsize{$n$}};
	\nbox{unshaded}{(0,0)}{.4}{0}{0}{$p$}
	\draw[fill=red] (0,.4) circle (.05cm);
\end{tikzpicture}
-
\begin{tikzpicture}[baseline=.1cm]
	\draw (.4,0)--(.4,1.2);
	\draw (-.4,0)--(-.4,1.2);
	\draw (-.2,1.2) -- (-.2,.7) arc (-180:0:.2cm) -- (.2,1.2);
	\node at (-.6,1) {\scriptsize{$n$}};
	\node at (.6,1) {\scriptsize{$n$}};
	\nbox{}{(0,.2)}{.6}{.1}{.1}{}
	\nbox{unshaded}{(0,.05)}{.3}{.2}{.2}{$p$}
	\draw[fill=red] (0,.8) circle (.05cm);
\end{tikzpicture}
}_{r_p}
\right).
$$
We already understand how the first projection on the right hand side decomposes in $C_{n+1}$ from its decomposition in the  Cuntz core.
Hence we must analyze the remainder term on the right hand side.

\begin{claim}
$r_p$ is a projection of rank exactly 1.
\end{claim}
\begin{proof}[Proof of Claim]
The only elements not killed by $r_p$ are those of degree exactly $n$ with no strings on the left side.
Consider a vector $y\in p_\star\cF(\Gamma)$ of the form
$$
\begin{tikzpicture}[baseline=0cm]
	\draw (0,0)--(0,.8);
	\draw (0,0)--(1.4,0);
	\node at (-.2,.6) {\scriptsize{$n$}};
	\nbox{unshaded}{(0,0)}{.4}{0}{0}{$y$}
	\nbox{unshaded}{(.9,0)}{.3}{0}{0}{$p_\alpha$}
\end{tikzpicture}
$$
for some $\alpha \in V(\Gamma)$. Now we see that $r_p y$ is given by
$$
\begin{tikzpicture}[baseline=0cm]
	\draw (-.8,.4) arc (180:90:.2cm) -- (-.2,.6) arc (90:0:.2cm);
	\draw (-1.2,0)--(-1.2,.8);
	\draw (0,0)--(1.4,0);
	\node at (-.4,.8) {\scriptsize{$n$}};
	\node at (-1.4,.6) {\scriptsize{$n$}};
	\nbox{unshaded}{(-1,0)}{.4}{0}{0}{$p$}
	\nbox{unshaded}{(0,0)}{.4}{0}{0}{$y$}
	\nbox{unshaded}{(.9,0)}{.3}{0}{0}{$p_\alpha$}
\end{tikzpicture}
$$
Since $p,p_\alpha$ are both minimal in $\cP_\bullet$, they are equivalent if $r_p y\neq 0$, and there is only a one-dimensional space of morphisms between both simples.
\end{proof}

From the above analysis, we see that the Bratteli diagram of $(C_n)_{n\geq 0}$ is obtained by taking the Bratteli diagram for the Cuntz core of $\OP{0}$, i.e., $(\cP_{2n})_{n\geq 0}$ under the right inclusion, and adjoining an $A_\infty$ tail to each vertex.
We draw this diagram by first drawing the Bratteli diagram for the core of $\OP{0}$, drawing a dotted line to the left, and attaching the $A_\infty$ tails to the left of the dotted line. 
We give an example in Figure \ref{fig:Bratteli}.

\begin{figure}[!htb]
$$
\begin{tikzpicture}
	\draw (0,0) -- (1.5,-1.5);
	\draw (0,-1) -- (1,-2);
	\draw (.5,-.5) -- (0,-1);
	\draw (1,-1) -- (0,-2);
	\draw (1.5,-1.5) -- (1,-2);
	\draw (0,0) -- (-1,-1) -- (-2,-1.5) -- (-3,-2);
	\draw (.5,-.5) -- (-.5,-1) -- (-1.5,-1.5) -- (-2.5,-2);
	\draw (0,-1) -- (-1,-1.5) -- (-2,-2);
	\draw (1,-1) -- (-.5,-1.5) -- (-1.5,-2);
	\draw (.5,-1.5) -- (-1,-2);
	\draw (1.5,-1.5) -- (-.5,-2);
	\filldraw (0,0) circle (.05cm);
	\filldraw (.5,-.5) circle (.05cm);
	\filldraw[fill=white] (0,-1) circle (.05cm);
	\filldraw (1,-1) circle (.05cm);
	\filldraw[fill=white] (.5,-1.5) circle (.05cm);
	\filldraw (1.5,-1.5) circle (.05cm);	
	\filldraw[fill=white] (0,-2) circle (.05cm);
	\filldraw[fill=white] (1,-2) circle (.05cm);
	\draw[dashed] (-.25,.5) -- (-.25, -2.5);	
	\filldraw[fill=DarkGreen] (-.5,-.5) circle (.05cm);	
	\filldraw[fill=DarkGreen] (-.5, -1) circle (.05cm);
	\filldraw[fill=DarkGreen] (-1, -1) circle (.05cm);
	\filldraw[fill=DarkGreen] (-.5, -1.5) circle (.05cm);
	\filldraw[fill=DarkGreen] (-1, -1.5) circle (.05cm);
	\filldraw[fill=DarkGreen] (-1.5, -1.5) circle (.05cm);
	\filldraw[fill=DarkGreen] (-2, -1.5) circle (.05cm);
	\filldraw[fill=DarkGreen] (-.5, -2) circle (.05cm);
	\filldraw[fill=DarkGreen] (-1, -2) circle (.05cm);
	\filldraw[fill=DarkGreen] (-1.5, -2) circle (.05cm);
	\filldraw[fill=DarkGreen] (-2, -2) circle (.05cm);
	\filldraw[fill=DarkGreen] (-2.5, -2) circle (.05cm);
	\filldraw[fill=DarkGreen] (-3, -2) circle (.05cm);	
\end{tikzpicture}
$$
\caption{Bratteli diagram for the core of $p_\star \cT(\TL_\bullet)p_\star$ for the $A_4$ factor planar algebra}\label{fig:Bratteli}
\end{figure}

We can now describe the core of $\TP{0}$ by looking at the reduction of $C$ by the right action of $p_\star$ on $p_\star \cF(\Gamma)$.
We take the Bratteli diagram for the core of $p_\star \TP{} p_\star$, and we eliminate all $A_\infty$ tails to the left of the dotted line which do not emanate from a vertex corresponding to a projection equivalent to the empty diagram, i.e., a vertex underneath the trivial vertex $\star$.
One sees this by repeating the above argument with $\alpha = \star$, so our vector $y$ lies in $p_\star \cF(\Gamma) p_\star$. Thus $p$ must be equivalent to the empty diagram to find a $y$ with $r_p y \neq 0$.
We give an example in Figure \ref{fig:NewBratteli}.

\begin{figure}[!htb]
$$
\begin{tikzpicture}
	\draw (0,0) -- (1.5,-1.5);
	\draw (0,-1) -- (1,-2);
	\draw (.5,-.5) -- (0,-1);
	\draw (1,-1) -- (0,-2);
	\draw (1.5,-1.5) -- (1,-2);
	\draw (0,0) -- (-2,-2);
	\draw (0,-1) -- (-1,-2);
	\filldraw (0,0) circle (.05cm);
	\filldraw (.5,-.5) circle (.05cm);
	\filldraw[fill=white] (0,-1) circle (.05cm);
	\filldraw (1,-1) circle (.05cm);
	\filldraw[fill=white] (.5,-1.5) circle (.05cm);
	\filldraw (1.5,-1.5) circle (.05cm);	
	\filldraw[fill=white] (0,-2) circle (.05cm);
	\filldraw[fill=white] (1,-2) circle (.05cm);
	\draw[dashed] (-.25,.5) -- (-.25, -2.5);	
	\filldraw[fill=DarkGreen] (-.5,-.5) circle (.05cm);	
	\filldraw[fill=DarkGreen] (-1, -1) circle (.05cm);
	\filldraw[fill=DarkGreen] (-.5, -1.5) circle (.05cm);
	\filldraw[fill=DarkGreen] (-1.5, -1.5) circle (.05cm);
	\filldraw[fill=DarkGreen] (-1, -2) circle (.05cm);
	\filldraw[fill=DarkGreen] (-2, -2) circle (.05cm);	
\end{tikzpicture}
$$
\caption{Bratteli diagram for the core of $\cT_0(\TL_\bullet)$ for the $A_4$ factor planar algebra}\label{fig:NewBratteli}
\end{figure}
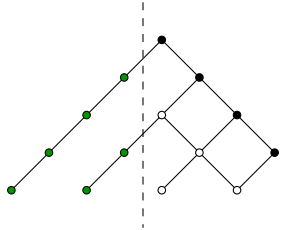
\end{ex}

\subsection{The shaded case}\label{sec:ShadedCase}

In this subsection, we briefly sketch what happens if $\cP_{\bullet}$ is shaded.
Let $\cP_{\bullet}$ be a subfactor planar algebra with principal graph $\Gamma_{+}$ and dual principal graph $\Gamma_{-}$.  We form two algebras $\cB_{+}$ and $\cB_{-}$ which are generated by 
$$
\bigoplus_{l, n, r = 0}^{\infty} \cP^{\pm}_{l, n, r}=\bigoplus_{l,n,r=0}^\infty \cP_{l+r+n,\pm}
$$
respectively.  Here, the boxes are drawn exactly as in Section \ref{sec:Operator} with the marked region on the bottom of the box.  
One can then form $\cB_{\pm}$ bimodules $\cX_{\pm}$, the Fock spaces $\cF(\cP_{\pm})$, the Toeplitz algebras $\cT(\cP_{\pm})$, the free semicircular algebras $\cT(\cP_{\pm})$ and the Cuntz-Pimsner algebras $\cO(\cP_{\pm})$.

Compressing $\cT(\cP_{\pm})$, $\cS(\cP_{\pm})$, and $\cO(\cP_{\pm})$ with the same techniques as the beginning of this section produces algebras isomorphic to $\cT(Y(\vec{\Gamma}_{\pm}))$, $\cS(\Gamma_{\pm}, \mu_{\pm})$, and $\cO_{\vec{\Gamma}_{\pm}}$ respectively.

In general, the $\pm$-algebras are not necessarily isomorphic.
If $\Gamma_+$ and $\Gamma_-$ have a different number of vertices, then the $K_0$-groups are not isomorphic.

We encourage the reader to see \cite{GJSCStar}, in particular Subsection 3.3, Remark 4.18, and Subsection 6.3, for more information on the shaded case.

\bibliographystyle{amsalpha}

{\footnotesize
\bibliography{../bibliography}
}
\end{document}